\documentclass[11pt,reqno]{amsart}

\usepackage{amsthm}
\usepackage{amssymb}
\usepackage{latexsym}
\usepackage{multicol}
\usepackage{verbatim,enumerate}
\usepackage{accents}
\usepackage{multicol}
\usepackage{tikz}
\usepackage{xcolor}
\usepackage[colorlinks=true, linkcolor=blue, citecolor=blue, urlcolor=blue]{hyperref}
\usepackage{hyperref}
\usepackage{amsmath, amscd}

\advance\textwidth by 1.3in \advance\oddsidemargin by -.6in \advance\evensidemargin by -.6in

\theoremstyle{definition}
\newtheorem*{lem}{Lemma}
\newtheorem*{prop}{Proposition}

\newtheorem*{thm}{Theorem}

\theoremstyle{definition}
\newtheorem*{defn}{Definition}
\newtheorem*{conj}{Conjecture}
\newtheorem*{example}{Example}

\newtheorem*{rem}{Remark}

\newtheorem*{theom}{Theorem}

\newcommand{\nc}{\newcommand}
\nc{\ch}{\text{ch}}

\newcounter{cnt}
\numberwithin{equation}{section}

\makeatletter
\def\section{\def\@secnumfont{\mdseries}\@startsection{section}{1}%
  \z@{.7\linespacing\@plus\linespacing}{.5\linespacing}%
  {\normalfont\scshape\centering}}
\def\subsection{\def\@secnumfont{\bfseries}\@startsection{subsection}{2}%
  {\parindent}{.5\linespacing\@plus.7\linespacing}{-.5em}%
  {\normalfont\bfseries}}
\makeatother
 
 \makeatletter
\def\mydggeometry{\makeatletter\dg@YGRID=1\dg@XGRID=20\unitlength=0.003pt\makeatother}
\makeatother \theoremstyle{remark}

\newcommand\al{\alpha}

\newcommand\Lg{\mathfrak{g}}

\newcommand{\bp}{{\bf p}}
\newcommand{\bs}{{\bf s}}

\newcommand{\bt}{{\bf t}}
\newcommand{\br}{{\bf r}}
\newcommand{\bq}{{\bf q}}

\newcommand{\Z}{\mathbb Z }
\newcommand{\tC}{\mathbb C }
\newcommand{\N}{\mathbb N }
\newcommand{\R}{\mathbb R }

\nc\bu{\mathbf U}

\newcommand{\lie}[1]{\mathfrak{#1}}
\newcommand{\wt}{\operatorname{wt}}
\DeclareMathOperator{\typ}{type}

\DeclareMathOperator{\gr}{gr}

\DeclareMathOperator{\spa}{span}
\DeclareMathOperator{\cha}{ch}
\DeclareMathOperator{\supp}{supp}

\DeclareMathOperator{\X}{X}
\DeclareMathOperator{\Ra}{R}
\DeclareMathOperator{\T}{T}

\begin{document}

\title[The PBW filtration and convex polytopes in type $\tt B$]{The PBW filtration and convex polytopes in type $\tt B$}
\author{Teodor Backhaus}
\address{Teodor Backhaus:\newline
Mathematisches Institut, Universit\"at zu K\"oln, Germany}
\email{tbackha@math.uni-koeln.de}
\thanks{T.B. was funded by the DFG Priority Program SPP 1388 “Representation theory”.}
\author{Deniz Kus}
\address{Deniz Kus:\newline
Mathematisches Institut, Universit\"at zu K\"oln, Germany}
\email{dkus@math.uni-koeln.de}
\thanks{D.K. was partially supported by the “SFB/TR 12-Symmetries and
Universality in Mesoscopic Systems”.}
\date{}

\subjclass[2010]{}
\begin{abstract}
We study the PBW filtration on irreducible finite--dimensional representations for the Lie algebra of type $\tt B_n$. We prove in various cases, including all multiples of the adjoint representation and all irreducible finite--dimensional representations for $\tt B_3$, that there exists a normal polytope such that the lattice points of this polytope parametrize a basis of the corresponding associated graded space. As a consequence we obtain several classes of examples for favourable modules and graded combinatorial character formulas.
\end{abstract}

\maketitle \thispagestyle{empty}
\section{Introduction}\label{section1}
Let $\Lg$ be a complex finite--dimensional simple Lie algebra with highest root $\theta$. The PBW filtration on finite--dimensional irreducible representations of $\Lg$ was studied in \cite{FFL11} and a description of the associated graded space in terms of generators and relations has been given in type $\tt A_n$ and $\tt C_n$ (see \cite{FFL11,FFL2011}). As a beautiful consequence the authors obtained a new class of bases parametrized by the lattice points of normal polytopes, which we call the FFL polytopes. A new class of bases for type $\tt G_2$ is established in \cite{G11} by using different arguments.\par
It turned out that the PBW theory has a lot of connections to many areas of representation theory. For example, in the branch of combinatorial representation theory the FFL polytopes can be used to provide models for Kirillov--Reshetikhin crystals (see \cite{K13,K12}). Further, a purely combinatorial research shows that there exists an explicit bijection between FFL polytopes and the well--known (generalized) Gelfand--Tsetlin polytopes (see \cite[Theorem 1.3]{ABS11}). Although Berenstein and Zelevinsky defined the $\tt B_n$--analogue of Gelfand--Tsetlin polytopes in \cite{BZ88} it is much more complicated to define the $\tt B_n$--analogue of FFL polytopes (see \cite[Section 4]{ABS11}). One of the motivations of the present paper is to better understand (the difficulties of) the PBW filtration in this type.\\ 
In the branch of geometric representation theory the PBW filtration can be used to study flat degenerations of generalized flag varieties. The degenerate flag variety of type $\tt A_n$ and $\tt C_n$ respectively can be realized inside a product of Grassmanians (see \cite[Theorem 2.5]{F11M} and \cite[Theorem 1.1]{FFL14D}) and furthermore the degenerate flag variety is isomorphic to an appropriate Schubert variety (see \cite[Theorem 1.1]{IL14}).
Another powerful tool of studying these varieties are favourable modules, where the properties of a favourable module are governed by the combinatorics of an associated normal polytope (see for details \cite{FFL13} or Section~\ref{section6}). It has been proved in \cite{FFL13} that the degenerate flag varieties associated to favourable modules have nice properties. For example, they are normal and Cohen--Macaulay and, moreover, the underlying polytope can be interpreted as the Newton-Okounkov body for the flag variety. In the same paper several classes of examples for favourable modules of type $\tt A_n$, $\tt C_n$ and $\tt G_2$ respectively are provided; more classes of examples were constructed in \cite{BD14,BF14,Fou14}.\par
Beyond these cases very little is known about the PBW filtration and whether there exists a normal polytope parametrizing a PBW basis of the associated graded space. In this paper we prove the existence of such polytopes for several classes of representations of type $\tt B_n$. Moreover, we construct favourable modules (see Section~\ref{section6}) and use the results of \cite{G11} to describe the associated graded space for type $\tt G_2$ in terms of generators and relations (see Section~\ref{section7}).\par
If $n\leq 3$ we obtain similar results as in the aforementioned cases, namely we associate to any dominant integral weight $\lambda$ a normal polytope and prove that a basis of the associated graded space can be parametrized by the lattice points of this polytope. In other words we observe that the difficulties of the PBW theory for type $\tt B_n$ show up if $n\geq 4$. Our results are the following; see Section~\ref{section5} for the precise definitions. 
\begin{theom}
Let $\Lg$ be the Lie algebra of type $\tt B_3$ and $\lambda\in P^+$ be a dominant integral weight. There exists a normal polytope $P(\lambda)$ with the following properties:
\begin{enumerate} 
\item The lattice points $S(\lambda)$ parametrize a basis of $V(\lambda)$ and $\gr V(\lambda)$ respectively. In particular, 
$$\{\X^{\mathbf s}v_{\lambda}\mid \mathbf s \in S(\lambda)\}$$
forms a basis of $\gr V(\lambda)$.
\item For $\lambda,\mu \in P^+$, we have 
$$S(\lambda)+S(\mu)=S(\lambda+\mu).$$
\item The character and graded $q$-character respectively is given by
$$\cha V(\lambda)=\sum_{\mu\in \lie h^{*}}|S(\lambda)^{\mu}|e^{\mu},\quad \cha_q \gr V(\lambda)=\sum_{\mathbf s\in S(\lambda)}e^{\lambda-\wt(\mathbf s)}q^{\sum s_{\beta}}.$$
\item We have an isomorphism of $S(\lie n^-)$--modules
$$\gr V(\lambda+\mu)\cong S(\lie n^-)(v_{\lambda}\otimes v_{\mu})\subseteq \gr V(\lambda)\otimes \gr V(\mu).$$
\item The module $V(\lambda)$ is favourable.
\end{enumerate}
\end{theom}
The describing inequalities of the poytope $P(\lambda)$ are given in Section~\ref{section5}. We remark that point (2) of the above theorem implies that the building blocks of $S(\lambda)$ are $S(\omega_i)$, $1\leq i \leq n$ as in the cases $\tt A_n$, $\tt C_n$ and $\tt G_2$. In particular, in order to construct a basis of $\gr V(\lambda)$ it is enough to construct the polytopes $P(\omega_i)$ associated to fundamental weights. For type $\tt B_n$ and $n\geq 4$ we need a different approach for the following reason. For $n=4$ we construct a polytope $P(\omega_3)$ such that the lattice points $S(\omega_3)$ parametrize a basis of $\gr V(\omega_3)$, but the Minkowski--sum $S(\omega_3)+S(\omega_3)$ has cardinatlity $\dim V(2\omega_3)-1$. We observe that the building blocks in this case are $S(\omega_3)$ and $S(2\omega_3)$. In particular, we construct polytopes $P(\omega_3)$ and $P(2\omega_3)$ such that a basis of $\gr V(m\omega_3)$ is given by 
$$\underbrace{S(2\omega_3)+\cdots+S(2\omega_3)}_{\lfloor\frac{m}{2}\rfloor}+\delta_{(m \text{ mod } 2), 1}S(\omega_3),$$
where $\delta_{r,s}$ denotes Kronecker's delta symbol. Our results are the following; we refer to Section~\ref{section4} and Section~\ref{section6} for the precise definition. 
\begin{theom}
Let $\Lg$ be the Lie algebra of type $\tt B_n$ and $\lambda=m\omega_i$ be a rectangular weight. There exists a convex polytope $P(\lambda)$ such that the follwing holds. If $1\leq i \leq 3$ ($n$ arbitrary) or $1\leq n\leq 4$ ($i$ arbitrary) we have
\begin{enumerate}
\item The lattice points $S(\lambda)$ parametrize a basis of $V(\lambda)$ and $\gr V(\lambda)$ respectively. In particular, 
$$\{\X^{\mathbf s}v_{\lambda}\mid \mathbf s \in S(\lambda)\}$$
forms a basis of $\gr V(\lambda)$.
\item 
We have $\gr V(\lambda)\cong S(\lie n^{-})/\mathbf I_{\lambda}$, where
$$\mathbf I_{\lambda}=S(\lie n^-)\Big(\bu(\lie n^+)\circ \spa\big\{x^{\lambda(\beta^{\vee})+1}_{-\beta}\mid \beta\in R^+\big\}\Big).$$
\item The character and graded $q$-character respectively is given by
$$\cha V(\lambda)=\sum_{\mu\in \lie h^{*}}|S(\lambda)^{\mu}|e^{\mu},\quad \cha_q \gr V(\lambda)=\sum_{\mathbf s\in S(\lambda)}e^{\lambda-\wt(\mathbf s)}q^{\sum s_{\beta}}.$$
\end{enumerate}
We set $\epsilon_i=1$ if $i\leq 2$ and $\epsilon_i=2$ else.
\begin{enumerate}
\item[(4)] We have an isomorphism of $S(\lie n^-)$--modules for all $\ell\in \Z_+$:
$$\gr V(\lambda+\epsilon_i\ell\omega_i)\cong S(\lie n^-)(v_{\lambda}\otimes v_{\epsilon_i\ell \omega_i})\subseteq \gr V(\lambda)\otimes \gr V(\epsilon_i\ell\omega_i).$$
\item[(5)] For all $k,\ell\in \Z_+$ we have
$$S((k+\epsilon_i\ell)\omega_i)=S(k\omega_i)+S(\epsilon_i\ell\omega_i).$$
\item[(6)] The module $V(\epsilon_i\lambda)$ is favourable.
\end{enumerate}
\end{theom} 
The describing inequalities of the poytope $P(m\omega_i)$ in terms of (double) Dyck paths are given in Section~\ref{section4}, where we also show that $S(m\omega_i)$ parametrizes a generating set of $\gr V(m\omega_i)$ for arbitrary $m\in \mathbb{Z}_+$ and $1\leq i \leq n$. We conjecture that the above theorem remains true for arbitrary rectangular weights (see Conjecture~\ref{conjbasistypeB}) and we verified the cases $n\leq 8$ and $m\leq 9$ with a computer program. \par
Our paper is organized as follows: In Section~\ref{section2} we give the main notations. In Section~\ref{section3} we present the PBW filtration and establish the elementary results needed in the rest of the paper. In Section~\ref{section4} we introduce the notion of Dyck paths for the special odd orthogonal Lie algebra and give in various cases a presentation for the associated graded space. In Section~\ref{section5} we associate to any dominant integral weight for $\tt B_3$ a normal polytope parametrizing a basis of the associated graded space. In Section~\ref{section6} we give classes of examples for favourable modules.
\section{Preliminaries}\label{section2}
We denote the set of complex numbers by $\tC$ and, respectively, the set of integers, non--negative integers, and positive integers  by $\Z$, $\Z_+$, and $\N$. Unless otherwise stated, all the vector spaces considered in this paper are $\tC$-vector spaces and $\otimes$ stands for $\otimes_{\mathbb C}$.



\subsection{}  We refer to \cite{Ca05,K90} for the general theory of Lie algebras. We denote by $\lie g$ a complex finite--dimensional simple Lie algebra. We fix a Cartan subalgebra $\lie h$ of $\lie g $ and denote by $R$ the set of roots of $\lie g$ with respect to $\lie h$. For $\alpha\in R$ we denote by $\alpha^{\vee}$ its coroot. We fix $\Delta=\{\alpha_1,\dots,\alpha_n\}$ a basis of simple roots for $R$; the corresponding sets of positive and negative roots are denoted as usual by $R^{\pm}$. For $1\leq i\leq n$, define $\omega_i\in\lie h^{*}$ by $\omega_i(\al_j^\vee)=\delta_{i,j}$, for $1\leq j\leq n$, where $\delta_{i,j}$  is the Kronecker's delta symbol. The element $\omega_i$ is the fundamental weight of $\lie g$ corresponding to the coroot $\al_i^\vee$.
Let  $Q=\oplus_{i=1}^n \Z \al_i$ be the root lattice of $R$ and $Q^+=\oplus_{i=1}^n \Z_+ \al_i$ be the respective $\Z_+$--cone. The weight lattice of $R$ is denoted by $P$ and the cone of dominant weights is denoted by $P^+$. Let $\Z[P]$ be the integral group ring of $P$ with basis $e^{\mu}, \mu\in P$.


\subsection{}
Given $\alpha\in R^{+}$ let $\lie g_{\pm\alpha}$ be the corresponding root space and fix a generator $x_{\pm \alpha}\in \lie g_{\pm\alpha}$. We define several subalgebras of $\lie g$ that will be needed later. Let $\lie b$ be the Borel subalgebra corresponding to $R^+$, and let $\lie n^+$ be its nilpotent radical,
$$\lie b=\lie h\oplus \lie n^+,\ \  \ \ \lie n^\pm =\bigoplus_{\alpha\in R^+}\lie g_{\pm \alpha}.$$  

The Lie algebra $\lie g$ has a triangular decomposition
$$\lie g = \lie n^-\oplus \lie h \oplus \lie n^+.$$
For the subset $\Delta-\{\alpha_{i_1},\dots,\alpha_{i_s}\}$ of $\Delta$ we denote by $\lie p_{i_1,\dots,i_s}$ the corresponding parabolic subalgebra of $\lie g$, i.e. the Lie algebra generated by $\lie b$ and all root spaces $\lie g_{-\alpha}$, \ $\alpha\in \Delta-\{\alpha_{i_1},\dots,\alpha_{i_s}\}$. The maximal parabolic subalgebras correspond to subsets of the form $\Delta-\{\alpha_{i}\}$, $1\leq i \leq n$. 
The Lie algebra $\lie g$ contains the parabolic subalgebra as a direct summand and therefore 

$$\lie g=\lie p_{i_1,\dots,i_s}\oplus \lie n_{i_1,\dots,i_s}^-.$$ 
 
We can split off $\lie p_{i_1,\dots,i_s}$ and consider the nilpotent vector space complement with root space decomposition 

 $$
 \lie n_{i_1,\dots,i_s}^-=\bigoplus_{\alpha\in R_{i_1,\dots,i_s}^+} \lie g_{-\alpha}.
 $$
For instance, if $\lie g$ is of type $\tt A_n$ we have $R^+=\{\alpha_{r,s}\mid 1\leq r\leq s \leq n\}$ and $R_{i}^+=\{\alpha_{r,s}\in R^+ \mid r\leq i \leq s\}$ where $\alpha_{r,s}=\sum_{j=r}^s\alpha_j$. In the following we shall be interested in maximal parabolic subalgebras. 
\section{PBW filtration and graded spaces}\label{section3}

We start by recalling some standard notation and results on the representation theory of $\lie g$.
\subsection{}
A $\lie g$--module $V$ is said to be a weight module if it is $\lie h$--semisimple, $$V=\bigoplus_{\mu\in\lie h^*}V^\mu,\ \ \ V^\mu=\{v\in V \mid hv=\mu(h)v,\ \ h\in\lie h\}.$$ Set $\wt V=\{\mu\in\lie h^*:V^\mu\ne 0\}$.  
Given $\lambda\in P^+$, let $V(\lambda)$ be the irreducible finite--dimensional $\lie g$--module generated by an element $v_\lambda$ with defining relations:
\begin{equation}\label{simplerelations}\lie n^+ v_\lambda=0,\ \ \ hv_\lambda=\lambda(h)v_\lambda,\ \ \ x_{-\alpha}^{\lambda(\alpha^\vee)+1}v_\lambda=0,\end{equation}
for all $h\in\lie h$ and $\alpha\in R^+$.  We have $\wt V(\lambda)\subset\lambda- Q^+$ and $\wt V(\lambda)$ is a $W$--invariant subset of $\lie h^*$. If $\dim V^{\mu}<\infty$ for all $\mu\in \wt V$, then we define $\ch V: \lie h^{*}\longrightarrow \Z_+$, by sending $\mu\mapsto \dim V^{\mu}$. If $\wt V$ is a finite set, then
$$\ch V=\sum_{\mu\in \lie h^{*}}\dim V^{\mu}e^{\mu}\in \Z[P].$$
\subsection{}
A $\Z_+$--filtration of a vector space $V$ is a collection of subspaces $\mathbf F=\{V_s\}_{s\in \Z_+}$, such that $V_{s-1}\subseteq V_s$ for all $s\geq 1$. We build the associated graded space with respect to the filtration $\mathbf F$ 
$$\gr^{\mathbf F} V=\bigoplus_{s\in \Z_+}V_s/V_{s-1}, \mbox{where $V_{-1}=0$}.$$
In this paper we shall be interested in the PBW filtration of the irreducible module $V(\lambda)$ which we will explain now. Consider the increasing degree filtration on the universal enveloping algebra $\bu(\lie n^-)$:
 \begin{equation*}\label{PBWfiltration}\bu(\lie n^-)_s=\spa\{x_{1}\cdots x_{l}\mid x_j\in\lie n^-, l\leq s\},\end{equation*}
for example, $\bu(\lie n^-)_0=\tC.$ The induced increasing filtration $\mathbf V=\{V(\lambda)_s\}_{s\in \Z_+}$ on $V(\lambda)$ where $V(\lambda)_s:=\bu(\lie n^-)_sv_{\lambda}$ is called the PBW filtration. With respect to the PBW filtration we build the associated graded space $\gr^{\mathbf V} V(\lambda)$ as above. To keep the notation as simple as possible, we will write $\gr^{} V(\lambda)$ to refer to $\gr^{\mathbf V} V(\lambda)$. The graded $q$--character is defined as 
$$\ch_q \gr V(\lambda)=\sum_{\mu\in \lie h^{*}}\Big(\sum_{s\geq 0}(\dim V(\lambda)^{\mu}_s/V(\lambda)^{\mu}_{s-1}) q^s\Big)e^{\mu}, \mbox{ where } \gr V(\lambda)^{\mu}=\bigoplus_{s\in \Z_+}V(\lambda)^{\mu}_s/V(\lambda)^{\mu}_{s-1}.$$ 
The following is immediate:
\begin{lem}\label{n+structure}
The action of $\bu(\lie n^-)$ on $V(\lambda)$ induces a structure of a $S(\lie n^-)$ module on $\gr V(\lambda)$. Moreover,
$$\gr V(\lambda)=S(\lie n^-)v_{\lambda}\cong S(\lie n^-)/\mathbf I_{\lambda},$$ 
for some homogeneous ideal $\mathbf I_{\lambda}$. The action of $\bu(\lie n^+)$ on $V(\lambda)$ induces a structure of a $\bu(\lie n^+)$ module on $\gr V(\lambda)$.
\end{lem} 
By the previous lemma, the representation $\gr V(\lambda)$ is cyclic as a $S(\lie n^-)$--module. By the PBW theorem and the defining relations \eqref{simplerelations} of $V(\lambda)$ we obtain the following proposition.
\begin{prop}\label{spanningset1}
The set 
$$\Big\{\prod_{\beta\in R^+}x_{-\beta}^{m_{\beta}}v_{\lambda} \mid m_{\beta}\in \Z_+, m_{\beta}\leq \lambda(\beta^{\vee})\Big\}$$
is a (finite) spanning set of $\gr V(\lambda)$. 
\end{prop}
For a multi--exponent $\mathbf s=(s_{\beta})_{\beta\in R^+}\in \Z^{|R^+|}_+$ \big(resp. $\mathbf s=(s_{\beta})_{\beta\in R_{i}^+} \in \Z^{|R_{i}^+|}_+$\big)  
we denote the corresponding monomial $\prod_{\beta\in R^+}x_{-\beta}^{s_{\beta}}$ \big(resp. $\prod_{\beta\in R_{i}^+} x_{-\beta}^{s_{\beta}}$\big) 
for simplicity by $\X^{\mathbf s}\in S(\lie n^-)$.

In recent years it became a popular goal to determine the $S(\lie n^-)$--structure of the representations $\gr V(\lambda)$, i.e. to describe the ideals $\mathbf I_{\lambda}$ and furthermore to find a PBW basis for these graded representations, favourably parametrized by the integral points of a suitable convex polytope. For the finite--dimensional Lie algebras of type $\tt A_n$, $\tt C_n$ and $\tt G_2$ various results are known which we will discuss later (see \cite{FFL11,FFL2011,G11}). The focus of this paper is on the Lie algebra of type $\tt B_n$ where many technical difficulties show up.

\subsection{}
Let $\mathbf D\subseteq\mathcal P(R^+)$ be a subset of the power set of $R^+$. We attach to each element $\mathbf p\in \mathbf D$ a non--negative integer $M_{\mathbf p}(\lambda)$. We consider the following polytope 
\begin{equation}\label{polyaundcund}P(\mathbf D,\lambda)=\Big\{\mathbf s=(s_{\beta})_{\beta\in R^+}\in \R_+^{|R^+|}\mid \forall \mathbf p\in \mathbf D: \sum_{\beta\in \mathbf p} s_{\beta} \leq M_{\mathbf p}(\lambda) \Big\}.\end{equation}
The integral points of the above polytope are denoted by $S(\mathbf D,\lambda)$.
The proof of part (i) of the following theorem for type $\tt A_n$ can be found in \cite{FFL11}, for type $\tt C_n$ in \cite{FFL2011} and for type $\tt G_2$ in \cite{G11}. Part (ii) is only proved for type $\tt A_n$ and $\tt C_n$, but a simple calculation shows that part (ii) for type $\tt G_2$ remains true (for a proof see Proposition~\ref{forg2id} in the appendix).
\begin{thm}\label{basistypeAandC}
There exist a set $\mathbf D\subseteq\mathcal P(R^+)$ and suitable non--negative integers $M_{\mathbf p}(\lambda)$ attached to each element $\mathbf p\in \mathbf D$, such that the following holds:
\renewcommand{\theenumi}{\roman{enumi}}%
\begin{enumerate}
\item The lattice points $S(\mathbf D,\lambda)$ parametrize a basis of $V(\lambda)$ and $\gr V(\lambda)$ respectively. In particular, 
$$\{\X^{\mathbf s}v_{\lambda}\mid \mathbf s \in S(\mathbf D,\lambda)\}$$
forms a basis of $\gr V(\lambda)$.
\item We have
$$\mathbf I_{\lambda}=S(\lie n^-)\big(\bu(\lie n^+)\circ \spa\{x^{\lambda(\beta^{\vee})+1}_{-\beta}\mid \beta\in R^+\}\big).$$
\end{enumerate}
\end{thm}
We note that the order in the theorem above is important when treating the representation $V(\lambda)$, but we can choose for any $\mathbf s \in S(\mathbf D,\lambda)$ an arbitrary order of factors $x_{-\beta}$ in the product $\X^{\mathbf s}$, such that the set $$\{\X^{\mathbf s}v_{\lambda}\mid \mathbf s \in S(\mathbf D,\lambda)\}$$ forms a basis of $V(\lambda)$.
\begin{rem}
The set $\mathbf D$ and non--negative integers $M_{\mathbf p}(\lambda)$ are explicitly described in these papers. Part (i) of the above theorem for type $\tt A_n$ was conjectured by Vinberg \cite{V05}. 
\end{rem}
Another interesting point is to understand the geometric aspects of the PBW filtration. In \cite{F2012S} degenerate flag varieties have been introduced which are certain varieties in the projectivization $\mathbb{P}(\gr V(\lambda))$ of $\gr V(\lambda)$. In type $\tt A_n$ (see \cite{F2012S,FF2013A}) and type $\tt C_n$ (see \cite{FFL14D}) it has been shown that the degenerate flag varieties can be embedded into a product of Grassmanians and desingularizations have been constructed.
Recently in \cite{FFL13} the notion of favourable modules has been introduced whose properties are governed by the combinatorics of an associated polytope and it has been shown that the corresponding degenerate flag varieties have nice properties, e.g. they are projectively normal and arithmetically Cohen-Macaulay varieties (see also Section~\ref{section7}). Especially it has been proved that $V(\lambda)$ for types $\tt A_n$, $\tt C_n$ and $\tt G_2$ are favourable (with respect to the polytope from Theorem~\ref{basistypeAandC}), where the proof of this fact uses the Minkowski sum property of these polytopes. Our aim is to obtain similar results to Theorem~\ref{basistypeAandC} for type $\tt B_n$ for certain dominant integral weights and, motivated by the corresponding nice geometry of favourable modules, to construct various favourable modules.

\section{Dyck path, polytopes and PBW bases}\label{section4}
The notion of Dyck paths is used in the papers \cite{FFL11,FFL2011} in order to describe the set $\mathbf D$ from Theorem~\ref{basistypeAandC} (and thus $S(\mathbf D,\lambda)$), but appears earlier in the literature in a different context. In this section we define two types of paths (type 1 and type 2), which we also call Dyck paths to avoid deviating from the established terminology. The set of Dyck paths of type 1 is similar to the definition given in \cite{FFL11,FFL2011}, while the type 2 Dyck paths are unions of type 1 Dyck paths with some extra conditions and are called double Dyck paths. 
\subsection{}
To each finite partially ordered set $(S,\leq)$ we can associate a diagram, called the Hasse diagram. The vertices are given by the elements in $S$ and we draw a line segment from $x$ to $y$ whenever $y$ covers $x$, that is, whenever $x < y$ and there is no $z$ such that $x < z < y$. We consider the partial order $\leq$ on $R^+$ given by $\alpha \leq \beta :\Leftrightarrow \beta -\alpha \in Q^+$. We shall be interested in the Hasse diagram of $(R^+,\leq)$ and $(R^+_{i},\leq)$. Note that the Hasse diagram of $R^+_{i}$ is obtained from the Hasse diagram of $R^+$ by erasing all vertices $\alpha\in R^+\backslash R^+_{i}$. 
\begin{example}

We find below the Hasse diagram of $(R^+,\leq)$ for type $\tt A_n$ and $\tt B_n$ respectively. The Hasse diagram for $(R^+_3,\leq)$ of type $\tt B_n$ is highlighted in red. Recall that the highest root is denoted by $\theta$.
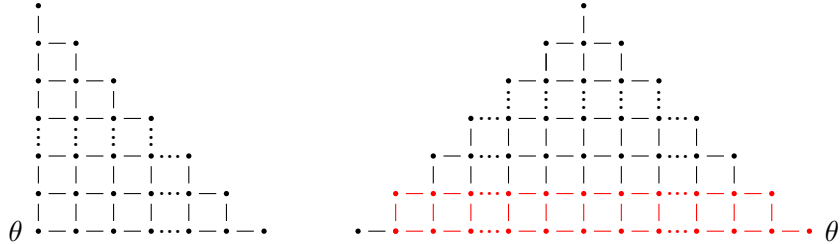
\begin{center}
\begin{tikzpicture}
  \node (3) at (0.5-9.5,0.25+6.25) {};
  \fill[black] (3) circle (1pt);
  \node (4) at (1-9.5,0.25+6.25) {};
  \fill[black] (4) circle (1pt);
  \node (5) at (1.5-9.5,0.25+6.25) {};
  \fill[black] (5) circle (1pt);
  \node (6) at (0.5-9.5,0.75+6.25) {};
  \fill[black] (6) circle (1pt);
  \node (7) at (1-9.5,0.75+6.25) {};
  \fill[black] (7) circle (1pt);
  \node (9) at (0.5-9.5,1.25+6.25) {};
  \fill[black] (9) circle (1pt);
 
  \node (12) at (0.5-9.5,-0.25+6.25) {};
  \fill[black] (12) circle (1pt);
  \node (13) at (1-9.5,-0.25+6.25) {};
 \fill[black] (13) circle (1pt);
  \node (14) at (1.5-9.5,-0.25+6.25) {};
  \fill[black] (14) circle (1pt);
  \node (27) at (0.5-9.5,-0.75+6.25) {};
  \fill[black] (27) circle (1pt);
  \node (28) at (1-9.5,-0.75+6.25) {};
  \fill[black] (28) circle (1pt);
  \node (29) at (1.5-9.5,-0.75+6.25) {};
  \fill[black] (29) circle (1pt);
 \node (21) at (2-9.5,-0.75+6.25) {};
 \fill[black] (21) circle (1pt);
 \node (25) at (2.5-9.5,-0.75+6.25) {};
 \node (40) at (2.25-9.5,-0.75+6.25) {...};
 \fill[black] (25) circle (1pt);
 \node (39) at (2-9.5,-0.25+6.25) {};
 \fill[black] (39) circle (1pt);
 \draw  (12) -- (13) -- (14) -- (39);
 \draw (27) -- (28) -- (29) -- (21);
     \node (16) at (1.5-9.5,-0.5+6.15) {.};
    \node (16x) at (1.5-9.5,-0.5+6.25) {.};
    \node (16xx) at (1.5-9.5,-0.5+6.35) {.};
   
    \node (a16) at (1-9.5,-0.5+6.15) {.};
    \node (a16x) at (1-9.5,-0.5+6.25) {.};
    \node (a16xx) at (1-9.5,-0.5+6.35) {.};
   
    \node (b16) at (0.5-9.5,-0.5+6.15) {.};
    \node (b16x) at (0.5-9.5,-0.5+6.25) {.};
    \node (b16xx) at (0.5-9.5,-0.5+6.35) {.};
   
    \node (c16) at (2-9.5,-0.5+6.15) {.};
    \node (c16x) at (2-9.5,-0.5+6.25) {.};
    \node (c16xx) at (2-9.5,-0.5+6.35) {.};
  \node (17) at (0.5-9.5,-1.25+6.25) {};
  \fill[black] (17) circle (1pt);
  \node (18) at (1-9.5,-1.25+6.25) {};
  \fill[black] (18) circle (1pt);
  \node (19) at (1.5-9.5,-1.25+6.25) {};
  \fill[black] (19) circle (1pt);
  \node (20) at (2-9.5,-1.25+6.25) {};
  \fill[black] (20) circle (1pt);
  \node (22) at (2.25-9.5,-1.25+6.25) {...};
  \node (23) at (2.5-9.5,-1.25+6.25) {};
  \fill[black] (23) circle (1pt);
  \node (24) at (3-9.5,-1.25+6.25) {};
  \fill[black] (24) circle (1pt);
 

	\node[label=left: $\theta\!$] (31) at (0.5-9.5,4.5) {};
  \fill[black] (31) circle (1pt);
  \node (32) at (1-9.5,4.5) {};
  \fill[black] (32) circle (1pt);
  \node (33) at (1.5-9.5,4.5) {};
  \fill[black] (33) circle (1pt);
  \node (34) at (2-9.5,4.5) {};
  \fill[black] (34) circle (1pt);
  \node (35) at (2.25-9.5,4.5) {...};
  \node (36) at (2.5-9.5,4.5) {};
  \fill[black] (36) circle (1pt);
  \node (37) at (-6.5,4.5) {};
  \fill[black] (37) circle (1pt);
  \node (38) at (-6,4.5) {};
  \fill[black] (38) circle (1pt);

   \draw (3)  -- (4);
   \draw (4)  -- (5);
   \draw (6)  -- (7);
   \draw (3)  -- (6);
   \draw (4)  -- (7);
   \draw (6)  -- (9);
   \draw (3)  -- (12) ;
   \draw (4)  -- (13) ;
   \draw (5)  -- (14) ;
   \draw (27) -- (17);
   \draw (28) -- (18);
   \draw  (29) -- (19);
   \draw (20) -- (21);
   \draw (17) -- (18) -- (19) -- (20);
   \draw (17) -- (31);
   \draw (18) -- (32);
   \draw (19) -- (33);
   \draw (20) -- (34);
   \draw (31) -- (32) -- (33) -- (34);
   \draw (36) -- (37) -- (38);
   \draw (23) -- (25);
    \draw (23) -- (24);
    \draw (23) -- (36);
    \draw (24) -- (37);

  \node (311) at (0.25,0.5 + 4) {};
  \fill[red] (311) circle (1pt);
  \node (411) at (0.25,1+ 4) {};
  \fill[red] (411) circle (1pt);
  \node (511) at (0.25,1.5+ 4) {};
  \fill[black] (511) circle (1pt);
  \node (611) at (0.75,0.5+ 4) {};
  \fill[red] (611) circle (1pt);
  \node (711) at (0.75,1+ 4) {};
  \fill[red] (711) circle (1pt);
  \node[label=right: $\!\theta$] (911) at (1.25,0.5+ 4) {};
  \fill[red] (911) circle (1pt);
 
  \node (1211) at (-0.25,0.5+ 4) {};
  \fill[red] (1211) circle (1pt);
  \node (1311) at (-0.25,1+ 4) {};
 \fill[red] (1311) circle (1pt);
  \node (1411) at (-0.25,1.5+ 4) {};
  \fill[black] (1411) circle (1pt);
  \node (2711) at (-0.75,0.5+ 4) {};
  \fill[red] (2711) circle (1pt);
  \node (2811) at (-0.75,1+ 4) {};
  \fill[red] (2811) circle (1pt);
  \node (2911) at (-0.75,1.5+ 4) {};
  \fill[black] (2911) circle (1pt);
 \node (2111) at (-0.75,2+ 4) {};
 \fill[black] (2111) circle (1pt);
 \node (2511) at (-0.75,2.5+ 4) {};

 \node (4011) at (-0.75,2.15+ 4) {.};
\node (4011x) at (-0.75,2.25+ 4) {.};
\node (4011xx) at (-0.75,2.35+ 4) {.};

 \fill[black] (2511) circle (1pt);
 \node (3911) at (-0.250,2+ 4) {};
 \fill[black] (3911) circle (1pt);
 \draw  (1311) -- (1411) -- (3911);
\draw[red]  (1211) -- (1311);
 \draw (2811) -- (2911) -- (2111);
     \node (1611) at (-0.5,0.5+ 4) {\textcolor{red}{...}};
    \node (1611x) at (-0.5,2+ 4) {...};
    \node (1611xx) at (-0.5,1.5+ 4) {...};
    \node (1611xxx) at (-0.5,1+ 4) {\textcolor{red}{...}};
  \node (1711) at (-1.25,0.5+ 4) {};
  \fill[red] (1711) circle (1pt);
  \node (1811) at (-1.25,1+ 4) {};
  \fill[red] (1811) circle (1pt);
  \node (1911) at (-1.25,1.5+ 4) {};
  \fill[black] (1911) circle (1pt);
  \node (2011) at (-1.25,2+ 4) {};
  \fill[black] (2011) circle (1pt);
   
  \node (2211) at (-1.25,2.15+ 4) {.};
    \node (2211x) at (-1.25,2.25+ 4) {.};
    \node (2211xx) at (-1.25,2.35+ 4) {.};
   
  \node (2311) at (-1.25,2.5+ 4) {};
  \fill[black] (2311) circle (1pt);
  \node (2411) at (-1.25,3+ 4) {};
  \fill[black] (2411) circle (1pt);
 
  \node (3111) at (-1.75,0.5+ 4) {};
  \fill[red] (3111) circle (1pt);
  \node (3211) at (-1.75,1+ 4) {};
  \fill[red] (3211) circle (1pt);
  \node (3311) at (-1.75,1.5+ 4) {};
  \fill[black] (3311) circle (1pt);
  \node (3411) at (-1.75,2+ 4) {};
  \fill[black] (3411) circle (1pt);
   
  \node (3511) at (-1.75,2.15+ 4) {.};
    \node (3511xx) at (-1.75,2.25+ 4) {.};
    \node (3511xxx) at (-1.75,2.35+ 4) {.};
   
  \node (3611) at (-1.75,2.5+ 4) {};
  \fill[black] (3611) circle (1pt);
  \node (3711) at (-1.75,3+ 4) {};
  \fill[black] (3711) circle (1pt);
  \node (3811) at (-1.75,3.5+ 4) {};
  \fill[black] (3811) circle (1pt);
 
 \node (17112) at (-2.25,0.5+ 4) {};
  \fill[red] (17112) circle (1pt);
  \node (18112) at (-2.25,1+ 4) {};
  \fill[red] (18112) circle (1pt);
  \node (19112) at (-2.25,1.5+ 4) {};
  \fill[black] (19112) circle (1pt);
  \node (20112) at (-2.25,2+ 4) {};
  \fill[black] (20112) circle (1pt);
   
  \node (22112x) at (-2.25,2.15+ 4) {.};
    \node (22112xx) at (-2.25,2.25+ 4) {.};
    \node (22112xxx) at (-2.25,2.35+ 4) {.};
   
  \node (23112) at (-2.25,2.5+ 4) {};
  \fill[black] (23112) circle (1pt);
  \node (24112) at (-2.25,3+ 4) {};
  \fill[black] (24112) circle (1pt);

\node (27112) at (-2.75,0.5+ 4) {};
  \fill[red] (27112) circle (1pt);
  \node (28112) at (-2.75,1+ 4) {};
  \fill[red] (28112) circle (1pt);
  \node (29112) at (-2.75,1.5+ 4) {};
  \fill[black] (29112) circle (1pt);
 \node (21112) at (-2.75,2+ 4) {};
 \fill[black] (21112) circle (1pt);
 \node (25112) at (-2.75,2.5+ 4) {};
\fill[black] (25112) circle (1pt);
\node (40112) at (-2.75,2.15+ 4) {.};
\node (40112x) at (-2.75,2.25+ 4) {.};
\node (40112xx) at (-2.75,2.35+ 4) {.};

\node (3112) at (-3.75,0.5 + 4) {};
  \fill[red] (3112) circle (1pt);
  \node (4112) at (-3.75,1+ 4) {};
  \fill[red] (4112) circle (1pt);
  \node (5112) at (-3.75,1.5+ 4) {};
  \fill[black] (5112) circle (1pt);
  \node (6112) at (-4.25,0.5+ 4) {};
  \fill[red] (6112) circle (1pt);
  \node (7112) at (-4.25,1+ 4) {};
  \fill[red] (7112) circle (1pt);
  \node (9112) at (-4.25,0.5+ 4) {};
  \fill[red] (9112) circle (1pt);
 
  \node (12112) at (-3.25,0.5+ 4) {};
  \fill[red] (12112) circle (1pt);
  \node (13112) at (-3.25,1+ 4) {};
 \fill[red] (13112) circle (1pt);
  \node (14112) at (-3.25,1.5+ 4) {};
  \fill[black] (14112) circle (1pt);

\node (39112) at (-3.25,2+ 4) {};
 \fill[black] (39112) circle (1pt);
 \draw  (13112) -- (14112) -- (39112);
\draw[red]  (12112) -- (13112);
 \draw (28112) -- (29112) -- (21112);
\draw[red] (27112) -- (28112);
   
    \node (16112) at (-3,1+ 4) {\textcolor{red}{...}};
     \node (16112x) at (-3,1.5+ 4) {...};
    \node (16112xx) at (-3,2+ 4) {...};
    \node (16112xxx) at (-3,0.5+ 4) {\textcolor{red}{...}};
   
\node (9112) at (-4.75,0.5+ 4) {};
 \fill[black] (9112) circle (1pt);
   \draw[red] (311)  -- (411);
   \draw (411)  -- (511);
   \draw[red] (611)  -- (711);
   \draw[red] (311)  -- (611);
   \draw[red] (411)  -- (711);
   \draw[red] (611)  -- (911);
   \draw[red] (311)  -- (1211) ;
   \draw[red] (411)  -- (1311) ;
   \draw (511)  -- (1411) ;
   \draw[red] (2711) -- (1711);
   \draw[red] (2811) -- (1811);
   \draw  (2911) -- (1911);
   \draw (2011) -- (2111);
   \draw[red] (1711) -- (1811);
    \draw[red] (2711) -- (2811);
    \draw (1811) -- (1911) -- (2011);
   \draw[red] (1711) -- (3111);
   \draw[red] (1811) -- (3211);
   \draw (1911) -- (3311);
   \draw (2011) -- (3411);
   \draw[red] (3111) -- (3211);
  \draw (3211) -- (3311) -- (3411);
    \draw (3611) -- (3711) -- (3811);
   \draw (2311) -- (2511);
    \draw (2311) -- (2411);
    \draw (2311) -- (3611);
    \draw (2411) -- (3711);
       
   \draw (18112) -- (19112) -- (20112);
    \draw[red] (17112) -- (18112);
   \draw[red] (17112) -- (3111);
   \draw[red] (18112) -- (3211);
   \draw (19112) -- (3311);
   \draw (20112) -- (3411);
    \draw (23112) -- (24112);
    \draw (23112) -- (3611);
    \draw (24112) -- (3711);

         \draw[red] (3112)  -- (4112);
   \draw (4112)  -- (5112);
   \draw[red] (6112)  -- (7112);
   \draw[red] (3112)  -- (6112);
   \draw[red] (4112)  -- (7112);
   \draw[black] (6112)  -- (9112);
   \draw[red] (3112)  -- (12112) ;
   \draw[red] (4112)  -- (13112) ;
   \draw (5112)  -- (14112) ;
   \draw[red] (27112) -- (17112);
   \draw[red] (28112) -- (18112);
   \draw  (29112) -- (19112);
   \draw (20112) -- (21112);
 
   \draw (23112) -- (25112);
    \draw (23112) -- (24112);

\end{tikzpicture}
\end{center}
\end{example}
\subsection{}\label{dyckfori}
\textit{For the rest of this section we fix $i\in \{1,\dots,n\}$ and let $\lambda=m\omega_i$ for some $m\in \Z_+$}. All roots of type $\tt B_n$ are of the form $\alpha_p+\cdots+\alpha_q$ for some $1\leq p\leq q \leq n$ or of the form $\alpha_p+\cdots+\alpha_{2n-q}+2\alpha_{2n-q+1}+\cdots+2\alpha_n$ for some $1\leq p\leq 2n-q <n$. To keep the notation as simple as possible we define 
$$\alpha_{p,q}:=\begin{cases}\alpha_p+\cdots+\alpha_q,&\text{ if $1\leq p\leq q \leq n$}\\
\alpha_p+\cdots+\alpha_{2n-q}+2\alpha_{2n-q+1}+\cdots+2\alpha_n, & \text{ if $1\leq p\leq 2n-q <n$.}\end{cases}$$
Furthermore, we write $R_{i}^+(\ell)$ for $R^+_{i}\backslash \big(R^+_{i}\cap\{\alpha_{p,q}\mid q>\ell\}\big)$. We call a subset of positive roots $\mathbf p=\{\beta(1),\dots,\beta(k)\}, k\geq 1$ a Dyck path of type 1 if and only if the following two conditions are satisfied
\begin{equation}\label{d1}
\begin{aligned} 
&\bullet \beta(1)=\al_{1,i}, \beta(k)=\al_{i,2n-i-1}\ \mbox{ or }\ \beta(1)=\al_{1,i+1},\beta(k)=\al_{i,2n-i}\\
&\bullet \mbox{ if }\beta(s) =\al_{p,q},\mbox{ then }\beta(s+1) =\al_{p,q+1}\mbox{ or }\beta(s+1) =\al_{p+1,q}.
\end{aligned}
\end{equation} The set of all type 1 Dyck path is denoted by $\mathbf D^{\typ 1}$ and $\mathbf D_1^{\typ 1}$ (resp. $\mathbf  D_2^{\typ 1}$) denotes the subset consisting of all type 1 Dyck paths starting at $\al_{1,i}$ (resp. $\al_{1,i+1}$). Furthermore, we call a subset of positive roots $\mathbf p=\{\beta(1),\dots,\beta(k)\}, k\geq 1$ a Dyck path of type 2 if and only if we can write $\mathbf p=\mathbf p_1\cup \mathbf p_2$ \big($\mathbf p_1=\{\beta_1(1),\dots,\beta_1(k_1)\}, k_1\geq 1, \mathbf p_2=\{\beta_2(1),\dots,\beta_2(k_2)\}, k_2\geq 1$\big) with the following properties:
\begin{equation*}\label{d2}
\begin{aligned}
&\bullet \beta_1(1)=\alpha_{1,i}, \beta_2(1)=\alpha_{2,i} \mbox{ and } \beta_1(k_1)=\alpha_{j,2n-j}, \beta_2(k_2) = \alpha_{j+1,2n-j-1} \mbox{ for some } 1 \leq j< i\\
&\bullet \mathbf p_1 \mbox{ and } \mathbf p_2 \mbox{ satisfy the second property of } \eqref{d1}\\
&\bullet \mathbf p_1\cap \mathbf p_2=\emptyset\\
\end{aligned}
\end{equation*}
The first property means that the last root in $\mathbf p_2$ is the upper right neighbour of the last root in $\mathbf p_1$ in the Hasse diagram of $(R^+_i,\leq)$. The set of all type 2 Dyck paths is denoted by $\mathbf D^{\typ 2}$. Summarizing, a type 1 Dyck path is a path in the sense of \cite{FFL11} in a specific area of the Hasse diagram of $(R^+_i,\leq)$ and a type 2 Dyck path can be written as a disjoint union of two single type 1 Dyck paths. For this reason, we call the elements in $\mathbf D^{\typ 2}$ double Dyck paths.
\begin{defn}\label{dyckpath}
We call a subset $\mathbf p$ of positive roots a Dyck path if and only if $\mathbf p\in \mathbf D:=\mathbf D^{\typ 1}\cup \mathbf D^{\typ 2}$.
\end{defn}
Note that $\mathbf D^{\typ 1}=\emptyset$ if $i=n$ and $\mathbf D^{\typ 2}=\emptyset$ if $i=1$ and $\mathbf D^{\typ 2}=\{R^+_2\}$ if $i=2$.
The interpretation of Dyck paths in the Hasse diagram is very helpful. The left figure (resp. right figure) shows the form of a type 1 (resp. type 2) Dyck path.
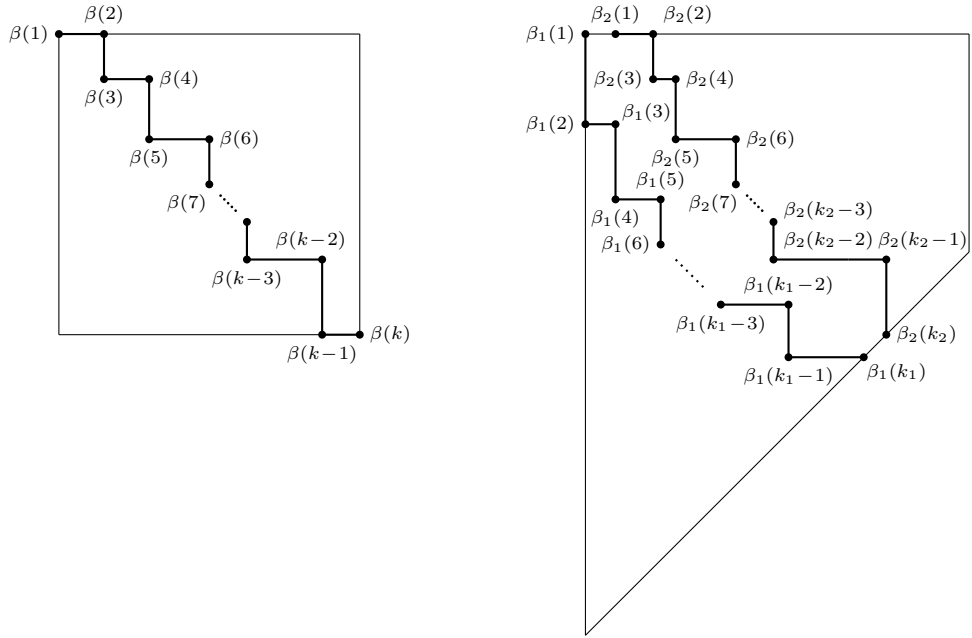
\begin{center}
\begin{tikzpicture}
\coordinate[label=left:\tiny{${\beta(1)}$}] (A) at (-2,2);
\coordinate[label=right:\tiny{$$}] (B) at (2,2);
\coordinate[label=right:\tiny{$\beta(k)$}] (C) at (2,-2);
\coordinate[label=left:\tiny{$$}] (D) at (-2,-2);
\coordinate[label=above:\tiny{$\beta(2)$}] (E) at (-1.4,2);
\coordinate[label=below:\tiny{$\beta(3)$}] (F) at (-1.4,1.4);
\coordinate[label=right:\tiny{$\beta(4)$}] (G) at (-0.8,1.4);
\coordinate[label=below:\tiny{$\beta(5)$}] (H) at (-0.8,0.6);
\coordinate[label=right:\tiny{$\beta(6)$}] (I) at (-0.0,0.6);
\coordinate[label=left:\tiny{$$}] (J) at (-0.0,0.0);
\coordinate[label=below:\tiny{$\beta(7)$}] (JJ) at (-0.3,0.0);
\coordinate[label=right:\tiny{$$}] (M1) at (0.15,-0.15);
\coordinate[label=right:\tiny{$$}] (M) at (0.2,-0.2);
\coordinate[label=right:\tiny{$$}] (K) at (0.25,-0.25);
\coordinate[label=right:\tiny{$$}] (L) at (0.3,-0.3);
\coordinate[label=right:\tiny{$$}] (M2) at (0.35,-0.35);

\coordinate[label=below:\tiny{$\beta(k\!-\!1)$}] (N) at (1.5,-2);
\coordinate[label=above:\tiny{$$}] (O) at (1.5,-1);
\coordinate[label=above:\tiny{$\beta(k\!-\!2)$}] (OO) at (1.35,-1);
\coordinate[label=below:\tiny{$\beta(k\!-\!3)$}] (P) at (0.5,-1);
\coordinate[label=below:\tiny{$$}] (Q) at (0.5,-0.5);
\draw (A) -- (B);
\draw (C) -- (B);
\draw (C) -- (D);
\draw (A) -- (D);
\draw[thick] (E) -- (A);
\draw[thick] (E) -- (F);
\draw[thick] (F) -- (G);
\draw[thick] (H) -- (G);
\draw[thick] (H) -- (I);
\draw[thick] (J) -- (I);
\draw[thick] (N) -- (C);
\draw[thick] (N) -- (O);
\draw[thick] (P) -- (O);
\draw[thick] (P) -- (Q);
\fill (A) circle (1.5pt);
\fill (C) circle (1.5pt);
\fill (E) circle (1.5pt);
\fill (F) circle (1.5pt);
\fill (G) circle (1.5pt);
\fill (H) circle (1.5pt);
\fill (I) circle (1.5pt);
\fill (J) circle (1.5pt);
\fill (K) circle (0.5pt);
\fill (L) circle (0.5pt);
\fill (M) circle (0.5pt);
\fill (M1) circle (0.5pt);
\fill (M2) circle (0.5pt);
\fill (N) circle (1.5pt);
\fill (O) circle (1.5pt);
\fill (P) circle (1.5pt);
\fill (Q) circle (1.5pt);
%
%
\coordinate[label=left:\tiny{$\beta_1(1)$}] (AA) at (-2+7,2);
\coordinate[label=right:\tiny{$$}] (BB) at (3.1+7,2);
\coordinate[label=right:\tiny{$$}] (CC) at (3.1+7,-0.9);
\coordinate[label=left:\tiny{$$}] (DD) at (-2+7,-6);
\coordinate[label=above:\tiny{$\beta_2(1)$}] (EE) at (-1.6+7,2);
\coordinate[label=above right:\tiny{$\beta_2(2)$}] (EFEF) at (-1.1+7,2);
\coordinate[label=left:\tiny{$\beta_2(3)$}] (FF) at (-1.1+7,1.4);
\coordinate[label=right:\tiny{$\beta_2(4)$}] (GG) at (-0.8+7,1.4);
\coordinate[label=below:\tiny{$\beta_2(5)$}] (HH) at (-0.8+7,0.6);
\coordinate[label=right:\tiny{$\beta_2(6)$}] (II) at (-0.0+7,0.6);
\coordinate[label=left:\tiny{$$}] (JJ) at (-0.0+7,0.0);
\coordinate[label=below:\tiny{$\beta_2(7)$}] (JJJJ) at (-0.3+7,0.0);
\coordinate[label=right:\tiny{$$}] (M1M1) at (0.15+7,-0.15);
\coordinate[label=right:\tiny{$$}] (MM) at (0.2+7,-0.2);
\coordinate[label=right:\tiny{$$}] (KK) at (0.25+7,-0.25);
\coordinate[label=right:\tiny{$$}] (LL) at (0.3+7,-0.3);
\coordinate[label=right:\tiny{$$}] (M2M2) at (0.35+7,-0.35);

\coordinate[label=above right:\tiny{$$}] (NN) at (2+7,-1);
\coordinate[label=above right:\tiny{$\beta_2(k_2\!-\!1)$}] (NNNN) at (1.75+7,-1);
\coordinate[label=right:\tiny{$$}] (OO) at (1.5+7,-1);
\coordinate[label=above:\tiny{$$}] (OOOO) at (1.35+7,-1);
\coordinate[label=above right:\tiny{$$}] (PP) at (0.5+7,-1);
\coordinate[label=above right:\tiny{$\beta_2(k_2\!-\!2)$}] (PPPP) at (0.5+7,-1);
\coordinate[label=above right:\tiny{$$}] (QQ) at (0.5+7,-0.5);
\coordinate[label=above right:\tiny{$\beta_2(k_2\!-\!3)$}] (QQQQ) at (0.5+7,-0.6);
\coordinate[label=right:\tiny{$\beta_2(k_2) $}] (RR) at (2+7,-2);

\coordinate[label=left:\tiny{$\beta_1(2)$}] (SS) at (-2+7,0.8);
\coordinate[label=above right:\tiny{$$}] (TT) at (-1.6+7,0.8);
\coordinate[label=above right:\tiny{$\beta_1(3)$}] (TTTT) at (-1.65+7,0.7);
\coordinate[label=below:\tiny{$\beta_1(4)$}] (UU) at (-1.6+7,-0.2);
\coordinate[label=above:\tiny{$\beta_1(5)$}] (VV) at (-1+7,-0.2);
\coordinate[label=left:\tiny{$\beta_1(6)$}] (WW) at (-1+7,-0.8);
\coordinate[label=left:\tiny{$$}] (W1W1) at (-0.8+7,-1);
\coordinate[label=left:\tiny{$$}] (X1X1) at (-0.4+7,-1.4);
\coordinate[label=below:\tiny{$\beta_1(k_1\!-\!3)$}] (XX) at (-0.2+7,-1.6);
\coordinate[label=above:\tiny{$\beta_1(k_1\!-\!2)$}] (YY) at (0.7+7,-1.6);
\coordinate[label=below:\tiny{$\beta_1(k_1\!-\!1)$}] (ZZ) at (0.7+7,-2.3);
\coordinate[label=right:\tiny{$$}] (Z1Z1) at (1.7+7,-2.3);
\coordinate[label=right:\tiny{$\beta_1(k_1)$}] (Z2Z2) at (1.6+7,-2.5);

\draw (AA) -- (BB);
\draw (CC) -- (BB);
\draw (CC) -- (DD);
\draw (AA) -- (DD);
\draw[thick] (EE) -- (EFEF);
\draw[thick] (FF) -- (EFEF);
\draw[thick] (FF) -- (GG);
\draw[thick] (HH) -- (GG);
\draw[thick] (HH) -- (II);
\draw[thick] (JJ) -- (II);
\draw[thick] (NN) -- (RR);
\draw[thick] (NN) -- (OO);
\draw[thick] (PP) -- (OO);
\draw[thick] (PP) -- (QQ);
\fill (AA) circle (1.5pt);
\fill (EE) circle (1.5pt);
\fill (EFEF) circle (1.5pt);
\fill (FF) circle (1.5pt);
\fill (GG) circle (1.5pt);
\fill (HH) circle (1.5pt);
\fill (II) circle (1.5pt);
\fill (JJ) circle (1.5pt);
\fill (KK) circle (0.5pt);
\fill (LL) circle (0.5pt);
\fill (MM) circle (0.5pt);
\fill (M1M1) circle (0.5pt);
\fill (M2M2) circle (0.5pt);
\fill (NN) circle (1.5pt);
\fill (PP) circle (1.5pt);
\fill (QQ) circle (1.5pt);
\fill (RR) circle (1.5pt);

\draw[thick] (SS) -- (AA);
\draw[thick] (SS) -- (TT);
\draw[thick] (UU) -- (TT);
\draw[thick] (UU) -- (VV);
\draw[thick] (WW) -- (VV);
\draw[thick, dotted] (W1W1) -- (X1X1);
\draw[thick] (XX) -- (YY);
\draw[thick] (YY) -- (ZZ);
\draw[thick] (Z1Z1) -- (ZZ);

\fill (SS) circle (1.5pt);
\fill (TT) circle (1.5pt);
\fill (UU) circle (1.5pt);
\fill (VV) circle (1.5pt);
\fill (WW) circle (1.5pt);
\fill (XX) circle (1.5pt);
\fill (YY) circle (1.5pt);
\fill (ZZ) circle (1.5pt);
\fill (Z1Z1) circle (1.5pt);
%
%
\end{tikzpicture}
\end{center}
\begin{example}
We list all Dyck paths for $\tt B_4$, $i=3$. We have 
\begin{align*}\mathbf D^{\typ 1}=\Big\{ &\{\al_{1,3},\al_{2,3}, \al_{3,3},\al_{3,4}\} , \{\al_{1,3},\al_{2,3}, \al_{2,4},\al_{3,4}\}, \{\al_{1,3},\al_{1,4}, \al_{2,4},\al_{3,4}\}, \{\al_{1,4},\al_{2,4}, \al_{3,4},\al_{3,5}\},&\\&\{\al_{1,4},\al_{2,4}, \al_{2,5},\al_{3,5}\}, \{\al_{1,4},\al_{1,5}, \al_{2,5},\al_{3,5}\} \Big\}.\end{align*}
\begin{align*}\mathbf D^{\typ 2}=\Big\{ &\{\al_{1,3},\al_{2,3}, \al_{1,4},\al_{2,4},\al_{1,5},\al_{2,5},\al_{1,6},\al_{2,6},\al_{1,7}\},\{\al_{1,3},\al_{2,3},\al_{1,4},\al_{2,4},\al_{1,5},\al_{2,5},\al_{1,6},\al_{2,6},\al_{3,5}\},&\\&\{\al_{1,3},\al_{2,3}, \al_{1,4},\al_{2,4},\al_{1,5},\al_{3,4},\al_{1,6},\al_{2,6},\al_{3,5}\},\{\al_{1,3},\al_{2,3}, \al_{1,4},\al_{3,3},\al_{1,5},\al_{3,4},\al_{1,6},\al_{2,6},\al_{3,5}\}, &\\&\{\al_{1,3},\al_{2,3}, \al_{1,4},\al_{3,3},\al_{1,5},\al_{3,4},\al_{2,5},\al_{2,6},\al_{3,5}\}, \{\al_{1,3},\al_{2,3}, \al_{1,4},\al_{2,4},\al_{1,5},\al_{3,4},\al_{2,5},\al_{2,6},\al_{3,5}\},&\\&\{\al_{1,3},\al_{2,3}, \al_{1,4},\al_{2,4},\al_{3,3},\al_{3,4},\al_{2,5},\al_{2,6},\al_{3,5}\} \Big\}.\end{align*}
\end{example}

The corresponding polytope is defined by

\begin{equation}\label{polyi}P(\mathbf D, m\omega_i)=\Big\{\mathbf s=(s_{\beta})\in \R_+^{|R_i^+|}\mid \forall \mathbf p\in \mathbf D: \sum_{\beta\in \mathbf p} s_{\beta} \leq M_{\mathbf p}(m\omega_i) \Big\},\end{equation}
where we set
$$M_{\mathbf p}(m\omega_i)=\begin{cases}m&\text{ if $\mathbf p \in \mathbf D^{\typ 1}$}\\
m\omega_i(\theta^{\vee})&\text{ if $\mathbf p \in \mathbf D^{\typ 2}$}\end{cases}$$
We consider the polytope $P(\mathbf D,m\omega_i)$ as a subset of $\R_+^{|R^+|}$ by requiring $s_{\beta}=0$ for $\beta\in R^+\backslash R^+_i$.
\begin{rem}
Note that the set $\mathbf D$ is a subset of $\mathcal P(R_i^+)$ and depends therefore on $i$ (unlike as in the $\tt A_n$, $\tt C_n$ and $\tt G_2$ case). We do not expect that there exists a set $\mathbf D^{'}\subset \mathcal P(R^+)$ such that the following holds: for any dominant integral weight $\mu$ there exists non--negative integers $M_{\mathbf p}(\mu)$ ($\mathbf p\in \mathbf D^{'}$) such that the integral points of the corresponding polytope \eqref{polyaundcund} parametrize a basis of $\gr V(\mu)$. We rather expect that there exists a polytope parametrizing a basis of the associated graded space where the coefficients of the describing inequalities might be greater than 1. We will demonstrate this in the $\tt B_3$ case (see Section~\ref{section5}).
\end{rem}
\subsection{} For $\mathbf s\in S(\mathbf D,m\omega_i)$ let $\wt(\mathbf s):=\sum_{\beta\in R_i^+}s_{\beta}\beta$ and $$S(\mathbf D,m\omega_i)^{\mu}=\{\mathbf s\in S(\mathbf D,m\omega_i)\mid m\omega_i-\wt(\mathbf s)=\mu\}.$$
We make the following conjecture and prove various cases in this paper. We set $\epsilon_i=1$ if $i\leq 2$ and $\epsilon_i=2$ else.
\begin{conj}\label{conjbasistypeB}
Let $\lie g$ be the Lie algebra of type $\tt B_n$ and $1\leq i \leq n$. 
\begin{enumerate}
\item The lattice points $S(\mathbf D, m\omega_i)$ parametrize a basis of $V(m\omega_i)$ and $\gr V(m\omega_i)$ respectively. In particular, 
$$\{\X^{\mathbf s}v_{m\omega_i}\mid \mathbf s \in S(\mathbf D, m\omega_i)\}$$
forms a basis of $\gr V(m\omega_i)$.
\item We have
$$\mathbf I_{m\omega_i}=S(\lie n^-)\Big(\bu(\lie n^+)\circ \spa\big\{x^{\omega_i(\theta^{\vee})m+1}_{-\theta},x^{m+1}_{-\alpha_{1,2n-i}},x_{-\beta}\mid \beta\in R^+\backslash R^+_i\big\}\Big).$$
\item The character and graded $q$-character respectively is given by
$$\cha V(m\omega_i)=\sum_{\mu\in \lie h^{*}}|S(\mathbf D,m\omega_i)^{\mu}|e^{\mu}$$
$$\cha_q \gr V(m\omega_i)=\sum_{\mathbf s\in S(\mathbf D,m\omega_i)}e^{m\omega_i-\wt(\mathbf s)}q^{\sum s_{\beta}}.$$
\item We have an isomorphism of $S(\lie n^-)$--modules for all $\ell\in \Z_+$:
$$\gr V\big((m+\epsilon_i\ell)\omega_i\big)\cong S(\lie n^-)(v_{m\omega_i}\otimes v_{\epsilon_i\ell \omega_i})\subseteq \gr V(m\omega_i)\otimes \gr V(\epsilon_i\ell\omega_i).$$
\end{enumerate}
\end{conj}
\subsection{}\label{induordm}
In this subsection we will reduce the proof of Conjecture~\ref{conjbasistypeB} to a technical lemma (see Lemma~\ref{reductlem}). For this we will need the notion of essential monomials which is due to Vinberg \cite{V05}. We fix an ordered basis $\{x_1,\dots,x_N\}$ of $\mathfrak{n}^-$ and an induced homogeneous monomial order $<$ on the monomials in $\{x_1,\dots,x_N\}$. Let $M$ be any finite--dimensional cyclic $\mathbf U(\mathfrak{n}^-)$--module with cyclic vector $v_M$ and let
$$\X^{\bs}v_M=x_1^{s_1}\dots x_N^{s_N}v_M\in M,$$
where $\bs\in\Z_+^N$ is a multi--exponent.
\begin{defn}
A pair $(M,\bs)$ is called essential if 
$$\X^{\bs}v_M\notin \spa\{\X^{\bq}v_M\mid \bq<\bs\}.$$
\end{defn}
If the pair $(M,\bs)$ is essential, then $\bs$ is called an essential multi--exponent and $\X^{\bs}$ is called an essential monomial in $M$. Note that the set of all essential monomials, denoted by ${\rm es}(M)\subseteq \Z_+^N$, parametrizes a basis of $M$.
\begin{lem}\label{reductlem}
The proof of Conjecture~\ref{conjbasistypeB} can be reduced to the following three statements:
\renewcommand{\theenumi}{\roman{enumi}}%
\begin{enumerate}
\item If $\mathbf{s}\notin S(\mathbf D, m\omega_i)$, then the following is true in $S(\lie n^-)/\mathbf I_{m\omega_i}$
$$\X^{\mathbf s}\in \text{span}\{\X^{\mathbf q}\mid \mathbf q < \mathbf s \}.$$
Hence $\{\X^{\mathbf s} \mid \mathbf s \in S(\mathbf D, m\omega_i)\}$ generates the module $S(\lie n^-)/\mathbf I_{m\omega_i}$. 
\item We have
$$S(\mathbf D,(m+\epsilon_i\ell)\omega_i)=S(\mathbf D,m\omega_i)+S(\mathbf D,\epsilon_i\ell\omega_i).$$
\item We have
$${\rm es}(V(\ell\omega_i))=S(\mathbf D,\ell\omega_i) \mbox{ for $\ell\leq \epsilon_i$.}$$
\end{enumerate}
\begin{proof}
Assume that part (1) of the conjecture holds. Part (3) of the conjecture follows immediately from part (1). Since $\mathbf I_{m\omega_i}v_{m\omega_i}=0$, we have a surjective map 
$$S(\lie n^-)/\mathbf I_{m\omega_i}\longrightarrow \gr V(m\omega_i)$$ and hence part (2) of the conjecture follows with part (1) and (i). It has been shown in  \cite[Proposition 3.7]{FFL2011} (cf. also \cite[Proposition 1.11]{FFL13}) that if $\{\X^{\mathbf s}v_{\lambda}\mid \mathbf s\in S(\mathbf D,\lambda)\}$ is a basis of $\gr V(\lambda)$ and $\{\X^{\mathbf s}v_{\mu}\mid \mathbf s\in S(\mathbf D,\mu)\}$ is a basis of $\gr V(\mu)$, then $\big\{\X^{\mathbf s}(v_{\lambda}\otimes v_{\mu}), \mathbf s\in S(\mathbf D,\lambda)+S(\mathbf D,\mu)\big\}$ is a linearly independent subset of $\gr V(\lambda)\otimes \gr V(\mu)$ and therefore also a linearly independent subset of $V(\lambda)\otimes V(\mu)$. Since we have a surjective map (cf. \cite[Lemma 6.1]{FFL13})
$$S(\lie n^-)/\mathbf I_{(m+\epsilon_i\ell)\omega_i}\cong\gr V((m+\epsilon_i\ell)\omega_i)\longrightarrow S(\lie n^-)(v_{m\omega_i}\otimes v_{\epsilon_i\ell\omega_i})\subseteq \gr V(m\omega_i)\otimes \gr V(\epsilon_i\ell\omega_i),$$
part (4) follows from part (1) and (ii). So it remains to prove that part (1) follows from (i)--(iii). If $m\leq \epsilon_i$ this follows from (iii), so let $m>\epsilon_i$. By induction we can suppose that $S(\mathbf D,(m-\epsilon_i)\omega_i)$ parametrizes a basis of $\gr V((m-\epsilon_i)\omega_i)$ and by (i) and (iii) that $S(\mathbf D,\epsilon_i\omega_i)$ parametrizes a basis of $\gr V(\epsilon_i\omega_i)$. Thus, together with (ii), we obtain similar as above that $\big\{\X^{\mathbf s}(v_{(m-\epsilon_i)\omega_i}\otimes v_{\epsilon_i\omega_i}), \mathbf s\in S(\mathbf D,m\omega_i)\big\}$ is a linearly independent subset of $V((m-\epsilon_i)\omega_i)\otimes  V(\epsilon_i\omega_i)$. Since $V(m\omega_i)\cong \mathbf U(\lie n^-)(v_{(m-\epsilon_i)\omega_i}\otimes v_{\epsilon_i\omega_i})$ and $\dim V(m\omega_i)=\dim \gr V(m \omega_i)$ part (1) follows.
\end{proof}
\end{lem}
Therefore it will be enough to prove the above lemma with respect to a choosen order. The first part of the lemma is proved in full generality in Section~\ref{spanprop1} whereas the second part is proved only for several special cases ($1\leq i\leq 3$ and $n$ arbitrary or $i$ arbitrary and $1\leq n\leq 4$) in Section~\ref{zweitehhh}. In order to prove the third part for these special cases it will be enough to show that ${\rm dim} V(\ell\omega_i)=|S(\mathbf D,\ell\omega_i)| \mbox{ for $\ell\leq \epsilon_i$}$, since the first part already implies ${\rm es}(V(m\omega_i))\subseteq S(\mathbf D,m\omega_i)$ for $m\in \mathbb{Z}_+$. The dimension argument is an easy calculation and will be omitted.
\subsection{Proof of Lemma~\ref{reductlem} (i)}\label{spanprop1}
We choose a total order $\prec$ on $R^+$:
\begin{equation*}
\al_{p,q} \prec \al_{s,t} :\Leftrightarrow q < t \mbox{ or } q = t \mbox{ and } p > s.
\end{equation*}
Interpreted in the Hasse diagram this means that we order the roots from the bottom to the top and from left to right. We extend this order to the induced homogeneous reverse lexicographic order on the monomials in $S(\mathfrak{n}^-)$. We order the set of positive roots $R^+=\{\beta_1,\dots,\beta_N\}$ with respect to $\prec$:
$$\beta_N\prec\beta_{N-1}\prec\dots\prec \beta_1.$$
The definition of the order $\prec$ implies the following. Let $\beta_{\ell} \prec \beta_{p}$ and $\nu \in R^+$, such that $\beta_{\ell} - \nu \in R^+$ and $\beta_{p}- \nu \in R^+$, then
\begin{equation*}\label{order}
\beta_{\ell} - \nu \prec \beta_{p} - \nu.
\end{equation*}
We define differential operators for $\al\in R^{+}$ on $S(\lie n^-)$ by:
\begin{equation*}
\partial_{\alpha} x_{-\beta} :=
\begin{cases}
x_{-\beta + \alpha}, \, &\textrm{if} \, \beta - \alpha \in R^+\\
0, \,& \textrm{else.}
\end{cases}
\end{equation*} The operators satisfy
\begin{equation*}
\partial_\alpha x_{-\beta} = c_{\alpha,\beta} [x_{\alpha}, x_{-\beta}],
\end{equation*}
where $c_{\al,\beta} \in \tC$ are some non--zero constants.
\begin{lem}\label{maxinre}
 Let $\sum_{\br \in \Z_{+}^{N}}c_{\br}\X^{\br} \in S(\mathfrak{n}^-)$ and $\nu \in R^+$. We set $$\bt=\max\big\{\br \mid \partial_{\nu}\X^{\br} \neq 0, c_{\br} \neq 0\big\}.$$ 
Then the maximal monomial in $\sum_{\br \in \Z_{+}^{N}}c_{\br}\partial_{\nu}\X^{\br}$ is a summand of $\partial_{\nu} \X^{\bt}$.
\end{lem}
\begin{proof} 
We express $\partial_{\nu}\X^{\bt}$ as a sum of monomials and let $\X^{\overline{\bt}}$ be the maximal element appearing in this expression. From the definition of the differential operators it is clear that $$\overline{t}_{\beta_{\ell}}=\begin{cases}t_{\beta_{\ell}},&\text{ if $\ell\neq j_{\bt}$, $\beta_{\ell}\neq \beta_{j_{\bt}}-\nu$}\\ t_{\beta_{\ell}}-1,&\text{ if $\ell=j_{\bt}$}\\ t_{\beta_{\ell}}+1,&\text{ if $\beta_{\ell}=\beta_{j_{\bt}}-\nu$} \end{cases},\
\mbox{ where } \beta_{j_{\bt}}=\max_{1\leq k\leq N}\big\{\beta_k\mid \partial_{\nu}x_{-\beta_k}\neq 0, t_{\beta_k}\neq 0\big\}.$$
With other words, $\X^{\overline{\bt}}$ is a scalar multiple of $$\prod_{\ell\neq j_{\bt}}x_{-\beta_\ell}^{t_{\beta_p}}x_{-\beta_{j_{\bt}}}^{t_{\beta_{j_{\bt}}}-1}x_{-\beta_{j_{\bt}}+\nu}.$$ Moreover, let $\X^{\overline{\br}}$ be any monomial with $c_{\br}\neq 0$ and $\partial_{\nu}\X^{\br}\neq 0$. Similar as above we denote by $\X^{\overline{\br}}$ the maximal element which appears as a summand of $\partial_{\nu}\X^{\br}$. 
In the rest of the proof we shall verify that $\overline{\bt}\succ\overline{\br}$. Since $\bt \succ \br$ this follows immediately if $j_{\bt}\leq j_{\br}$. So suppose that $j_{\bt}> j_{\br}$ and $\overline{\bt}\prec\overline{\br}$. This is only possible if $r_{\beta_{j_{\br}}}-1 < t_{\beta_{j_{\br}}}$ and $t_{\beta_p}=r_{\beta_p}$ for $1\leq p<j_{\br}$. Therefore we can deduce from $\bt \succ \br$ that $r_{\beta_{j_{\br}}}=t_{\beta_{j_{\br}}}$. It follows $t_{\beta_{j_{\br}}}\neq 0$, $\partial_{\nu}x_{-\beta_{j_{\br}}}\neq 0$ and $\beta_{j_{\bt}}\prec \beta_{j_{\br}}$, which is a contradiction to the choice of $\beta_{j_{\bt}}$ .
\end{proof}
The  proof of Lemma~\ref{reductlem} (i) proceeds as follows. 
Let $\X^{\mathbf s}$, $\mathbf s\notin S(\mathbf D,m\omega_i)$ be a monomial in $S(\lie n^-)/\mathbf I_{m\omega_i}$. Then there exists a Dyck path $\bp$ such that $\sum_{\beta}s_{\beta}>M_{\bp}(m\omega_i)$. We define another multi--exponent $\br=(r_{\beta})$ by $r_\beta=s_\beta$ if $\beta\in \bp$ and $r_{\beta}=0$ otherwise. Since we have a monomial order it will be enough to prove that $\X^{\br}$ can be written as a sum of smaller monomials. Hence the following proposition proves Lemma~\ref{reductlem} (i).
\begin{prop}\label{spanprop2}
Let $\bp \in \mathbf D$ and $\mathbf s \in \Z_{+}^{|R^+_i|}$ be a multi--exponent supported on $\bp$, i.e. $s_{\beta} = 0$ for $\beta \notin \bp$. Suppose $\sum_{\beta \in \bp}{s_\beta} > M_{\mathbf p}(m\omega_i)$. Then there exist constants $c_{\bt} \in \tC, \bt \in \Z_{+}^{|R^+_i|}$ such that
\begin{equation*}
\X^{\mathbf s} + \sum_{\bt \prec\, \bs} c_{\bt} \X^{\bt} \in \bf I_{\lambda}.
\end{equation*}
\end{prop}
\begin{proof}
First we assume that $\bp = \{\beta(1),\dots,\beta(k)\} \in \mathbf D_2^{\typ 1}$. Note that the ideal $\bf I_{\lambda}$ is stable under the action of the differential operators and $x_{-\al_{1,2n - i}}^{s_{\beta(1)} + \dots + s_{\beta(k)}} \in \mathbf I_\lambda.$ In the following we write simply $x_{p,q} := x_{-\al_{p,q}}$ and $s_{p,q} := s_{\al_{p,q}}$ and rewrite the monomial $x_{-\beta(1)} \cdots x_{-\beta(k)}$ as follows. We can choose a sequence of integers 
$$1 = p_0 \leq p_1 < p_2 < \dots < p_{r-1} < p_r = i<i+1 = q_0 < q_1 < q_2 < \dots < q_{r-1} \leq q_r = 2n-i$$
$\mbox{ with } 1 \leq p_\ell \leq q_\ell \leq n \mbox{ or } 1 \leq p_\ell \leq 2n - q_\ell <n$ for all $0\leq \ell \leq r$ such that
\begin{equation*}
x_{-\beta(1)} \cdots x_{-\beta(k)}=x_{1,i+1} \cdots x_{p_1,i+1} x_{p_1 ,i+2} \cdots x_{p_1, q_1} x_{p_1 + 1, q_1} \cdots x_{p_2, q_1} x_{p_2 ,q_1 + 1} \cdots x_{p_2 ,q_2} \cdots x_{p_r,q_r}.
\end{equation*}
See the picture below for an illustration:
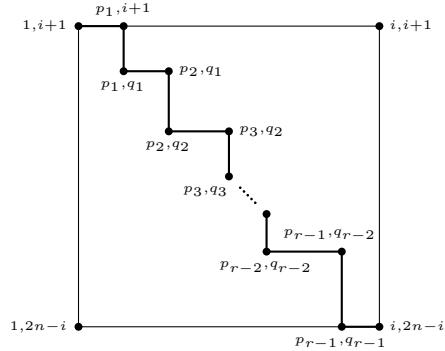
\begin{center}
\begin{tikzpicture}
\coordinate[label=left:\tiny{$_{1,i+1}$}] (A) at (-2,2);
\coordinate[label=right:\tiny{$_{i,i+1}$}] (B) at (2,2);
\coordinate[label=right:\tiny{$_{i,2n-i}$}] (C) at (2,-2);
\coordinate[label=left:\tiny{$_{1,2n-i}$}] (D) at (-2,-2);
\coordinate[label=above:\tiny{$_{p_1,i+1}$}] (E) at (-1.4,2);
\coordinate[label=below:\tiny{$_{p_1,q_1}$}] (F) at (-1.4,1.4);
\coordinate[label=right:\tiny{$_{p_2,q_1}$}] (G) at (-0.8,1.4);
\coordinate[label=below:\tiny{$_{p_2,q_2}$}] (H) at (-0.8,0.6);
\coordinate[label=right:\tiny{$_{p_3,q_2}$}] (I) at (-0.0,0.6);
\coordinate[label=left:\tiny{$$}] (J) at (-0.0,0.0);
\coordinate[label=below:\tiny{$_{p_3,q_3}$}] (JJ) at (-0.3,0.0);
\coordinate[label=right:\tiny{$$}] (M1) at (0.15,-0.15);
\coordinate[label=right:\tiny{$$}] (M) at (0.2,-0.2);
\coordinate[label=right:\tiny{$$}] (K) at (0.25,-0.25);
\coordinate[label=right:\tiny{$$}] (L) at (0.3,-0.3);
\coordinate[label=right:\tiny{$$}] (M2) at (0.35,-0.35);

\coordinate[label=below:\tiny{$_{p_{r-1},q_{r-1}}$}] (N) at (1.5,-2);
\coordinate[label=above:\tiny{$$}] (O) at (1.5,-1);
\coordinate[label=above:\tiny{$_{p_{r-1},q_{r-2}}$}] (OO) at (1.35,-1);
\coordinate[label=below:\tiny{$_{p_{r-2},q_{r-2}}$}] (P) at (0.5,-1);
\coordinate[label=below:\tiny{$$}] (Q) at (0.5,-0.5);
\draw (A) -- (B);
\draw (C) -- (B);
\draw (C) -- (D);
\draw (A) -- (D);
\draw[thick] (E) -- (A);
\draw[thick] (E) -- (F);
\draw[thick] (F) -- (G);
\draw[thick] (H) -- (G);
\draw[thick] (H) -- (I);
\draw[thick] (J) -- (I);
\draw[thick] (N) -- (C);
\draw[thick] (N) -- (O);
\draw[thick] (P) -- (O);
\draw[thick] (P) -- (Q);
\fill (A) circle (1.5pt);
\fill (B) circle (1.5pt);
\fill (C) circle (1.5pt);
\fill (D) circle (1.5pt);
\fill (E) circle (1.5pt);
\fill (F) circle (1.5pt);
\fill (G) circle (1.5pt);
\fill (H) circle (1.5pt);
\fill (I) circle (1.5pt);
\fill (J) circle (1.5pt);
\fill (K) circle (0.5pt);
\fill (L) circle (0.5pt);
\fill (M) circle (0.5pt);
\fill (M1) circle (0.5pt);
\fill (M2) circle (0.5pt);
\fill (N) circle (1.5pt);
\fill (O) circle (1.5pt);
\fill (P) circle (1.5pt);
\fill (Q) circle (1.5pt);
%
%
\end{tikzpicture}
\end{center}
For $0\leq \ell \leq r$ we define $s_{p_{\ell}} := s_{p_{\ell},q_{{\ell}-1} +1} + \dots + s_{p_{\ell},q_{\ell}} + s_{p_{\ell} + 1,q_{\ell}} + \dots + s_{p_{{\ell}+1},q_{\ell}}$ and $|\mathbf s| := s_{\beta(1)}+\cdots+s_{\beta(k)}$. Then
\begin{equation*}
\partial_{\al_{1,p_1-1}}^{^{s_{p_1}}} x_{1,2n-i}^{^{|\mathbf s|}} = x_{1,2n-i}^{^{|\mathbf s| - s_{p_1}}} x_{p_1,2n-i}^{^{s_{p_1}}}\in \bf I_{\lambda}.
\end{equation*}
Since $\partial_{\al_{1,l}} x_{t, 2n-i} = 0$ for $1 < t \leq l<i$ we conclude with $p_1 < p_2 < \dots < p_r$:
\begin{equation*}
\partial_{\al_1,p_r -1}^{s_{p_r}} \cdots \partial_{\al_1,p_2 -1}^{s_{p_2}} \partial_{\al_1,p_1 -1}^{s_{p_1}}   x_{1,2n-i}^{|\mathbf s|} =
x_{1,2n-i}^{|\mathbf s| - \sum_{t=1}^r s_{p_t}} x_{p_1,2n-i}^{s_{p_1}} x_{p_2,2n-i}^{s_{p_2}}  \cdots  x_{p_r,2n-i}^{s_{p_r}} \in \bf I_{\lambda}.
\end{equation*}
Note that the operator $\partial_{\al_{i+1,2n-(i + 1)}}$ acts non--trivially on each $x_{p_j,2n-i}$. The choice of the order implies that the largest monomial in
\begin{equation}\label{firstsum}
\partial_{\al_{i+1,2n-(i + 1)}}^{s_{1,i+1}+ \dots + s_{p_1,i+1}} x_{1,2n-i}^{|\mathbf s| - \sum_{t=1}^r s_{p_t}} x_{p_1,2n-i}^{s_{p_1}} x_{p_2,2n-i}^{s_{p_2}}  \dots  x_{p_r,2n-i}^{s_{p_r}}
\end{equation}
is obtained by acting with $\partial_{\al_{i+1,2n-(i + 1)}}$ only on the the largest element $x_{1,2n-i}$. So the largest monomial in $\eqref{firstsum}$ with respect to $\prec$ is 
\begin{equation}\label{firstelement}
x_{1,i+1}^{s_{1,i+1}+ \dots + s_{p_1,i+1}} x_{p_1,2n-i}^{s_{p_1}} x_{p_2,2n-i}^{s_{p_2}}  \dots  x_{p_r,2n-i}^{s_{p_r}}.
\end{equation}
Each of the operators $\partial_{\al_{p_1-1, p_1-1}},  \dots, \partial_{\al_{2,2}}, \partial_{\al_{1,1}}$ act trivially on each $x_{p_j,2n-i}$. Since
\begin{equation*}
\partial_{\al_{p_1-1, p_1-1}}^{s_{p_1,i+1}}  \dots \partial_{\al_{2,2}}^{s_{3,i+1}+ \dots + s_{p_1,i+1}}  \partial_{\al_{1,1}}^{s_{2,i+1}+ \dots + s_{p_1,i+1}}  x_{1,i+1}^{s_{1,i+1}+ \dots + s_{p_1,i+1}} = x_{1,i+1}^{s_{1,i+1}} \dots x_{p_1,i+1}^{s_{p_1,i+1}}
\end{equation*}
we obtain by acting with these operators on $\eqref{firstelement}$  that 
\begin{equation}\label{imnachhinein}
x_{1,i+1}^{s_{1,i+1}} \dots x_{p_1,i+1}^{s_{p_1,i+1}} x_{p_1,2n-i}^{s_{p_1}} x_{p_2,2n-i}^{s_{p_2}}  \dots  x_{p_r,2n-i}^{s_{p_r}} + \sum \mathrm{smaller \, monomials} \in \bf I_{\lambda}.
\end{equation}
In the next step we act with the operators $\partial_{\al_{i+1,2n - q_{1}}}, \partial_{\al_{i+1,2n - q_{1}+1}},\dots,\partial_{\al_{i+1,2n - (q_{0}+1)}}$ on $x_{p_1,2n-i}$ and obtain with Lemma~\ref{maxinre}:
\begin{align}\label{stern2}
\partial_{\al_{i+1,2n - q_{1}}}^{s_{p_1}-(s_{p_1,q_1-1}+\cdots+{s_{p_1,q_0+1}})}&\partial_{\al_{i+1,2n - q_{1}+1}}^{s_{p_1,q_1-1}} \dots\partial_{\al_{i+1,2n - (q_{0}+1)}}^{s_{p_1,q_0+1}}x_{p_1,2n-i}^{s_{p_1}}&\\&
\notag=x_{p_1,q_1}^{s_{p_1}-(s_{p_1,q_1-1}+\cdots+{s_{p_1,q_0+1}})}x_{p_1,q_1-1}^{s_{p_1,q_1-1}}\cdots x_{p_1,q_0+1}^{s_{p_1,q_0+1}}+\sum \mathrm{smaller \, monomials} 
\end{align}
Since $x_{p_1,2n-i}$ is the maximal element with respect to $\prec$ among the factors in the leading term of $\eqref{imnachhinein}$ we get by combining Lemma~\ref{maxinre} and $\eqref{stern2}$
\begin{align}\label{secondelement4}
 x_{1,i+1}^{s_{1,i+1}} &\dots x_{p_1,i+1}^{s_{p_1,i+1}} x_{p_1,q_1}^{\sum_{\ell=p_1}^{p_2} s_{\ell,q_1}}x_{p_1,q_1 - 1}^{s_{p_1,q_1 - 1}} \dots x_{p_1,q_0 + 1}^{s_{p_1,q_0 + 1}}x_{p_2,2n-i}^{s_{p_2}}  \dots  x_{p_r,2n-i}^{s_{p_r}}+ \sum \mathrm{smaller \, monomials}\in \bf I_{\lambda}.
\end{align}
Now we act with the operators $\partial_{\al_{p_2 - 1,p_2 -1}},\dots \partial_{\al_{p_1 +1,p_1 + 1}},\partial_{\al_{p_1,p_1}}$:
\begin{equation}
\label{thirdsum}
\begin{aligned}
 \partial_{\al_{p_2 - 1,p_2 -1}}^{s_{p_2,q_1}} \dots \partial_{\al_{p_1 +1,p_1 + 1}}^{s_{p_1 + 2,q_1} + \dots + s_{p_2,q_1}} \partial_{\al_{p_1,p_1}}^{s_{p_1 + 1,q_1} + s_{p_1 + 2,q_1} + \dots + s_{p_2,q_1}} x_{p_1,q_1}^{s_{p_1,q_1} + s_{p_1 + 1,q_1} + \dots + s_{p_2,q_1}}=\\
x_{p_1,q_1}^{s_{p_1,q_1}} x_{p_1 + 1,q_1}^{s_{p_1 + 1,q_1}} \dots x_{p_2,q_1}^{s_{p_2,q_1}}.
\end{aligned}
\end{equation}
Since $\partial_{\al_{p_2 - 1,p_2 -1}},\dots \partial_{\al_{p_1 +1,p_1 + 1}},\partial_{\al_{p_1,p_1}}$ act trivially on each $x_{pj,2n-1}$ and $x_{p_1,q_1}$ is the largest element with respect to $\prec$ among the remaining factors in the leading term of $\eqref{secondelement4}$ we get by combining $\eqref{secondelement4}$ and $\eqref{thirdsum}$ that the following element is the sum of strictly smaller monomials in $S(\mathfrak{n}^-)/\bf I_{\lambda}$:
\begin{equation*}
\label{thirdelement}
x_{1,i+1}^{s_{1,i+1}} \dots x_{p_1,i+1}^{s_{p_1,i+1}} x_{p_1,q_1}^{s_{p_1,q_1}} x_{p_1,q_1 - 1}^{s_{p_1,q_1 - 1}}x_{p_1,q_1 - 2}^{s_{p_1,q_1 - 2}} \dots x_{p_1,q_0 + 1}^{s_{p_1,q_0 + 1}} x_{p_1 + 1,q_1}^{s_{p_1 + 1,q_1}} \dots x_{p_2,q_1}^{s_{p_2,q_1}} x_{p_2,2n-i}^{s_{p_2}}  \dots  x_{p_r,2n-i}^{s_{p_r}} 
\end{equation*}
If we repeat the above steps with $x_{p_2,2n-i}^{s_{p_2}}  \dots  x_{p_r,2n-i}^{s_{p_r}}$ we can deduce the proposition for $\bp \in \mathbf D_2^{\mathrm{type \, 1}}$. Now suppose that $\bp \in \mathbf D_1^{\mathrm{type \, 1}}$ is of the form $$\bp =\{\al_{1,i},\al_{2,i} \dots \al_{\ell,i}, \al_{\ell,i+1}, \dots,\al_{r,i+1},\al_{r,i+2},\dots \al_{i,2n-i-1}\}.$$ We shall construct another Dyck path as follows. We set $\bq=\{\al_{\ell,i+1}, \dots,\al_{r,i+1},\al_{r,i+2},\dots \al_{i,2n-i-1}\}$. Then it is easy to see that we can find an element $\widetilde{\bq}\in \mathcal{P}(R_i^+)$ such that the path $\overline{\bq}:=\bq \cup \widetilde{\bq}\in \mathbf D_2^{\typ 1}$. We define a multi--exponent $s(\overline{\bq})$ by
$$s(\overline{\bq})_{\beta}=s_{\beta},\ \text{if $\beta\in \bq$},\ s(\overline{\bq})_{\alpha_{1,i+1}}=s_{\alpha_{1,i}}+\cdots+s_{\alpha_{\ell,i}},\ \text{ and else }s(\overline{\bq})_{\beta}=0.$$
By our previous calculations we get
\begin{equation}\label{haha2}\X^{\bs(\overline{\bq})}+\sum_{\bt\prec s(\overline{\bq})}c_{\bt}\X^{\bt}\in \mathbf I_{\lambda}.\end{equation}
Note that each operator $\partial_{\al_{1,1}},\dots,\partial_{\al_{\ell-1,\ell-1}}$ acts trivially on $x_{\beta}$ for all $\beta\in\bq$ and $\partial_{\al_{i+1,i+1}}$ acts trivially on $x_{\beta}$ for all $\beta\in \bq\backslash \{\al_{\ell+1,i+1},\dots \al_{r,i+1}\}$. Since $x_{1,i+1}\succ x_{j,i+1}$ for all $\ell+1\leq j \leq r$ the maximal element when acting with $\partial_{\al_{i+1,i+1}}$ on $\eqref{haha2}$ is obtained by acting with $\partial_{\al_{i+1,i+1}}$ on $x_{1,i+1}$. We have
\begin{equation}\label{endvor}\partial_{\al_{i+1,i+1}}^{s_{1,i}+\cdots+s_{\ell,i}}\X^{s(\overline{\bq})}+=x_{1,i}^{s_{1,i}+\cdots+s_{\ell,i}}\X^{s(\bq)}+\sum \mathrm{smaller \, monomials}\in \mathbf I_{\lambda},\end{equation}
where $s(\bq)$ is the multi--exponent defined by $s(\bq)_{\beta}=s_{\beta}$ if $\beta\in \bq$ and $s(\bq)_{\beta}=0$ otherwise.
In the last step we act with $\partial_{\al_{\ell-1,\ell-1}}^{s_{\ell,i}}\partial_{\al_{\ell-2,\ell-2}}^{s_{\ell-1,i}+s_{\ell,i}}\cdots\partial_{\al_{1,1}}^{s_{2,i}+\cdots+s_{\ell,i}}$ on $\eqref{endvor}$ and get
$$\X^{\bs}+\sum_{\bt\prec \bs}c_{\bt}\X^{\bt}\in \mathbf I_{\lambda}.$$
Now we assume that $\bp \in \bf D^{\mathrm{type \, 2}}$, which means that $\bp$ can be written as a union $\bp = \bp_1 \cup \bp_2$ with $\mathbf p_1=\{\beta_1(1),\dots,\beta_1(k_1)\}$ and $\mathbf p_2=\{\beta_2(1),\dots,\beta_2(k_2)\}$ such that $\beta_1(k_1)=\al_{j-1,2n-j+1}$ and  $\beta_2(k_2)=\al_{j,2n-j}$.
We have 
\begin{equation}\label{benn34}
x_{{1,2n - 1}}^{s_{\beta_1(1)} + \dots + s_{\beta_1(k_1)} + s_{\beta_2(1)} + \dots + s_{\beta_2(k_2)}} \in \mathbf I_\lambda.\end{equation}
We will prove the statement of the proposition by upward induction on $j\in\{2,\dots,i\}$. If $j=2$, we have
$$\bp_1=\{\al_{1,i},\al_{1,i+1},\dots,\al_{1,2n-1}\} \mbox{ and } \bp_2=\{\al_{2,i},\al_{2,i+1},\dots,\al_{2,2n-2}\}$$
and therefore by acting on \eqref{benn34} we get
\begin{align*}
\partial_{\al_{1,2n-i}}^{s_{2,i}}\cdots\partial_{\al_{1,3}}^{s_{2,2n-3}}&\partial_{\al_{1,2}}^{s_{2,2n-2}}\partial_{\al_{2,2n-i}}^{s_{1,i}}\cdots\partial_{\al_{2,3}}^{s_{1,2n-2}}\partial_{\al_{2,2}}^{s_{1,2n-2}}x_{{1,2n - 1}}^{s_{\beta_1(1)} + \dots + s_{\beta_1(k_1)} + s_{\beta_2(1)} + \dots + s_{\beta_2(k_2)}}=
&\\& = x_{1,2n-1}^{s_{1,2n-1}}\cdots x_{1,i+1}^{s_{1,i+1}}x_{1,i}^{s_{1,i}}x_{2,2n-2}^{s_{2,2n-i-1}}\cdots x_{2,i+1}^{s_{2,i+1}} x_{2,i}^{s_{2,i}}+ \sum \mathrm{smaller \, monomials}\in \mathbf I_\lambda
\end{align*}
and the induction begins. As before we rewrite the Dyck path as follows:
\begin{flalign*}
&x_{-\beta_1(1)} x_{-\beta_1(2)} \cdots x_{-\beta_1(k_1)} = x_{1,i} x_{1,i+1} \cdots x_{b_1,c_1} x_{b_1 + 1 ,c_1} \cdots x_{b_2, c_1} x_{b_2, c_1+1} \cdots x_{b_2, c_2} \dots x_{b_r,c_r}&\\&
x_{-\beta_2(1)} x_{-\beta_2(2)} \cdots x_{-\beta_2(k_2)}  = x_{2,i} x_{3,i}\cdots x_{p_1,i} x_{p_1,i+1} \cdots x_{p_1,q_1} x_{p_1 + 1 ,q_1} \cdots x_{p_2, q_1} x_{p_2, q_1+1} \cdots x_{p_2, q_2} \cdots x_{p_t,q_t}
\end{flalign*}
where 
$$1 = b_0 = b_1 < b_2 < \cdots < b_{r-1} \leq b_r = j-1,\ i = c_0 < c_1 < c_2 < \cdots < c_{r-1} \leq c_r = 2n-j+1,$$ 
$$2 = p_0 \leq p_1 < p_2 < \dots < p_{r-1} \leq p_t =j \mbox{ and } i = q_0 < q_1 < q_2 < \dots < q_{t-1} \leq q_t = 2n-j.$$
For a pictorial illustration see the picture below:
\begin{center}
\begin{tikzpicture}
\coordinate[label=left:\tiny{$_{1,i}$}] (A) at (-2,2);
\coordinate[label=right:\tiny{$_{i,i}$}] (B) at (3.1,2);
\coordinate[label=right:\tiny{$_{i,2n-i}$}] (C) at (3.1,-0.9);
\coordinate[label=left:\tiny{$_{1,2n-1}$}] (D) at (-2,-6);
\coordinate[label=above:\tiny{$_{2,i}$}] (E) at (-1.6,2);
\coordinate[label=above right:\tiny{$_{p_1,i}$}] (EF) at (-1.1,2);
\coordinate[label=below:\tiny{$_{p_1,q_1}$}] (F) at (-1.1,1.4);
\coordinate[label=right:\tiny{$_{p_2,q_1}$}] (G) at (-0.8,1.4);
\coordinate[label=below:\tiny{$_{p_2,q_2}$}] (H) at (-0.8,0.6);
\coordinate[label=right:\tiny{$_{p_3,q_2}$}] (I) at (-0.0,0.6);
\coordinate[label=left:\tiny{$$}] (J) at (-0.0,0.0);
\coordinate[label=below:\tiny{$_{p_3,q_3}$}] (JJ) at (-0.3,0.0);
\coordinate[label=right:\tiny{$$}] (M1) at (0.15,-0.15);
\coordinate[label=right:\tiny{$$}] (M) at (0.2,-0.2);
\coordinate[label=right:\tiny{$$}] (K) at (0.25,-0.25);
\coordinate[label=right:\tiny{$$}] (L) at (0.3,-0.3);
\coordinate[label=right:\tiny{$$}] (M2) at (0.35,-0.35);

\coordinate[label=above right:\tiny{$$}] (N) at (2,-1);
\coordinate[label=above right:\tiny{$_{p_{t},q_{t-1}}$}] (NN) at (1.9,-1.1);
\coordinate[label=right:\tiny{$$}] (O) at (1.5,-1);
\coordinate[label=above:\tiny{$$}] (OO) at (1.35,-1);
\coordinate[label=above right:\tiny{$$}] (P) at (0.5,-1);
\coordinate[label=above right:\tiny{$_{p_{t-1},q_{t-1}}$}] (PP) at (0.5,-1);
\coordinate[label=above right:\tiny{$$}] (Q) at (0.5,-0.5);
\coordinate[label=above right:\tiny{$_{p_{t-1},q_{t-2}}$}] (QQ) at (0.5,-0.6);
\coordinate[label=right:\tiny{$_{p_{t},q_{t}} = {}_{j,2n-j}$}] (R) at (2,-2);

\coordinate[label=left:\tiny{$_{b_1,c_1}$}] (S) at (-2,0.8);
\coordinate[label=above right:\tiny{$$}] (T) at (-1.6,0.8);
\coordinate[label=above right:\tiny{$_{b_2,c_1}$}] (TT) at (-1.65,0.7);
\coordinate[label=below:\tiny{$_{b_2,c_2}$}] (U) at (-1.6,-0.2);
\coordinate[label=above:\tiny{$_{b_3,c_2}$}] (V) at (-1,-0.2);
\coordinate[label=left:\tiny{$_{b_3,c_3}$}] (W) at (-1,-0.8);
\coordinate[label=left:\tiny{$$}] (W1) at (-0.8,-1);
\coordinate[label=left:\tiny{$$}] (X1) at (-0.4,-1.4);
\coordinate[label=below:\tiny{$_{b_{r-2},c_{r-1}}$}] (X) at (-0.2,-1.6);
\coordinate[label=right:\tiny{$\!_{b_{r-1},c_{r-1}}$}] (Y) at (0.7,-1.6);
\coordinate[label=below:\tiny{$_{b_{r-1},c_r}$}] (Z) at (0.7,-2.3);
\coordinate[label=right:\tiny{$$}] (Z1) at (1.7,-2.3);
\coordinate[label=right:\tiny{$_{b_r,c_r} = {}_{j-1,2n-j+1}$}] (Z2) at (1.6,-2.5);

\draw (A) -- (B);
\draw (C) -- (B);
\draw (C) -- (D);
\draw (A) -- (D);
\draw[thick] (E) -- (EF);
\draw[thick] (F) -- (EF);
\draw[thick] (F) -- (G);
\draw[thick] (H) -- (G);
\draw[thick] (H) -- (I);
\draw[thick] (J) -- (I);
\draw[thick] (N) -- (R);
\draw[thick] (N) -- (O);
\draw[thick] (P) -- (O);
\draw[thick] (P) -- (Q);
\fill (A) circle (1.5pt);
\fill (B) circle (1.5pt);
\fill (C) circle (1.5pt);
\fill (D) circle (1.5pt);
\fill (E) circle (1.5pt);
\fill (EF) circle (1.5pt);
\fill (F) circle (1.5pt);
\fill (G) circle (1.5pt);
\fill (H) circle (1.5pt);
\fill (I) circle (1.5pt);
\fill (J) circle (1.5pt);
\fill (K) circle (0.5pt);
\fill (L) circle (0.5pt);
\fill (M) circle (0.5pt);
\fill (M1) circle (0.5pt);
\fill (M2) circle (0.5pt);
\fill (N) circle (1.5pt);
\fill (P) circle (1.5pt);
\fill (Q) circle (1.5pt);
\fill (R) circle (1.5pt);

\draw[thick] (S) -- (A);
\draw[thick] (S) -- (T);
\draw[thick] (U) -- (T);
\draw[thick] (U) -- (V);
\draw[thick] (W) -- (V);
\draw[thick, dotted] (W1) -- (X1);
\draw[thick] (X) -- (Y);
\draw[thick] (Y) -- (Z);
\draw[thick] (Z1) -- (Z);

\fill (S) circle (1.5pt);
\fill (T) circle (1.5pt);
\fill (U) circle (1.5pt);
\fill (V) circle (1.5pt);
\fill (W) circle (1.5pt);
\fill (X) circle (1.5pt);
\fill (Y) circle (1.5pt);
\fill (Z) circle (1.5pt);
\fill (Z1) circle (1.5pt);
%
%
\end{tikzpicture}
\end{center}
We will construct another path $\overline{\bp}\in \mathbf D^{\typ 2}$. We set 
$$\widetilde{\bp}_1=\bp\backslash\{\al_{p_t,q_{t-1}},\al_{p_t,q_{t-1}+1},\dots,\al_{p_t,q_t}\}.$$ Then it is easy to see that there exists a unique element $\widetilde{\bp}_2\in \mathcal{P}(R_i^+)$ such that $\overline{\bp}=\widetilde{\bp}_1\cup \widetilde{\bp}_2\in \mathbf D^{\typ 2}$ and the roots $\alpha_{j-2,2n-j+2}$, $\alpha_{j-1,2n-j+1}$ appear in $\overline{\bp}$. We define a multi--exponent $s(\overline{\bp})$ by
$$s(\overline{\bp})_{\beta}=s_{\beta},\ \text{ if $\beta\in \widetilde{\bp}_1\backslash\{\al_{b_{r-1},c_r}\}$},\ s(\overline{\bp})_{\al_{b_{r-1},c_r}}=s_{b_{r-1},c_r}+s_{{p_t,q_{t-1}}}+s_{{p_t,q_{t-1}+1}}+\cdots+s_{{p_t,q_{t}}}$$
and $s(\overline{\bp})_{\beta}=0$ otherwise. The induction hypothesis yields 
\begin{equation}\label{hunw}\X^{s(\overline{\bp})}+\sum_{\bt\prec s(\overline{\bp})}c_{\bt}\X^{\bt}\in \mathbf I_{\lambda}.\end{equation}
Now we want to act with suitable operators on $\eqref{hunw}$ such that the leading term is the required monomial $\X^{\bs}$. Since $x_{b_{r-1},c_r}$ is the maximal element in $\X^{s(\overline{\bp})}$ and $\partial_{\al_{b_{r-1},j}},\dots ,\partial_{\al_{b_{r-1},2n-q_{t-1}}}$ act non trivially on $x_{b_{r-1},c_r}$ we obtain the desired property
\begin{align*}\partial_{\al_{b_{r-1},2n-q_{t-1}}}^{s_{p_t,q_{t-1}}}\cdots \partial_{\al_{b_{r-1},j}}^{s_{p_t,q_t}}\X^{s(\overline{\bp})}&+\sum_{\bt\prec s(\overline{\bp})}c_{\bt}\partial_{\al_{b_{r-1},2n-q_{t-1}}}^{s_{p_t,q_{t-1}}}\cdots \partial_{\al_{b_{r-1},j}}^{s_{p_t,q_t}}\X^{\bt}=&\\& =\X^{\bs}+\sum \mathrm{smaller \, monomials}\in \mathbf I_{\lambda}.
\end{align*}
\end{proof}
\subsection{Proof of Lemma~\ref{reductlem} (ii) in various cases}\label{zweitehhh}
In this section we shall prove various cases of Lemma~\ref{reductlem} (ii). Consider the partial order 
$$\al_{j,k}\leq \al_{p,r}\Leftrightarrow  (j\geq p \wedge k\geq r)$$
and suppose we are given a multi--exponent $\bs \in S(\mathbf D,m\omega_i)$. Recall the defintion of $R^+_i(\ell)$ from Section~\ref{dyckfori}. Let $\Ra^{\bs}=\{\beta\in R^+_i(2n-i)\mid s_{\beta}\neq 0\}$
and $\T^{\bs}$ the set of minimal elements in $\Ra^{\bs}$ with respect to $\leq$. We define a multi--exponent $\bt^{\bs}$ by $t_{\beta}=1$,\ if $\beta\in \T^{\bs}$ and $t_{\beta}=0$ otherwise and call it the multi--exponent associated to $\bs$. The following lemma can be deduced from \cite[Proposition 3.7]{FFL11}.
\begin{lem}\label{Anfolg}
Let $\bs\in S(\mathbf D,m\omega_i)$ such that $s_{\beta}\neq 0$ implies $\beta\in R^+_i(2n-i-1)$ \big(resp. $\beta \in (R^+_i\cap R^+_{i+1})(2n-i)$\big). Then we have
$$\bs-\bt^{\bs}\in S(\mathbf D,(m-1)\omega_i).$$
\end{lem}
For a multi--exponent $\bt\in \Z_+^{|R_i^+|}$ define 
$$\supp(\bt)=\{\beta\in R_i^+\mid t_{\beta}\neq 0\},$$
and let 
$$\mathbf T(1)=\{\bt\in \Z_+^{|R_i^+|}\mid t_{\beta}\leq 1, \forall \beta\in R_i^+\}.$$
The following proposition proves Lemma~\ref{reductlem} (ii) for $1\leq i\leq 3$, where the proof for $i=3$ is very technical and is given in the appendix (see Section~\ref{i3}). 
\begin{prop}\label{ghtzu}
Let $1\leq i\leq 3$ and $m\geq \epsilon_i$. Then we have
$$S(\mathbf D,m\omega_i)=S(\mathbf D,(m-\epsilon_i)\omega_i)+S(\mathbf D,\epsilon_i\omega_i).$$
\begin{proof}
The proof for $i=1$ is straightforward since $S(\mathbf D,m\omega_1)$ is determined by two inequalities. 
\textit{Proof for $i=2$:} Suppose $\bs\in S(\mathbf D,m\omega_2)$ and recall that $\mathbf D^{\typ 2}=\{R^+_2\}$. We will construct a multi--exponent $\bt\in S(\mathbf D,\omega_2)$ such that $\bs-\bt\in S(\mathbf D,(m-1)\omega_2)$. We prove the statement by induction on $s_{\theta}$ and start with $s_{\theta}=0$. In this case we note that $\sum_{\beta\in \mathbf p}(s_{\beta}-t_{\beta})\leq m-1$ for all $\mathbf p\in \mathbf D^{\typ 1}$ implies already $\bs-\bt\in S(\mathbf D,(m-1)\omega_2)$. The proof proceeds by several case considerations. For the readers convenience we illustrate each case by means of the Hasse diagram. We make the following convention: a bold dot (resp. square) in the Hasse diagram indicates that the corresponding entry of $\bs$ is zero (resp. non--zero).  \vspace{3pt}

\textbf{Case {1}:}
In this case we suppose $s_{2,2n-2}\neq 0$. 
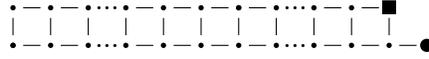
\begin{center}
\begin{tikzpicture}
\node (311) at (0.25,0.5 + 4) {};
  \fill[black] (311) circle (1pt);
  \node (411) at (0.25,1+ 4) {};
  \fill[black] (411) circle (1pt);\node (611) at (0.75,0.5+ 4) {};
  \fill[black] (611) circle (1pt);
  \node (711) at (0.75,1+ 4) {};
  \fill[black] (0.66,4.91) rectangle (0.84,5.09);
	\node (911) at (1.25,0.5+ 4) {};
  \fill[black] (911) circle (2.5pt);\node (1211) at (-0.25,0.5+ 4) {};
  \fill[black] (1211) circle (1pt);
  \node (1311) at (-0.25,1+ 4) {};
 \fill[black] (1311) circle (1pt);\node (2711) at (-0.75,0.5+ 4) {};
  \fill[black] (2711) circle (1pt);\node (2811) at (-0.75,1+ 4) {};
  \fill[black] (2811) circle (1pt);
  \node (2911) at (-0.75,1.5+ 4) {};\draw[black]  (1211) -- (1311);\node (1611) at (-0.5,0.5+ 4) {\textcolor{black}{...}};\node (1611xxx) at (-0.5,1+ 4) {\textcolor{black}{...}};\node (1711) at (-1.25,0.5+ 4) {};
  \fill[black] (1711) circle (1pt);\fill[black] (1811) circle (1pt);
	\node (1811) at (-1.25,1+ 4) {};
  \fill[black] (1811) circle (1pt);\node (3111) at (-1.75,0.5+ 4) {};
  \fill[black] (3111) circle (1pt);
  \node (3211) at (-1.75,1+ 4) {};
  \fill[black] (3211) circle (1pt);\node (17112) at (-2.25,0.5+ 4) {};
  \fill[black] (17112) circle (1pt);
  \node (18112) at (-2.25,1+ 4) {};
  \fill[black] (18112) circle (1pt);\node (27112) at (-2.75,0.5+ 4) {};
  \fill[black] (27112) circle (1pt);
  \node (28112) at (-2.75,1+ 4) {};
  \fill[black] (28112) circle (1pt);\node (3112) at (-3.75,0.5 + 4) {};
  \fill[black] (3112) circle (1pt);
  \node (4112) at (-3.75,1+ 4) {};
  \fill[black] (4112) circle (1pt);\fill[black] (6112) circle (1pt);
  \node (7112) at (-4.25,1+ 4) {};
  \fill[black] (7112) circle (1pt);\node (6112) at (-4.25,0.5+ 4) {};\node (9112) at (-4.25,0.5+ 4) {};
  \fill[black] (9112) circle (1pt);\node (12112) at (-3.25,0.5+ 4) {};
  \fill[black] (12112) circle (1pt);
  \node (13112) at (-3.25,1+ 4) {};
 \fill[black] (13112) circle (1pt);\draw[black]  (12112) -- (13112);\draw[black] (27112) -- (28112);\node (16112) at (-3,1+ 4) {\textcolor{black}{...}};\node (16112xxx) at (-3,0.5+ 4) {\textcolor{black}{...}};
\draw[black] (311)  -- (411);\draw[black] (611)  -- (711);
   \draw[black] (311)  -- (611);
   \draw[black] (411)  -- (711);\draw[black] (611)  -- (911);\draw[black] (311)  -- (1211) ;
   \draw[black] (411)  -- (1311) ;\draw[black] (2711) -- (1711);
   \draw[black] (2811) -- (1811);\draw[black] (1711) -- (1811);
    \draw[black] (2711) -- (2811);\draw[black] (1711) -- (3111);
   \draw[black] (1811) -- (3211);\draw[black] (3111) -- (3211);\draw[black] (17112) -- (18112);\draw[black] (17112) -- (3111);
   \draw[black] (18112) -- (3211);\draw[black] (3112)  -- (4112);\draw[black] (6112)  -- (7112);
   \draw[black] (3112)  -- (6112);
   \draw[black] (4112)  -- (7112);\draw[black] (3112)  -- (12112) ;
   \draw[black] (4112)  -- (13112) ;\draw[black] (27112) -- (17112);
   \draw[black] (28112) -- (18112);
\end{tikzpicture}
\end{center}
If $s_{1,2}=s_{2,2}=0$ the statement follows from Lemma~\ref{Anfolg}. So let $\bt\in\mathbf T(1)$ be the multi--exponent with $\supp(\bt)=\{\al_{2,2n-2}, \al_{k,2}\}$, where $k=\min \{1\leq j\leq 2 \mid s_{j,2}\neq 0\}$. It is easy to see that $\bt\in S(\mathbf D,\omega_2)$ and $\bs-\bt\in S(\mathbf D,(m-1)\omega_2)$.\vspace{3pt}

\textbf{Case {2}:}
In this case we suppose that $s_{2,2n-2}=0$ and $s_{1,2}\neq 0$. 
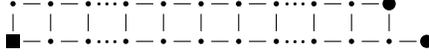
\begin{center}
\begin{tikzpicture}
\node (311) at (0.25,0.5 + 4) {};
  \fill[black] (311) circle (1pt);
  \node (411) at (0.25,1+ 4) {};
  \fill[black] (411) circle (1pt);\node (611) at (0.75,0.5+ 4) {};
  \fill[black] (611) circle (1pt);
  \node (711) at (0.75,1+ 4) {};
  \fill[black] (711) circle (2.5pt);
	\node (911) at (1.25,0.5+ 4) {};
  \fill[black] (911) circle (2.5pt);\node (1211) at (-0.25,0.5+ 4) {};
  \fill[black] (1211) circle (1pt);
  \node (1311) at (-0.25,1+ 4) {};
 \fill[black] (1311) circle (1pt);\node (2711) at (-0.75,0.5+ 4) {};
  \fill[black] (2711) circle (1pt);\node (2811) at (-0.75,1+ 4) {};
  \fill[black] (2811) circle (1pt);
  \node (2911) at (-0.75,1.5+ 4) {};\draw[black]  (1211) -- (1311);\node (1611) at (-0.5,0.5+ 4) {\textcolor{black}{...}};\node (1611xxx) at (-0.5,1+ 4) {\textcolor{black}{...}};\node (1711) at (-1.25,0.5+ 4) {};
  \fill[black] (1711) circle (1pt);\fill[black] (1811) circle (1pt);
	\node (1811) at (-1.25,1+ 4) {};
  \fill[black] (1811) circle (1pt);\node (3111) at (-1.75,0.5+ 4) {};
  \fill[black] (3111) circle (1pt);
  \node (3211) at (-1.75,1+ 4) {};
  \fill[black] (3211) circle (1pt);\node (17112) at (-2.25,0.5+ 4) {};
  \fill[black] (17112) circle (1pt);
  \node (18112) at (-2.25,1+ 4) {};
  \fill[black] (18112) circle (1pt);\node (27112) at (-2.75,0.5+ 4) {};
  \fill[black] (27112) circle (1pt);
  \node (28112) at (-2.75,1+ 4) {};
  \fill[black] (28112) circle (1pt);\node (3112) at (-3.75,0.5 + 4) {};
  \fill[black] (3112) circle (1pt);
  \node (4112) at (-3.75,1+ 4) {};
  \fill[black] (4112) circle (1pt);\fill[black] (6112) circle (1pt);
  \node (7112) at (-4.25,1+ 4) {};
  \fill[black] (7112) circle (1pt);\node (6112) at (-4.25,0.5+ 4) {};
	\node (9112) at (-4.25,0.5+ 4) {};
  \fill[black] (-4.34,4.41) rectangle (-4.16,4.59);
	\node (12112) at (-3.25,0.5+ 4) {};
  \fill[black] (12112) circle (1pt);
  \node (13112) at (-3.25,1+ 4) {};
 \fill[black] (13112) circle (1pt);\draw[black]  (12112) -- (13112);\draw[black] (27112) -- (28112);\node (16112) at (-3,1+ 4) {\textcolor{black}{...}};\node (16112xxx) at (-3,0.5+ 4) {\textcolor{black}{...}};
\draw[black] (311)  -- (411);\draw[black] (611)  -- (711);
   \draw[black] (311)  -- (611);
   \draw[black] (411)  -- (711);\draw[black] (611)  -- (911);\draw[black] (311)  -- (1211) ;
   \draw[black] (411)  -- (1311) ;\draw[black] (2711) -- (1711);
   \draw[black] (2811) -- (1811);\draw[black] (1711) -- (1811);
    \draw[black] (2711) -- (2811);\draw[black] (1711) -- (3111);
   \draw[black] (1811) -- (3211);\draw[black] (3111) -- (3211);\draw[black] (17112) -- (18112);\draw[black] (17112) -- (3111);
   \draw[black] (18112) -- (3211);\draw[black] (3112)  -- (4112);\draw[black] (6112)  -- (7112);
   \draw[black] (3112)  -- (6112);
   \draw[black] (4112)  -- (7112);\draw[black] (3112)  -- (12112) ;
   \draw[black] (4112)  -- (13112) ;\draw[black] (27112) -- (17112);
   \draw[black] (28112) -- (18112);
\end{tikzpicture}
\end{center}

If $s_{1,2n-2}=0$ the statement follows as above from Lemma~\ref{Anfolg}. So let $\bt\in\mathbf T(1)$ be the multi--exponent with $\supp(\bt)=\{\al_{1,2}, \al_{1,2n-2}\}$. It is straightforward to prove that $\bt\in S(\mathbf D,\omega_2)$ and $\bs-\bt\in S(\mathbf D,(m-1)\omega_2)$.\vspace{3pt}

\textbf{Case {3}}: In this case we suppose $s_{1,2}=s_{2,2n-2}=0$. Again with Lemma~\ref{Anfolg} we can assume that $s_{2,2}\neq 0$ and $s_{1,2n-2}\neq 0$.
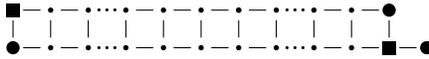
\begin{center}
\begin{tikzpicture}
\node (311) at (0.25,0.5 + 4) {};
  \fill[black] (311) circle (1pt);
  \node (411) at (0.25,1+ 4) {};
  \fill[black] (411) circle (1pt);
	\node (611) at (0.75,0.5+ 4) {};
  \fill[black] (0.66,4.41) rectangle (0.84,4.59);
	
  \node (711) at (0.75,1+ 4) {};
  \fill[black] (711) circle (2.5pt);
	\node (911) at (1.25,0.5+ 4) {};
  \fill[black] (911) circle (2.5pt);
	\node (1211) at (-0.25,0.5+ 4) {};
  \fill[black] (1211) circle (1pt);
  \node (1311) at (-0.25,1+ 4) {};
 \fill[black] (1311) circle (1pt);\node (2711) at (-0.75,0.5+ 4) {};
  \fill[black] (2711) circle (1pt);\node (2811) at (-0.75,1+ 4) {};
  \fill[black] (2811) circle (1pt);
  \node (2911) at (-0.75,1.5+ 4) {};\draw[black]  (1211) -- (1311);\node (1611) at (-0.5,0.5+ 4) {\textcolor{black}{...}};\node (1611xxx) at (-0.5,1+ 4) {\textcolor{black}{...}};\node (1711) at (-1.25,0.5+ 4) {};
  \fill[black] (1711) circle (1pt);\fill[black] (1811) circle (1pt);
	\node (1811) at (-1.25,1+ 4) {};
  \fill[black] (1811) circle (1pt);\node (3111) at (-1.75,0.5+ 4) {};
  \fill[black] (3111) circle (1pt);
  \node (3211) at (-1.75,1+ 4) {};
  \fill[black] (3211) circle (1pt);\node (17112) at (-2.25,0.5+ 4) {};
  \fill[black] (17112) circle (1pt);
  \node (18112) at (-2.25,1+ 4) {};
  \fill[black] (18112) circle (1pt);\node (27112) at (-2.75,0.5+ 4) {};
  \fill[black] (27112) circle (1pt);
  \node (28112) at (-2.75,1+ 4) {};
  \fill[black] (28112) circle (1pt);\node (3112) at (-3.75,0.5 + 4) {};
  \fill[black] (3112) circle (1pt);
  \node (4112) at (-3.75,1+ 4) {};
  \fill[black] (4112) circle (1pt);\fill[black] (6112) circle (1pt);
  \node (7112) at (-4.25,1+ 4) {};
  \fill[black] (-4.34,4.91) rectangle (-4.16,5.09);
	\node (6112) at (-4.25,0.5+ 4) {};
	\node (9112) at (-4.25,0.5+ 4) {};
	\fill[black] (9112) circle (2.5pt);
	\node (12112) at (-3.25,0.5+ 4) {};
  \fill[black] (12112) circle (1pt);
  \node (13112) at (-3.25,1+ 4) {};
 \fill[black] (13112) circle (1pt);\draw[black]  (12112) -- (13112);\draw[black] (27112) -- (28112);\node (16112) at (-3,1+ 4) {\textcolor{black}{...}};\node (16112xxx) at (-3,0.5+ 4) {\textcolor{black}{...}};
\draw[black] (311)  -- (411);\draw[black] (611)  -- (711);
   \draw[black] (311)  -- (611);
   \draw[black] (411)  -- (711);\draw[black] (611)  -- (911);\draw[black] (311)  -- (1211) ;
   \draw[black] (411)  -- (1311) ;\draw[black] (2711) -- (1711);
   \draw[black] (2811) -- (1811);\draw[black] (1711) -- (1811);
    \draw[black] (2711) -- (2811);\draw[black] (1711) -- (3111);
   \draw[black] (1811) -- (3211);\draw[black] (3111) -- (3211);\draw[black] (17112) -- (18112);\draw[black] (17112) -- (3111);
   \draw[black] (18112) -- (3211);\draw[black] (3112)  -- (4112);\draw[black] (6112)  -- (7112);
   \draw[black] (3112)  -- (6112);
   \draw[black] (4112)  -- (7112);\draw[black] (3112)  -- (12112) ;
   \draw[black] (4112)  -- (13112) ;\draw[black] (27112) -- (17112);
   \draw[black] (28112) -- (18112);
\end{tikzpicture}
\end{center}
Let $\bt\in\mathbf T(1)$ be the multi--exponent with $\supp(\bt)=\{\al_{2,2}, \al_{1,k}\}$, where $k=\min \{3\leq j\leq 2n-2 \mid s_{1,j}\neq 0\}$ (see the red dots below). 
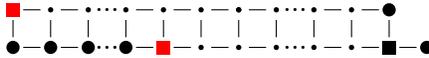
\begin{center}
\begin{tikzpicture}
\node (311) at (0.25,0.5 + 4) {};
  \fill[black] (311) circle (1pt);
  \node (411) at (0.25,1+ 4) {};
  \fill[black] (411) circle (1pt);
	\node (611) at (0.75,0.5+ 4) {};
  \fill[black] (0.66,4.41) rectangle (0.84,4.59);
	
  \node (711) at (0.75,1+ 4) {};
  \fill[black] (711) circle (2.5pt);
	\node (911) at (1.25,0.5+ 4) {};
  \fill[black] (911) circle (2.5pt);
	\node (1211) at (-0.25,0.5+ 4) {};
  \fill[black] (1211) circle (1pt);
  \node (1311) at (-0.25,1+ 4) {};
 \fill[black] (1311) circle (1pt);\node (2711) at (-0.75,0.5+ 4) {};
  \fill[black] (2711) circle (1pt);\node (2811) at (-0.75,1+ 4) {};
  \fill[black] (2811) circle (1pt);
  \node (2911) at (-0.75,1.5+ 4) {};\draw[black]  (1211) -- (1311);\node (1611) at (-0.5,0.5+ 4) {\textcolor{black}{...}};\node (1611xxx) at (-0.5,1+ 4) {\textcolor{black}{...}};\node (1711) at (-1.25,0.5+ 4) {};
  \fill[black] (1711) circle (1pt);\fill[black] (1811) circle (1pt);
	\node (1811) at (-1.25,1+ 4) {};
  \fill[black] (1811) circle (1pt);\node (3111) at (-1.75,0.5+ 4) {};
  \fill[black] (3111) circle (1pt);
  \node (3211) at (-1.75,1+ 4) {};
  \fill[black] (3211) circle (1pt);
	
	\node (17112) at (-2.25,0.5+ 4) {};
  \fill[red] (-2.34,4.41) rectangle (-2.16,4.59);
	
  \node (18112) at (-2.25,1+ 4) {};
  \fill[black] (18112) circle (1pt);\node (27112) at (-2.75,0.5+ 4) {};
  \fill[black] (27112) circle (2.5pt);
  \node (28112) at (-2.75,1+ 4) {};
  \fill[black] (28112) circle (1pt);
	\node (3112) at (-3.75,0.5 + 4) {};
  \fill[black] (3112) circle (2.5pt);
  \node (4112) at (-3.75,1+ 4) {};
  \fill[black] (4112) circle (1pt);\fill[black] (6112) circle (1pt);
  \node (7112) at (-4.25,1+ 4) {};
  \fill[red] (-4.34,4.91) rectangle (-4.16,5.09);
	\node (6112) at (-4.25,0.5+ 4) {};
	\node (9112) at (-4.25,0.5+ 4) {};
	\fill[black] (9112) circle (2.5pt);
	\node (12112) at (-3.25,0.5+ 4) {};
  \fill[black] (12112) circle (2.5pt);
  \node (13112) at (-3.25,1+ 4) {};
 \fill[black] (13112) circle (1pt);\draw[black]  (12112) -- (13112);\draw[black] (27112) -- (28112);\node (16112) at (-3,1+ 4) {\textcolor{black}{...}};\node (16112xxx) at (-3,0.5+ 4) {\textcolor{black}{...}};
\draw[black] (311)  -- (411);\draw[black] (611)  -- (711);
   \draw[black] (311)  -- (611);
   \draw[black] (411)  -- (711);\draw[black] (611)  -- (911);\draw[black] (311)  -- (1211) ;
   \draw[black] (411)  -- (1311) ;\draw[black] (2711) -- (1711);
   \draw[black] (2811) -- (1811);\draw[black] (1711) -- (1811);
    \draw[black] (2711) -- (2811);\draw[black] (1711) -- (3111);
   \draw[black] (1811) -- (3211);\draw[black] (3111) -- (3211);\draw[black] (17112) -- (18112);\draw[black] (17112) -- (3111);
   \draw[black] (18112) -- (3211);\draw[black] (3112)  -- (4112);\draw[black] (6112)  -- (7112);
   \draw[black] (3112)  -- (6112);
   \draw[black] (4112)  -- (7112);\draw[black] (3112)  -- (12112) ;
   \draw[black] (4112)  -- (13112) ;\draw[black] (27112) -- (17112);
   \draw[black] (28112) -- (18112);
\end{tikzpicture}
\end{center}
It follows $\bt\in S(\mathbf D,\omega_2)$. Suppose we are given a Dyck path $\mathbf p\in \mathbf D_1^{\typ 1}$ with $\sum_{\beta\in \mathbf p}(s_{\beta}-t_{\beta})=m$, which is only possible if $t_{\beta}=0$ for all $\beta\in \mathbf p$. 
It follows that $\mathbf p$ is of the form
$$\mathbf p=\{\al_{1,2},\dots,\al_{1,p},\al_{2,p},\dots,\al_{2,2n-3}\},\ \mbox{for some $3\leq p<k$.}$$
Since $s_{1,r}=0$ for all $2\leq r<k$ we get
$$\sum_{\beta\in \mathbf p}s_{\beta}\leq s_{2,3}+\cdots+s_{2,2n-3}\leq (s_{2,2}-1)+s_{2,3}+\cdots+s_{2,2n-3}\leq m-1,$$
which is a contradiction. Similarly, for $\mathbf p\in \mathbf D_2^{\typ 1}$ we get $\sum_{\beta\in \mathbf p}(s_{\beta}-t_{\beta})\leq m-1$.  Hence $\bs-\bt\in S(\mathbf D,(m-1)\omega_2)$ and the induction begins.\vspace{3pt}

Assume that $s_{\theta}\neq 0$ and let $\bs^1$ be the multi--exponent obtained from $\bs$ by replacing $s_{\theta}$ by $s_{\theta}-1$. By induction there exists a multi--exponent $\bt^1\in S(\mathbf D,\omega_2)$ such that $\br^1:=\bs^1-\bt^1 \in S(\mathbf D,(m-1)\omega_2)$. If $\sum_{\beta\in R_2^+}t^1_{\beta}\leq 1$ we set $\bt$ to be the multi--exponent obtained from $\bt^1$ by replacing $t^1_\theta$ by $t^1_\theta+1$. Then we get $\bt\in S(\mathbf D,\omega_2)$ and $\bs-\bt=\br^1$. Otherwise we set $\br$ to be the multi--exponent obtained from $\br^1$ by replacing $r^1_{\theta}$ by $r^1_{\theta}+1$. Since $\sum_{\beta\in R_2^+}t^1_{\beta}= 2$, we get $\sum_{\beta\in R_2^+}r_{\beta}\leq 2m-2$ and therefore 
$$\bs=\br+\bt^1, \text{ and } \bs-\bt^1\in S(\mathbf D,(m-1)\omega_2).$$

\end{proof}
\end{prop}
In order to cover the remaining special cases, we prove Lemma~\ref{reductlem} (ii) for $n=i=4$. Let $\bs\in S(\mathbf D,m\omega_4)$. We prove the Minkowski property by induction on $s_{4,4}+s_{1,7}$. If $s_{4,4}=s_{1,7}=0$, we consider two cases.\vspace{3pt}

\textbf{Case {1}:} In this case we suppose that $s_{1,6},s_{2,5}$ and $s_{3,4}$ are non--zero.

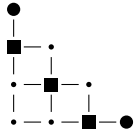
\begin{center}
\begin{tikzpicture}
[rotate=270]


\node (3112) at (-3.75,0.5 + 4) {};
  \fill[black] (3112) circle (1pt);
  \node (4112) at (-3.75,1+ 4) {};
  \fill[black] (-3.66,4.91) rectangle (-3.84,5.09);
  \node (5112) at (-3.75,1.5+ 4) {};
  \fill[black] (5112) circle (1pt);
  \node (6112) at (-4.25,0.5+ 4) {};
  \fill[black] (-4.16,4.41) rectangle (-4.34,4.59);
  \node (7112) at (-4.25,1+ 4) {};
  \fill[black] (7112) circle (1pt);

  \node (9112) at (-4.25,0.5+ 4) {};
  \fill[black] (9112) circle (1pt);
  
  \node (12112) at (-3.25,0.5+ 4) {};
  \fill[black] (12112) circle (1pt);
  \node (13112) at (-3.25,1+ 4) {};
 \fill[black] (13112) circle (1pt);
  \node (14112) at (-3.25,1.5+ 4) {};
  \fill[black] (-3.16,5.41) rectangle (-3.34,5.59);

\node (39112) at (-3.25,2+ 4) {};
 \fill[black] (39112) circle (2.5pt);
 \draw  (13112) -- (14112) -- (39112);
\draw[black]  (12112) -- (13112);

\node (9112) at (-4.75,0.5+ 4) {};
 \fill[black] (9112) circle (2.5pt);

         \draw[black] (3112)  -- (4112);
   \draw (4112)  -- (5112);
   \draw[black] (6112)  -- (7112);
   \draw[black] (3112)  -- (6112);
   \draw[black] (4112)  -- (7112);
   \draw[black] (6112)  -- (9112);
   \draw[black] (3112)  -- (12112) ;
   \draw[black] (4112)  -- (13112) ;
   \draw (5112)  -- (14112) ;

\end{tikzpicture}
\end{center}

 Then we define $\bt\in S(\mathbf D,2\omega_4)$ to be the multi--exponent with $t_{1,6}=t_{2,5}=t_{3,4}=1$ and $0$ else. It is immediate that the difference $\bs-\bt\in S(\mathbf D, (m-2)\omega_4)$.

\vspace{3pt}

\textbf{Case {2}:} In this case we suppose that one of the entries $s_{1,6},s_{2,5}$ or $s_{3,4}$ is zero. Then there is a Dyck path $\mathbf p$ such that $\bs$ is supported on $\bp$ and the statement is immediate.

\vspace{3pt}
So suppose that either $s_{4,4}\neq 0$ or $s_{1,7}\neq 0$. The proof in both cases is similar, so that we can assume $s_{4,4}\neq 0$. 

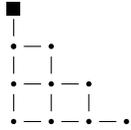
\begin{center}
\begin{tikzpicture}
[rotate=270]


\node (3112) at (-3.75,0.5 + 4) {};
  \fill[black] (3112) circle (1pt);
  \node (4112) at (-3.75,1+ 4) {};
  \fill[black] (4112) circle (1pt);
  \node (5112) at (-3.75,1.5+ 4) {};
  \fill[black] (5112) circle (1pt);
  \node (6112) at (-4.25,0.5+ 4) {};
  \fill[black] (6112) circle (1pt);
  \node (7112) at (-4.25,1+ 4) {};
  \fill[black] (7112) circle (1pt);

  \node (9112) at (-4.25,0.5+ 4) {};
  \fill[black] (9112) circle (1pt);
  
  \node (12112) at (-3.25,0.5+ 4) {};
  \fill[black] (12112) circle (1pt);
  \node (13112) at (-3.25,1+ 4) {};
 \fill[black] (13112) circle (1pt);
  \node (14112) at (-3.25,1.5+ 4) {};
  \fill[black] (14112) circle (1pt);

\node (39112) at (-3.25,2+ 4) {};
 \fill[black] (39112) circle (1pt);
 \draw  (13112) -- (14112) -- (39112);
\draw[black]  (12112) -- (13112);

\node (9112) at (-4.75,0.5+ 4) {};
 \fill[black] (-4.66,4.41) rectangle (-4.84,4.59);

         \draw[black] (3112)  -- (4112);
   \draw (4112)  -- (5112);
   \draw[black] (6112)  -- (7112);
   \draw[black] (3112)  -- (6112);
   \draw[black] (4112)  -- (7112);
   \draw[black] (6112)  -- (9112);
   \draw[black] (3112)  -- (12112) ;
   \draw[black] (4112)  -- (13112) ;
   \draw (5112)  -- (14112) ;

\end{tikzpicture}
\end{center}

We set $\mathbf{s}^1$ to be the multi--exponent obtained from $\bs$ by replacing $s_{4,4}$ by $s_{4,4}-1$. By induction we can find $\bt^1\in S(\mathbf D,2\omega_4)$ such that $\bs^{1}-\bt^{1}\in S(\mathbf D,(m-2)\omega_4)$. Now we define $\bt$ to be the multi--exponent obtained from $\bt^1$ by replacing $t_{4,4}$ by $t_{4,4}+1$ if the resulting element stays in $S(\mathbf D,2\omega_4 )$ and otherwise we set $\bt=\bt^1$. In either case $\bs-\bt\in S(\mathbf D,(m-2)\omega_4)$.

\begin{rem}\label{remspat}\mbox{}
\begin{enumerate}
\item
 The set $S(\mathbf D,m\omega_i)$ does not satisfy the usual Minkowski sum property in general, e.g. the element $(m_{\beta})\in S(\mathbf D,2\omega_4)$ ($n=4$) with $m_{\beta}=1$ for $\beta\in\{\alpha_{1,6},\alpha_{2,5},\alpha_{3,4}\}$ and else $0$ is not contained in $S(\mathbf D,\omega_4)+S(\mathbf D,\omega_4)$. Another example is the element $(m_{\beta})\in S(\mathbf D,2\omega_3)$ ($n=4$) with $m_{\beta}=1$ for $\beta\in\{\alpha_{1,3},\alpha_{1,4},\alpha_{1,6},\alpha_{2,5},\alpha_{3,3}\}$ and else $0$.
\item
The polytope $P(\mathbf D,\epsilon_im\omega_i)$ is defined by inequalities with integer coefficients and hence the Minkowski property in Lemma~\ref{reductlem} (ii) ensures that $P(\mathbf D,\epsilon_im\omega_i)$ is a normal polytope for $1\leq i\leq 3$ and $n$ arbitrary or $i$ arbitrary and $1\leq n\leq 4$. The proof is exactly the same as in \cite[Lemma 8.7]{FFL13}.
\end{enumerate}
\end{rem}

Summarizing, we have proved Conjecture~\ref{conjbasistypeB} for arbitrary $n$ and $1\leq i\leq 3$ or arbitrary $i$ and $1\leq n\leq 4$. Moreover the proof of the general case can be reduced to the proof of Lemma~\ref{reductlem} (ii) and Lemma~\ref{reductlem} (iii).
\section{Dyck path, polytopes and PBW bases for \texorpdfstring{$\mathfrak{so}_7$}{so7}}\label{section5} If the Lie algebra is of type $\tt B_3$ we shall associate to any dominant integral weight $\lambda$ a normal polytope and prove that a basis of $\gr V(\lambda)$ can be parametrized by the lattice points of this polytope. We emphasize at this point that the polytopes we will define for $\tt B_3$ are quasi compatible with the polytopes defined in Section~\ref{dyckfori}; see Remark~\ref{remvertr} for more details. 
\subsection{}\label{polytope}We use the following abbreviations:
$$\beta_1:=\al_{1,5}, \beta_2:=\al_{1,4}, \beta_3:=\al_{2,4}, \beta_4:=\al_{1,3}, \beta_5:=\al_{2,3}, \beta_6:=\al_{1,2}, \beta_7:=\al_{2,2}, \beta_8:=\al_{3,3}, \beta_9:=\al_{1,1}.$$

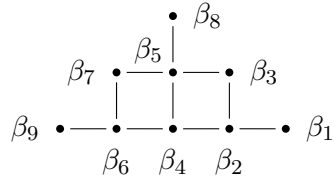
\begin{center}
\begin{tikzpicture}[scale=1.5]

 \node[label=right: $\beta_1$] (2511) at (-0.75,2.5+ 4) {};

 \fill[black] (2511) circle (1pt);
  \node[label=below: $\beta_2$] (2311) at (-1.25,2.5+ 4) {};
  \fill[black] (2311) circle (1pt);
  \node[label=right: $\beta_3$] (2411) at (-1.25,3+ 4) {};
  \fill[black] (2411) circle (1pt);
 
  \node[label=below: $\beta_4$] (3611) at (-1.75,2.5+ 4) {};
  \fill[black] (3611) circle (1pt);
	\node[label=above left: $\beta_5\!$] (37112) at (-1.7,3+ 3.9) {};
  \node (3711) at (-1.75,3+ 4) {};
  \fill[black] (3711) circle (1pt);
  \node[label=right: $\beta_8$] (3811) at (-1.75,3.5+ 4) {};
  \fill[black] (3811) circle (1pt);

  \node[label=below: $\beta_6$] (23112) at (-2.25,2.5+ 4) {};
  \fill[black] (23112) circle (1pt);
  \node[label=left: $\beta_7$] (24112) at (-2.25,3+ 4) {};
  \fill[black] (24112) circle (1pt);

 \node[label=left: $\beta_9$] (25112) at (-2.75,2.5+ 4) {};
\fill[black] (25112) circle (1pt);

    \draw (3611) -- (3711) -- (3811);
   \draw (2311) -- (2511);
    \draw (2311) -- (2411);
    \draw (2311) -- (3611);
    \draw (2411) -- (3711);

    \draw (23112) -- (24112);
    \draw (23112) -- (3611);
    \draw (24112) -- (3711);

   \draw (23112) -- (25112);
    \draw (23112) -- (24112);

\end{tikzpicture}
\end{center}

Let $\lambda=m_1\omega_1+m_2\omega_2+m_3\omega_3$, $s_i:=s_{\beta_i}$ for $1\leq i\leq 9$ and set $(a,b,c):=am_1+bm_2+cm_3$. We denote by $P(\lambda)\subseteq \R_+^9$ the polytope determined by the following inequalities:
\setlength{\multicolsep}{0,5cm}
\begin{multicols}{2}
\begin{enumerate}
\item  $s_2+s_3+s_4+s_8+s_9\leq (1,1,1)$ 
\item $s_3+s_4+s_5+s_8+s_9\leq (1,1,1)$
\item $s_4+s_5+s_6+s_8+s_9\leq (1,1,1)$ 
\item $s_5+s_6+s_7+s_8+s_9\leq (1,1,1)$
\item $s_3+s_5+s_8\leq (0,1,1)$
\item $s_5+s_7+s_8\leq (0,1,1)$
\item $s_6+s_7+s_9\leq (1,1,0)$
\item $s_7\leq (0,1,0)$
\item $s_8\leq (0,0,1)$ 
\item $s_9\leq (1,0,0)$
\item $s_3+s_4+s_5+s_6+s_7+s_8+s_9\leq (1,2,1)$
\item $s_1+s_2+s_3+s_4+s_5+s_7+s_9\leq (1,2,1)$
\item $s_1+s_3+s_4+s_5+s_6+s_7+s_9\leq (1,2,1)$
\item $s_2+s_3+s_4+s_5+s_7+s_8+s_9\leq (1,2,1)$
\item $s_1+s_2+s_3+s_4+s_5+s_6+s_7+2s_9\leq (2,2,1)$
\item $s_2+s_3+s_4+s_5+s_6+s_7+s_8+2s_9\leq (2,2,1)$
\item $s_1+s_2+2(s_3+s_4+s_5)+s_6+s_7+s_8+2s_9\leq (2,3,2)$
\item $s_2+2(s_3+s_4+s_5)+s_6+s_7+2(s_8+s_9)\leq (2,3,2)$
\item $s_3+s_4+2s_5+s_6+s_7+2s_8+s_9\leq (1,2,2)$
\end{enumerate}
\end{multicols}
As before we set $S(\lambda)=P(\lambda)\cap \mathbb Z_+^9$.
\begin{rem}\label{remvertr}
Assume that $\lambda=m\omega_i$ for some $1\leq i \leq 3$. If $i\neq 1$, then the polytope $P(\mathbf D,m\omega_i)$ defined in Section~\ref{dyckfori} coincides with the polytope given by the inequalities $(1)-(19)$. If $i=1$ these polytopes slightly differ in the following sense:
the polytope $P(\mathbf D,m\omega_1)$ from Section~\ref{dyckfori} is determined by the inequalities 
\begin{multicols}{2}
\begin{enumerate}
\item  $s_1+s_2+s_4+s_6\leq m$ 
\item  $s_2+s_4+s_6+s_9\leq m$
\end{enumerate}
\end{multicols}
whereas the above polytope can be simplified and is determined by the inequalities
\setlength{\multicolsep}{0,2cm}
\begin{multicols}{2}
\begin{enumerate}
\item  $s_1+s_4+s_6+s_9\leq m$ 
\item  $s_1+s_2+s_4+s_9\leq m$.
\end{enumerate}
\end{multicols}
We need to change $P(\mathbf D,\omega_1)$ for type $\tt B_3$ in order to ensure the Minkowski--sum property.
\end{rem}
\subsection{}
For the rest of this section we prove the following theorem.
\begin{thm}\label{mainthmsect5}
Let $\lie g$ be of type $\tt B_3$.
\begin{enumerate} 
\item The lattice points $S(\lambda)$ parametrize a basis of $V(\lambda)$ and $\gr V(\lambda)$ respectively. In particular, 
$$\{\X^{\mathbf s}v_{\lambda}\mid \mathbf s \in S(\lambda)\}$$
forms a basis of $\gr V(\lambda)$.
\item The character and graded $q$-character respectively is given by
$$\cha V(\lambda)=\sum_{\mu\in \lie h^{*}}|S(\lambda)^{\mu}|e^{\mu}$$
$$\cha_q \gr V(\lambda)=\sum_{\mathbf s\in S(\lambda)}e^{\lambda-\wt(\mathbf s)}q^{\sum s_{\beta}}.$$
\item We have an isomorphism of $S(\lie n^-)$--modules
$$\gr V(\lambda+\mu)\cong S(\lie n^-)(v_{\lambda}\otimes v_{\mu})\subseteq \gr V(\lambda)\otimes \gr V(\mu)$$
\end{enumerate}
\end{thm}
We choose the following order on the positive roots
\begin{equation}\label{rootorderB3}
\beta_7 \succ \beta_6 \succ \beta_1 \succ \beta_2 \succ \beta_3 \succ \beta_4 \succ \beta_5\succ \beta_8 \succ \beta_9.
\end{equation} As in Section~\ref{section4} we can deduce the above theorem from the following lemma.
\begin{lem}\label{reductlem2}\mbox{}
\renewcommand{\theenumi}{\roman{enumi}}%
\begin{enumerate}
\item Let $\lambda,\mu\in P^+$. We have
$$S(\lambda+\mu)=S(\lambda)+S(\mu)$$
\item For all $\lambda\in P^+$: $$\dim V(\lambda)=|S(\lambda)|$$
\item We have $$\text{es}(V(\omega_i))=S(\omega_i), \text{ for } 1\leq i\leq 3.$$
\end{enumerate}
\end{lem}
In part (iii) of the above lemma we mean the essential monomials with respect to the choosen order \eqref{rootorderB3}. The proof of Lemma~\ref{reductlem2} (i) is given in Section~\ref{section53} and the proof of Lemma~\ref{reductlem2} (ii) can be found in Section~\ref{section54}. Similar as in Section~\ref{spanprop1} we can prove for all $\bs\notin S(\omega_j)$ that $$\X^{\bs}v_{\omega_j}\in \spa\{\X^{\bq}v_{\omega_j}\mid \bq \prec \bs \},$$
hence the third part follows from (ii).
\subsection{Proof of Lemma~\ref{reductlem2} (i)}\label{section53}
For this part of the lemma it is enough to prove that $S(\lambda)=S(\lambda-\omega_j)+S(\omega_j)$ where $j$ is the minimal integer such that $\lambda(\alpha_j^{\vee})\neq 0$. Assume that $j=1$ and $\mathbf s=(s_i)_{1\leq i\leq 9}\in S(\lambda)$. We will consider several cases.
\vspace{3pt}

\textbf{Case {1}:} Assume that $s_9\neq 0$ and let $\mathbf t=(t_i)_{1\leq i\leq 9}$ be the multi--exponent given by $t_9=1$ and $t_j=0$ otherwise. It follows immediately $\mathbf t\in S(\omega_1)$ and $\mathbf s-\mathbf t\in S(\lambda-\omega_1)$.
\vspace{3pt}

\textbf{Case {2}:} In this case we suppose that $s_9=0$ and $s_2,s_6\neq 0$. 

\vspace{3pt}
\textbf{Case {2.1}:} If in addition $s_3+s_4+s_5+s_8<(1,1,1)$ we let $\mathbf t=(t_i)_{1\leq i\leq 9}$ to be the multi--exponent given by $t_2=t_6=1$ and $t_j=0$ otherwise. It is easy to show that $\mathbf t\in S(\omega_1)$ and $\mathbf s-\mathbf t\in S(\lambda-\omega_1)$, since $\mathbf s-\mathbf t\notin S(\lambda-\omega_1)$ forces $s_3+s_4+s_5+s_8=(1,1,1)$.

\vspace{3pt}
\textbf{Case {2.2}:} Now we suppose that $s_3+s_4+s_5+s_8=(1,1,1)$. Together with $(5)$ we obtain $s_4\geq m_1>0$. We let $\mathbf t=(t_i)_{1\leq i\leq 9}$ to be the multi--exponent with $t_4=1$ and $t_j=0$ otherwise. Suppose that $\mathbf s-\mathbf t\notin S(\lambda-\omega_1)$, which is only possible if $(4), (7), (15)$ or $(16)$ is violated. Assume that $(4)$ is violated, which means $s_5+s_6+s_7+s_8=(1,1,1)$. We obtain
$$(s_3+s_4+s_5+s_8)+(s_5+s_6+s_7+s_8)=s_3+s_4+2s_5+s_6+s_7+2s_8=(2,2,2),$$
which is a contradiction to $(19)$. Assume that $(7)$ is violated, which means $s_6+s_7=(1,1,0)$. We get
$$(s_3+s_4+s_5+s_8)+(s_6+s_7)=(2,2,1),$$
which is a contradiction to $(11)$. In the remaining two cases (inequality $(15)$ and $(16)$ respectively is violated) we obtain similarly contradictions to $(17)$ and $(18)$ respectively.

\vspace{3pt}
\textbf{Case {3}:} Assume that $s_2=s_9=0$ and $s_6\neq 0$. In this case many inequalities are redundant. In particular,  
for a multi--exponent $\mathbf t$ with $t_j\leq s_j$ for $1\leq j\leq 9$ we have $\mathbf s-\mathbf t\in S(\lambda-\omega_1)$ if and only if $\mathbf s-\mathbf t$ satisfies $(2)-(11), (13)$ and $(19).$ To be more precise,  
$$\bs-\bt \mbox{ satisfies $(2)$} \Rightarrow \bs-\bt \mbox{ satisfies $(1)$}$$
$$\bs-\bt \mbox{ satisfies $(13)$} \Rightarrow \bs-\bt \mbox{ satisfies $(12),(15)$}$$
$$\bs-\bt \mbox{ satisfies $(11)$} \Rightarrow \bs-\bt \mbox{ satisfies $(14),(16)$}$$
$$\bs-\bt \mbox{ satisfies $(2)$ and $(13)$} \Rightarrow \bs-\bt \mbox{ satisfies $(17)$}$$
$$\bs-\bt \mbox{ satisfies $(2)$ and $(11)$} \Rightarrow \bs-\bt \mbox{ satisfies $(18)$}$$

\vspace{3pt}
\textbf{Case {3.1}:} If in addition $s_3+s_4+s_5+s_8<(1,1,1)$ we let $\mathbf t=(t_i)_{1\leq i\leq 9}$ to be the multi--exponent given by $t_6=1$ and $t_j=0$ otherwise. It is straightforward to check that $\mathbf t\in S(\omega_1)$ and $\mathbf s-\mathbf t\in S(\lambda-\omega_1)$.
\vspace{3pt} 

\textbf{Case {3.2}:} If $s_3+s_4+s_5+s_8=(1,1,1)$ we let $\mathbf t=(t_i)_{1\leq i\leq 9}$ to be the multi--exponent with $t_4=1$ and $t_j=0$ otherwise. Note that $\mathbf s-\mathbf t\notin S(\lambda-\omega_1)$ is only possible if $(4)$ or $(7)$ is violated. If $(4)$ and $(7)$ respectively is violated we get similarly as in Case 2.2 a contradiction to $(19)$ and $(11)$ respectively.
\vspace{3pt}

\textbf{Case {4}:} Assume that $s_6=s_9=0$ and $s_2\neq 0$. This case works similar to Case 3 and will be omitted.

\vspace{3pt}
\textbf{Case {5}:} In this case we suppose $s_6=s_9=s_2=0$ and simplify further the defining inequalities of the polytope. As in Case 3, for a multi--exponent $\mathbf t$ with $t_j\leq s_j$ for $1\leq j\leq 9$ we have $\mathbf s-\mathbf t\in S(\lambda-\omega_1)$ if and only if $\mathbf s-\mathbf t$ satisfies $(2),(5),(6),(8)-(11), (13)$ and $(19).$ To be more precise, 
$$\bs-\bt \mbox{ satisfies $(2)$} \Rightarrow \bs-\bt \mbox{ satisfies $(3)$}$$
$$\bs-\bt \mbox{ satisfies $(6)$} \Rightarrow \bs-\bt \mbox{ satisfies $(4)$}$$
$$\bs-\bt \mbox{ satisfies $(8)$} \Rightarrow \bs-\bt \mbox{ satisfies $(7)$}$$

\vspace{3pt}
\textbf{Case {5.1}:} We suppose that $s_4\neq 0$ and let $\mathbf t=(t_i)_{1\leq i\leq 9}$ to be the multi--exponent given by $t_4=1$ and $t_j=0$ otherwise. The desired property follows immediately.

\vspace{3pt}
\textbf{Case {5.2}:} Let $s_4=0$. Then again we can simplify the inequalities and obtain that $\mathbf s-\mathbf t\in S(\lambda-\omega_1)$ if and only if $\mathbf s-\mathbf t$ satisfies $(5),(6),(8)-(10)$, and $(13).$ To be more precise,
$$\bs-\bt \mbox{ satisfies $(5)$} \Rightarrow \bs-\bt \mbox{ satisfies $(2)$}$$
$$\bs-\bt \mbox{ satisfies $(5)$ and $(8)$} \Rightarrow \bs-\bt \mbox{ satisfies $(11)$}$$
$$\bs-\bt \mbox{ satisfies $(5)$ and $(8)$} \Rightarrow \bs-\bt \mbox{ satisfies $(19)$}$$

\vspace{3pt}
\textbf{Case {5.2.1}:} If $s_1=0$ we already have $\mathbf s\in S(\lambda-m_1\omega_1)$. If $s_1\neq 0$, let $\mathbf t=(t_i)_{1\leq i\leq 9}$ be the multi--exponent with $t_1=1$ and $t_j=0$ otherwise. It follows immediately $\mathbf t\in S(\omega_1)$ and $\mathbf s-\mathbf t\in S(\lambda-\omega_1)$.\par
If $j=3$, many of the inequalities are redundant and the polytope can be simply described by the inequalities
$$s_1+s_2+s_3+s_4+s_5\leq (0,0,1),\ s_2+s_3+s_4+s_5+s_8\leq (0,0,1).$$
The proof of the lemma in that case is obvious. If $j=2$, there are again redundant inequalities and the polytope can simply described by the inequalities $(1)-(4), (7), (9)-(10)$ and $(15)-(16)$.
A straightforward calculation proves the proposition in that case; the details will be omitted.
$\hfill\Box$

\begin{rem}\label{remspat2}
The polytope $P(\lambda)$ is defined by inequalities with integer coefficients and hence the Minkowski property in Lemma~\ref{reductlem2} (i) ensures that $P(\lambda)$ is a normal polytope. The proof is exactly the same as in \cite[Lemma 8.7]{FFL13}.
\end{rem}
\subsection{Proof of Lemma~\ref{reductlem2} (ii)}\label{section54}
We consider the convex lattice polytopes $P_i:=P(\omega_i)\subseteq \R_+^9$ for $1\leq i \leq 3$. By \cite[Problem 3, pg. 164]{B08} there exists a 3--variate polynomial $E(T_1,T_2,T_3)$ of total degree $\leq 9$ such that 
$$E(m_1,m_2,m_3)=|(m_1P_1+m_2P_2+m_3P_3)\cap \Z_+^9|, \mbox{ for non--negative integers $m_1,m_2,m_3$.}$$
By Lemma~\ref{reductlem2} (i) we get 
$$E(m_1,m_2,m_3)=|S(\lambda)|, \mbox{ for non--negative integers $m_1,m_2,m_3$}$$
and by Weyl's dimension formula, we know that there is another 3--variate polynomial $W(T_1,T_2,T_3)$ of total degree $\leq 9$ such that 
$$W(m_1,m_2,m_3)= \dim V(\lambda).$$
The polynomial is given by
\begin{align*}
W(T_1,T_2,T_3) = \frac{1}{720}& (T_1 + 1 )(T_2 + 1)(T_2 +1)(T_1 + 2T_2 + T_3 +4)(2T_1+2T_2+T_3 + 5)\\
&(T_1 + T_2+ T_3 + 3)(T_1 +T_2 + 2)(T_2 + T_3 + 2)(2T_2 + T_3 + 3).\\
\end{align*}
Hence it will be enough to prove that both polynomials coincide. By using the code given in Section~\ref{code}, written in Java, we can deduce $E(\lambda_0,\lambda_1,\lambda_2)=W(\lambda_0,\lambda_1,\lambda_2)$ for all $(\lambda_0,\lambda_1,\lambda_2)\in \Z_+^{3}$ with $\lambda_0+\lambda_1+\lambda_2\leq 9$. We claim that this fact already implies $E(T_1,T_2,T_3)=W(T_1,T_2,T_3)$.
Let $I=\{(\lambda_0,\lambda_1,\lambda_2)\in \Z_+^{3}\mid \lambda_0+\lambda_1+\lambda_2\leq 9\}$ and write 
$$E(T_1,T_2,T_3)=\sum_{(n,m,k)\in I}e_{n,m,k}T_1^{n}T_2^{m}T_3^{k},\quad W(T_1,T_2,T_3)=\sum_{(n,m,k)\in I}w_{n,m,k}T_1^{n}T_2^{m}T_3^{k}$$

We obtain with our assumption that 
$$\sum_{(n,m,k)\in I}\big(e_{n,m,k}-w_{n,m,k}\big)\lambda_0^{n}\lambda_1^{m}\lambda_2^{k}=0.$$
We can translate this into a system of linear equations where the underlying matrix is given by
$$(\lambda^{\mu_0}_0\lambda_1^{\mu_1}\lambda_2^{\mu_2})_{\boldsymbol \lambda, \boldsymbol \mu \in I}.$$
This matrix is invertible by \cite[Theorem 1]{BCR92} and therefore the claim is proven. $\hfill\Box$
\section{Construction of favourable modules}\label{section6}
In \cite{FFL13} the notion of favourable modules has been introduced and several classes of examples for type $\tt A_n$, $\tt C_n$ and $\tt G_2$ have been discussed. This section is dedicated to give further examples of favourable modules in type $\tt B_n$. Let us first recall the definition.
\subsection{}
As in Section~\ref{induordm} we fix an ordered basis $\{x_1,\dots,x_N\}$ of $\mathfrak{n}^-$ and an induced homogeneous lexicographic order $<$ on the monomials in $\{x_1,\dots,x_N\}$. Let $M$ be any finite--dimensional cyclic $\mathbf U(\mathfrak{n}^-)$--module with cyclic vector $v_M$. We introduce subspaces $F_{\bs}(M)^{-}\subseteq F_{\bs}(M)\subseteq M$:
$$F_{\bs}(M)^{-}=\spa\{\X^{\bq}v_M\mid \bq<\bs\},\ F_{\bs}(M)=\spa\{\X^{\bq}v_M\mid \bq\leq \bs\}.$$
These subspaces define an increasing filtration on $M$ and the associated graded space with respect to this filtration is defined by
$$M^{t}=\bigoplus_{\bs\in \Z_+^{N}}F_{\bs}(M)/F_{\bs}(M)^{-}.$$
Similar as in Section~\ref{section3} we can define the PBW filtration on $M$ and the associated graded space $\gr M$ with respect to the PBW filtration. The following proposition follows from the construction of $M^{t}$ and $\gr M$ (see also \cite[Proposition.1.5]{FFL13}).
\begin{prop}\label{kppp}
The set $\{\X^{\bs}\mid \bs\in {\rm es}(M)\}$ forms a basis of $M^{t}$, $\gr M$ and $M$.
\end{prop}
\subsection{}
We recall the definition of favourable modules.
\begin{defn}\label{favo}
We say that a finite--dimensional cyclic $\mathbf{U}(\mathbf{n}^-)$--module $M$ is favourable if there exists an
ordered basis $x_1,\dots,x_N$ of $\mathbf{n}^-$ and an induced homogeneous monomial order on the PBW basis
such that
\begin{itemize}
\item There exists a normal polytope $P(M)\subset \mathbb R^{N}$
such that ${\rm es}(M)$ is exactly the set $S(M)$ of lattice points in $P(M)$.
\item $\forall\, k\in\mathbb N:\,\dim \mathbf U(\mathfrak{n}^-)(\underbrace{v_M\otimes\cdots \otimes v_M}_{k})=|\underbrace{S(M)+\cdots+S(M)}_{k}|.$
\end{itemize}
\end{defn}
Let $N$ be a complex algebraic unipotent group such that $\mathfrak{n}^{-}$ is the corresponding Lie algebra. Similarly on the group level, we have a commutative
unipotent group $\gr N$ with Lie algebra $\gr \mathfrak{n}^{-}$ acting on $\gr M$ and $M^{t}$. We associate to the action of the unipotent groups projective varieties, which are called flag varieties in analogy to the classical highest weight orbits (see \cite{FFL13} for details)
$$
\mathfrak{F}(M)=\overline{N.[v_M]}\subseteq \mathbb P(M),\quad
\mathfrak{F}(\gr M)=\overline{\gr N.[v_M]}\subseteq \mathbb P(\gr M),\quad
\mathfrak{F}(M^t)=\overline{\gr N.[v_M]}\subset \mathbb P(M^{t}).
$$
A motivation for constructing favourable modules is that the flag varieties associated to favourable modules have nice properties (see \cite{FFL13} for details), such as

\begin{enumerate}
\item $\mathfrak{F}(M^t)\subseteq \mathbb P(M^t)$ is a toric variety.
\item There exists a flat degeneration of $\mathfrak{F}(M)$ into
$\mathfrak{F}(\gr M)$, and for both there exists a flat degeneration
into $\mathfrak{F}(M^t)$.
\item The projective flag varieties $\mathfrak{F}(M)\subseteq \mathbb P(M)$ and its abelianized versions
$\mathfrak{F}(\gr M)\subseteq \mathbb P(\gr M)$ and
$\mathfrak{F}(M^t)\subseteq \mathbb P(M^{t})$ are projectively normal and arithmetically Cohen--Macaulay
varieties.
\item The polytope $P(M)$ is the Newton--Okounkov body for the flag variety and its abelianized version,
i.e. $\Delta(\mathfrak{F}(M))=P(M)=\Delta(\mathfrak{F}(\gr M))$.
\end{enumerate}

\subsection{} 
In \cite[Section 8]{FFL13} the authors provided concrete classes of examples of favourable modules for the types $\tt A_n$, $\tt C_n$ and $\tt G_2$. The following theorem gives us classes of examples of favourable modules in type $\tt B_n$ (including multiples of the adjoint representation).
\begin{thm}
Let $\Lg$ be the Lie algebra of type $\tt B_n$ and $\lambda$ be a dominant integral weight satisfying one of the following
\begin{enumerate}
\item $n=3$ and $\lambda$ is arbitrary 
\item $n$ is arbitrary and $\lambda=m\omega_1$ or $\lambda=m\omega_2$
\item $n$ is arbitrary and $\lambda=2m\omega_3$ or $n=4$ and $\lambda=2m\omega_4$
\end{enumerate}
 Then there exists an ordered basis on $\mathfrak{n}^{-}$ and an induced homogeneous monomial order on the PBW basis such that $V(\lambda)$ is a favourable $\mathfrak{n^{-}}$--module.


\begin{proof}
We will show that $V(\lambda)$ satisfies the properties from Definition~\ref{favo}. We consider the appropriate polytopes from \eqref{polyi} and $P(\lambda)$ from Section~\ref{section5}. These polytopes are normal by Remark~\ref{remspat} and Remark~\ref{remspat2} and therefore the natural candidates for showing the properties from  Definition~\ref{favo}. For simplicity we will denote these polytopes by $P(\lambda)$ since it will be clear from the context which polytope we mean. The second property follows immediately since on the one hand $\mathbf U(\mathfrak{n^{-}})(v_{\lambda}\otimes\dots\otimes v_{\lambda})\cong V(k\lambda)$ and on the other hand the $k$--fold Minkowski sum parametrizes a basis of $V(k\lambda)$ by Theorem~\ref{mainthmsect5} (1), Lemma~\ref{reductlem2} (i), Conjecture~\ref{conjbasistypeB} (1) (which is proved in theses cases) and Lemma~\ref{reductlem} (ii). Hence it remains to prove that ${\rm es}(V(\lambda))$ (with respect to a fixed order) is exactly the set $S(\lambda)$. Let $\lambda=\sum_{j=1}^nm_ja_j\omega_j$. By \cite[Proposition 1.11]{FFL13} we know that
\begin{equation}\label{minkowskiessential}
{\rm es}(V(\lambda))\supseteq \underbrace{{\rm es}(V(a_1\omega_1))+\dots+{\rm es}(V(a_1\omega_1))}_{m_1}+\dots +
\underbrace{{\rm es}(V(a_n\omega_n))+\dots+{\rm es}(V(a_n\omega_n))}_{m_n},
\end{equation}
and hence it is enough to show that there exists an ordered basis on $R^+$ and an induced homogeneous monomial order on a PBW basis such that ${\rm es}(V(a_j\omega_j))=S(a_j\omega_j)$ for all $j$ with $m_j\neq 0$ (recall from Proposition~\ref{kppp} that $|{\rm es}(V(\lambda))|=|S(\lambda)|$). If we are in case $(2)$ or $(3)$ (then $a_j=1$ and $a_k=0$ for all $k\neq j$ in case $(2)$ and in case $(3)$ we have $a_3=2$ respectively $a_4=2$ and $a_k=0$ else), we choose the order given in Section~\ref{spanprop1} (we ordered the roots in the Hasse diagram from the bottom to the top and from left to right). Otherwise ($a_j=1$ for all $j$) we choose the order defined in \eqref{rootorderB3}. With respect to this order we proved ${\rm es}(V(a_j\omega_j))=S(a_j\omega_j)$ in Section~\ref{section4} and Section~\ref{section5} respectively.
\end{proof}
\end{thm}
\section{Appendix}\label{section7}
In this section we complete the proof of Proposition~\ref{ghtzu} for $i=3$. Moreover, we give a proof of the second part of Theorem~\ref{basistypeAandC} for type $\tt G_2$.
\subsection{}
We consider the Lie algebra of type $\tt G_2$ and the following order on the positive roots:
$$\beta_1:=3\alpha_1+2\alpha_2 \succ \beta_2:= 3\alpha_1+\alpha_2 \succ \beta_3:=2\alpha_1+\alpha_2 \succ \beta_4:=\alpha_1+\alpha_2 \succ \beta_5:=\alpha_2\succ \beta_6:=\alpha_1.$$
As before, we extend the above order to the induced homogeneous reverse lexicographic order on the monomials in $S(\mathfrak{n}^-)$. The order is chosen in a way such that Lemma~\ref{maxinre} can be applied.
Let $\lambda=m_1\omega_1+m_2\omega_2$, $s_i:=s_{\beta_i}$ for $1\leq i\leq 6$ and set $(a,b):=am_1+bm_2$. It has been proved in \cite{G11} that the lattice points $S(\lambda)$ of the following polytope $P(\lambda)$ parametrize a basis of $\gr V(\lambda)$:
\setlength{\multicolsep}{0,5cm}
\begin{multicols}{2}
\begin{enumerate}
\item  $s_6\leq (1,0)$ 
\item $s_5\leq (0,1)$
\item $s_2+s_3+s_6\leq (1,1)$
\item $s_3+s_4+s_6\leq (1,1)$
\item $s_4+s_5+s_6\leq (1,1)$ 
\item $s_1+s_2+s_3+s_4+s_5\leq (1,2)$
\item $s_2+s_3+s_4+s_5+s_6\leq (1,2)$
\end{enumerate}
\end{multicols}
\begin{prop}\label{forg2id}
We have $\gr V(\lambda)\cong S(\mathfrak{n}^-)/\mathbf I_{\lambda}$, where
$$\mathbf I_{\lambda}=S(\lie n^-)\big(\bu(\lie n^+)\circ \spa\{x^{\lambda(\beta^{\vee})+1}_{-\beta}\mid \beta\in R^+\}\big).$$
\proof
Since we have a surjective map 
$$ S(\mathfrak{n}^-)/\mathbf I_{\lambda}\longrightarrow \gr V(\lambda),$$
it will be enough to show by the result of \cite{G11} that the set $\{\X^{\bs}v_{\lambda}\mid \bs \in S(\lambda)\}$ generates $S(\mathfrak{n}^-)/\mathbf I_{\lambda}$. As in Section~\ref{section4} we will simply show that any multi--exponent $\bs$ violating on of the inequalities $(1)-(7)$ can be written as a sum of strictly smaller monomials. It means there exists constants $c_{\bt}\in \mathbb C$ such that 
$$\X^{\bs}+\sum_{\bt \prec \bs}c_{\bt}\X^{\bt}\in \mathbf I_{\lambda}.$$
The proof for all inequalities is similar and therefore we provide the proof only when $\bs$ violates $(7)$. So let $\bs$ be a multi--exponent with $s_1=0$ and
$s_2+s_3+s_4+s_5+s_6>(1,2)$. We apply the operators $\partial^{s_4+s_6}_{\beta_3}\partial^{s_5}_{\beta_2}$ on $\X_{\beta_1}$ and obtain 
$$\partial^{s_4+s_6}_{\beta_3}\partial^{s_5}_{\beta_2}\X_{\beta_1}^{s_2+s_3+s_4+s_5+s_6}=c \X_{\beta_1}^{s_2+s_3}\X_{\beta_4}^{s_4+s_6}\X^{s_5}_{\beta_5}\in \mathbf I_{\lambda} ,\ \mbox{ for some non--zero constant $c\in \mathbb C$.}$$
Further we apply with $\partial_{\beta_5}^{s_2}\partial^{s_3}_{\beta_4}$ on $\X_{\beta_1}^{s_2+s_3}\X_{\beta_4}^{s_4+s_6}\X^{s_5}_{\beta_5}$ and obtain with Lemma~\ref{maxinre} that there exists constants $c_{\bt}\in \mathbb C$ such that
\begin{equation}\label{ztzt}\partial_{\beta_5}^{s_2}\partial^{s_3}_{\beta_4}\X_{\beta_1}^{s_2+s_3}\X_{\beta_4}^{s_4+s_6}\X^{s_5}_{\beta_5}= \X_{\beta_2}^{s_2}\X^{s_3}_{\beta_3}\X_{\beta_4}^{s_4+s_6}\X^{s_5}_{\beta_5}+\sum_{\bt \prec \bs}c_{\bt}\X^{\bt}\in \mathbf I_{\lambda}.\end{equation}
Finally, we act with the operator $\partial^{s_6}_{\beta_5}$ on \eqref{ztzt} and get once more with Lemma~\ref{maxinre} the desired property. 
\endproof
\end{prop}
\subsection{}\label{i3}
\textit{Proof of Proposition~\ref{ghtzu} for $i=3$:} Recall that a bold dot (resp. square) in the Hasse diagram indicates that the corresponding entry of $\bs$ is zero (resp. non--zero). Let $i=3$ and $\bs\in S(\mathbf D,m\omega_3)$. If $s_{3,j}=0$ for all $3\leq j \leq 2n-3$ the statement of the proposition can be easily deduced from the $i=2$ case. So we can suppose for the rest of the proof that $s_{3,j}\neq 0$ for some $3\leq j \leq 2n-3$. In contrast to the $i=2$ case we will construct a multi--exponent $\bt\in S(\mathbf D,p\omega_3)$ such that $\bs-\bt\in S(\mathbf D,(m-p)\omega_3)$ where $p=1$ or $p=2$. A similar induction argument as in the $i=2$ case shows that it is enough to prove the statement for all multi--exponents $\bs$ with $s_{\theta}=0$. Since $s_{\theta}=0$ it is sufficient to check the defining inequalities of the polytope for all $\mathbf p\in \mathbf D \backslash \mathbf q$, where $\mathbf q$ is the unique type 2 Dyck path with $\theta\in \mathbf q$. In other words
\begin{equation*}\label{reichtl}\sum_{\beta\in \mathbf p}(s_{\beta}-t_{\beta})\leq M_{\bp}((m-p)\omega_3),\ \forall \mathbf p\in \mathbf D\backslash \mathbf q \ \Rightarrow \bs-\bt\in S(\mathbf D,(m-p)\omega_3).\end{equation*}
We consider several cases.\vspace{3pt}

\textbf{Case {1}:} In this case we suppose $s_{3,2n-3}\neq 0$.

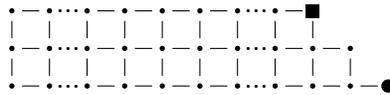
\begin{center}
\begin{tikzpicture}

  \node (311) at (0.25,0.5 + 4) {};
  \fill[black] (311) circle (1pt);
  \node (411) at (0.25,1+ 4) {};
  \fill[black] (411) circle (1pt);
  \node (511) at (0.25,1.5+ 4) {};
  \fill[black] (0.16,5.41) rectangle (0.34,5.59);
  \node (611) at (0.75,0.5+ 4) {};
  \fill[black] (611) circle (1pt);
  \node (711) at (0.75,1+ 4) {};
  \fill[black] (711) circle (1pt);
  \node (911) at (1.25,0.5+ 4) {};
  \fill[black] (911) circle (2.5pt);
 
  \node (1211) at (-0.25,0.5+ 4) {};
  \fill[black] (1211) circle (1pt);
  \node (1311) at (-0.25,1+ 4) {};
 \fill[black] (1311) circle (1pt);
  \node (1411) at (-0.25,1.5+ 4) {};
  \fill[black] (1411) circle (1pt);
  \node (2711) at (-0.75,0.5+ 4) {};
  \fill[black] (2711) circle (1pt);
  \node (2811) at (-0.75,1+ 4) {};
  \fill[black] (2811) circle (1pt);
  \node (2911) at (-0.75,1.5+ 4) {};
  \fill[black] (2911) circle (1pt);


 \draw  (1311) -- (1411);
\draw[black]  (1211) -- (1311);
 \draw (2811) -- (2911);
     \node (1611) at (-0.5,0.5+ 4) {\textcolor{black}{...}};
    \node (1611xx) at (-0.5,1.5+ 4) {...};
    \node (1611xxx) at (-0.5,1+ 4) {\textcolor{black}{...}};
	
  \node (1711) at (-1.25,0.5+ 4) {};
  \fill[black] (1711) circle (1pt);
  \node (1811) at (-1.25,1+ 4) {};
  \fill[black] (1811) circle (1pt);
  \node (1911) at (-1.25,1.5+ 4) {};
  \fill[black] (1911) circle (1pt);
	
   
   
 
  \node (3111) at (-1.75,0.5+ 4) {};
  \fill[black] (3111) circle (1pt);
  \node (3211) at (-1.75,1+ 4) {};
  \fill[black] (3211) circle (1pt);
  \node (3311) at (-1.75,1.5+ 4) {};
  \fill[black] (3311) circle (1pt);
   
   
 
 \node (17112) at (-2.25,0.5+ 4) {};
  \fill[black] (17112) circle (1pt);
  \node (18112) at (-2.25,1+ 4) {};
  \fill[black] (18112) circle (1pt);
  \node (19112) at (-2.25,1.5+ 4) {};
  \fill[black] (19112) circle (1pt);
   
   

\node (27112) at (-2.75,0.5+ 4) {};
  \fill[black] (27112) circle (1pt);
  \node (28112) at (-2.75,1+ 4) {};
  \fill[black] (28112) circle (1pt);
  \node (29112) at (-2.75,1.5+ 4) {};
  \fill[black] (29112) circle (1pt);

\node (3112) at (-3.75,0.5 + 4) {};
  \fill[black] (3112) circle (1pt);
  \node (4112) at (-3.75,1+ 4) {};
  \fill[black] (4112) circle (1pt);
  \node (5112) at (-3.75,1.5+ 4) {};
  \fill[black] (5112) circle (1pt);
 
  \node (12112) at (-3.25,0.5+ 4) {};
  \fill[black] (12112) circle (1pt);
  \node (13112) at (-3.25,1+ 4) {};
 \fill[black] (13112) circle (1pt);
  \node (14112) at (-3.25,1.5+ 4) {};
  \fill[black] (14112) circle (1pt);

 \draw  (13112) -- (14112);
\draw[black]  (12112) -- (13112);
 \draw (28112) -- (29112) ;
\draw[black] (27112) -- (28112);
   
    \node (16112) at (-3,1+ 4) {\textcolor{black}{...}};
     \node (16112x) at (-3,1.5+ 4) {...};
    \node (16112xxx) at (-3,0.5+ 4) {\textcolor{black}{...}};
   
   \draw[black] (311)  -- (411);
   \draw (411)  -- (511);
   \draw[black] (611)  -- (711);
   \draw[black] (311)  -- (611);
   \draw[black] (411)  -- (711);
   \draw[black] (611)  -- (911);
   \draw[black] (311)  -- (1211) ;
   \draw[black] (411)  -- (1311) ;
   \draw (511)  -- (1411) ;
   \draw[black] (2711) -- (1711);
   \draw[black] (2811) -- (1811);
   \draw  (2911) -- (1911);
   \draw[black] (1711) -- (1811);
    \draw[black] (2711) -- (2811);
    \draw (1811) -- (1911);
   \draw[black] (1711) -- (3111);
   \draw[black] (1811) -- (3211);
   \draw (1911) -- (3311);
   \draw[black] (3111) -- (3211);
  \draw (3211) -- (3311);
       
   \draw (18112) -- (19112);
    \draw[black] (17112) -- (18112);
   \draw[black] (17112) -- (3111);
   \draw[black] (18112) -- (3211);
   \draw (19112) -- (3311);

         \draw[black] (3112)  -- (4112);
   \draw (4112)  -- (5112);
   \draw[black] (3112)  -- (12112) ;
   \draw[black] (4112)  -- (13112) ;
   \draw (5112)  -- (14112) ;
   \draw[black] (27112) -- (17112);
   \draw[black] (28112) -- (18112);
   \draw  (29112) -- (19112);
 

\end{tikzpicture}
\end{center}

 Let $\bt\in\mathbf T(1)$ be the multi--exponent with $\supp(\bt)=\{\al_{3,2n-3},\al_{k,3}\}$, where $k=\min \{1\leq j\leq 2 \mid s_{j,3}\neq 0\}$. If $k$ exists, it is easy to see that $\bt\in S(\mathbf D,\omega_3)$ and $\bs-\bt\in S(\mathbf D,(m-1)\omega_3)$. So suppose that $s_{1,3}=s_{2,3}=0.$

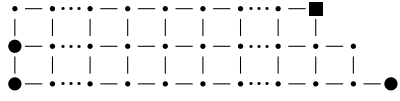
\begin{center}
\begin{tikzpicture}

  \node (311) at (0.25,0.5 + 4) {};
  \fill[black] (311) circle (1pt);
  \node (411) at (0.25,1+ 4) {};
  \fill[black] (411) circle (1pt);
  \node (511) at (0.25,1.5+ 4) {};
  \fill[black] (0.16,5.41) rectangle (0.34,5.59);
  \node (611) at (0.75,0.5+ 4) {};
  \fill[black] (611) circle (1pt);
  \node (711) at (0.75,1+ 4) {};
  \fill[black] (711) circle (1pt);
  \node (911) at (1.25,0.5+ 4) {};
  \fill[black] (911) circle (2.5pt);
 
  \node (1211) at (-0.25,0.5+ 4) {};
  \fill[black] (1211) circle (1pt);
  \node (1311) at (-0.25,1+ 4) {};
 \fill[black] (1311) circle (1pt);
  \node (1411) at (-0.25,1.5+ 4) {};
  \fill[black] (1411) circle (1pt);
  \node (2711) at (-0.75,0.5+ 4) {};
  \fill[black] (2711) circle (1pt);
  \node (2811) at (-0.75,1+ 4) {};
  \fill[black] (2811) circle (1pt);
  \node (2911) at (-0.75,1.5+ 4) {};
  \fill[black] (2911) circle (1pt);


 \draw  (1311) -- (1411);
\draw[black]  (1211) -- (1311);
 \draw (2811) -- (2911);
     \node (1611) at (-0.5,0.5+ 4) {\textcolor{black}{...}};
    \node (1611xx) at (-0.5,1.5+ 4) {...};
    \node (1611xxx) at (-0.5,1+ 4) {\textcolor{black}{...}};
	
  \node (1711) at (-1.25,0.5+ 4) {};
  \fill[black] (1711) circle (1pt);
  \node (1811) at (-1.25,1+ 4) {};
  \fill[black] (1811) circle (1pt);
  \node (1911) at (-1.25,1.5+ 4) {};
  \fill[black] (1911) circle (1pt);
	
   
   
 
  \node (3111) at (-1.75,0.5+ 4) {};
  \fill[black] (3111) circle (1pt);
  \node (3211) at (-1.75,1+ 4) {};
  \fill[black] (3211) circle (1pt);
  \node (3311) at (-1.75,1.5+ 4) {};
  \fill[black] (3311) circle (1pt);
   
   
 
 \node (17112) at (-2.25,0.5+ 4) {};
  \fill[black] (17112) circle (1pt);
  \node (18112) at (-2.25,1+ 4) {};
  \fill[black] (18112) circle (1pt);
  \node (19112) at (-2.25,1.5+ 4) {};
  \fill[black] (19112) circle (1pt);
   
   

\node (27112) at (-2.75,0.5+ 4) {};
  \fill[black] (27112) circle (1pt);
  \node (28112) at (-2.75,1+ 4) {};
  \fill[black] (28112) circle (1pt);
  \node (29112) at (-2.75,1.5+ 4) {};
  \fill[black] (29112) circle (1pt);

\node (3112) at (-3.75,0.5 + 4) {};
  \fill[black] (3112) circle (2.5pt);
  \node (4112) at (-3.75,1+ 4) {};
  \fill[black] (4112) circle (2.5pt);
  \node (5112) at (-3.75,1.5+ 4) {};
  \fill[black] (5112) circle (1pt);
 
  \node (12112) at (-3.25,0.5+ 4) {};
  \fill[black] (12112) circle (1pt);
  \node (13112) at (-3.25,1+ 4) {};
 \fill[black] (13112) circle (1pt);
  \node (14112) at (-3.25,1.5+ 4) {};
  \fill[black] (14112) circle (1pt);

 \draw  (13112) -- (14112);
\draw[black]  (12112) -- (13112);
 \draw (28112) -- (29112) ;
\draw[black] (27112) -- (28112);
   
    \node (16112) at (-3,1+ 4) {\textcolor{black}{...}};
     \node (16112x) at (-3,1.5+ 4) {...};
    \node (16112xxx) at (-3,0.5+ 4) {\textcolor{black}{...}};
   
   \draw[black] (311)  -- (411);
   \draw (411)  -- (511);
   \draw[black] (611)  -- (711);
   \draw[black] (311)  -- (611);
   \draw[black] (411)  -- (711);
   \draw[black] (611)  -- (911);
   \draw[black] (311)  -- (1211) ;
   \draw[black] (411)  -- (1311) ;
   \draw (511)  -- (1411) ;
   \draw[black] (2711) -- (1711);
   \draw[black] (2811) -- (1811);
   \draw  (2911) -- (1911);
   \draw[black] (1711) -- (1811);
    \draw[black] (2711) -- (2811);
    \draw (1811) -- (1911);
   \draw[black] (1711) -- (3111);
   \draw[black] (1811) -- (3211);
   \draw (1911) -- (3311);
   \draw[black] (3111) -- (3211);
  \draw (3211) -- (3311);
       
   \draw (18112) -- (19112);
    \draw[black] (17112) -- (18112);
   \draw[black] (17112) -- (3111);
   \draw[black] (18112) -- (3211);
   \draw (19112) -- (3311);

         \draw[black] (3112)  -- (4112);
   \draw (4112)  -- (5112);
  
   \draw[black] (3112)  -- (12112) ;
   \draw[black] (4112)  -- (13112) ;
   \draw (5112)  -- (14112) ;
   \draw[black] (27112) -- (17112);
   \draw[black] (28112) -- (18112);
   \draw  (29112) -- (19112);
 

\end{tikzpicture}
\end{center}

Now we consider two additional cases. 
\vspace{3pt}

\textbf{Case {1.1}:} First we assume that $\sum^{2n-4}_{k=3}s_{3,k}=m$ (sum over the red dots and the red square), which forces $s_{3,3}\neq 0$. 

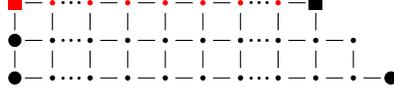
\begin{center}
\begin{tikzpicture}

  \node (311) at (0.25,0.5 + 4) {};
  \fill[black] (311) circle (1pt);
  \node (411) at (0.25,1+ 4) {};
  \fill[black] (411) circle (1pt);
  \node (511) at (0.25,1.5+ 4) {};
  \fill[black] (0.16,5.41) rectangle (0.34,5.59);
  \node (611) at (0.75,0.5+ 4) {};
  \fill[black] (611) circle (1pt);
  \node (711) at (0.75,1+ 4) {};
  \fill[black] (711) circle (1pt);
  \node (911) at (1.25,0.5+ 4) {};
  \fill[black] (911) circle (2.5pt);
 
  \node (1211) at (-0.25,0.5+ 4) {};
  \fill[black] (1211) circle (1pt);
  \node (1311) at (-0.25,1+ 4) {};
 \fill[black] (1311) circle (1pt);
  \node (1411) at (-0.25,1.5+ 4) {};
  \fill[red] (1411) circle (1pt);
  \node (2711) at (-0.75,0.5+ 4) {};
  \fill[black] (2711) circle (1pt);
  \node (2811) at (-0.75,1+ 4) {};
  \fill[black] (2811) circle (1pt);
  \node (2911) at (-0.75,1.5+ 4) {};
  \fill[red] (2911) circle (1pt);


 \draw  (1311) -- (1411);
\draw[black]  (1211) -- (1311);
 \draw (2811) -- (2911);
     \node (1611) at (-0.5,0.5+ 4) {\textcolor{black}{...}};
    \node (1611xx) at (-0.5,1.5+ 4) {...};
    \node (1611xxx) at (-0.5,1+ 4) {\textcolor{black}{...}};
	
  \node (1711) at (-1.25,0.5+ 4) {};
  \fill[black] (1711) circle (1pt);
  \node (1811) at (-1.25,1+ 4) {};
  \fill[black] (1811) circle (1pt);
  \node (1911) at (-1.25,1.5+ 4) {};
  \fill[red] (1911) circle (1pt);
	
   
   
 
  \node (3111) at (-1.75,0.5+ 4) {};
  \fill[black] (3111) circle (1pt);
  \node (3211) at (-1.75,1+ 4) {};
  \fill[black] (3211) circle (1pt);
  \node (3311) at (-1.75,1.5+ 4) {};
  \fill[red] (3311) circle (1pt);
   
   
 
 \node (17112) at (-2.25,0.5+ 4) {};
  \fill[black] (17112) circle (1pt);
  \node (18112) at (-2.25,1+ 4) {};
  \fill[black] (18112) circle (1pt);
  \node (19112) at (-2.25,1.5+ 4) {};
  \fill[red] (19112) circle (1pt);
   
   

\node (27112) at (-2.75,0.5+ 4) {};
  \fill[black] (27112) circle (1pt);
  \node (28112) at (-2.75,1+ 4) {};
  \fill[black] (28112) circle (1pt);
  \node (29112) at (-2.75,1.5+ 4) {};
  \fill[red] (29112) circle (1pt);

\node (3112) at (-3.75,0.5 + 4) {};
  \fill[black] (3112) circle (2.5pt);
  \node (4112) at (-3.75,1+ 4) {};
  \fill[black] (4112) circle (2.5pt);
  \node (5112) at (-3.75,1.5+ 4) {};
  \fill[red] (-3.84,5.41) rectangle (-3.66,5.59);
 
  \node (12112) at (-3.25,0.5+ 4) {};
  \fill[black] (12112) circle (1pt);
  \node (13112) at (-3.25,1+ 4) {};
 \fill[black] (13112) circle (1pt);
  \node (14112) at (-3.25,1.5+ 4) {};
  \fill[red] (14112) circle (1pt);

 \draw  (13112) -- (14112);
\draw[black]  (12112) -- (13112);
 \draw (28112) -- (29112) ;
\draw[black] (27112) -- (28112);
   
    \node (16112) at (-3,1+ 4) {\textcolor{black}{...}};
     \node (16112x) at (-3,1.5+ 4) {...};
    \node (16112xxx) at (-3,0.5+ 4) {\textcolor{black}{...}};
   
   \draw[black] (311)  -- (411);
   \draw (411)  -- (511);
   \draw[black] (611)  -- (711);
   \draw[black] (311)  -- (611);
   \draw[black] (411)  -- (711);
   \draw[black] (611)  -- (911);
   \draw[black] (311)  -- (1211) ;
   \draw[black] (411)  -- (1311) ;
   \draw (511)  -- (1411) ;
   \draw[black] (2711) -- (1711);
   \draw[black] (2811) -- (1811);
   \draw  (2911) -- (1911);
   \draw[black] (1711) -- (1811);
    \draw[black] (2711) -- (2811);
    \draw (1811) -- (1911);
   \draw[black] (1711) -- (3111);
   \draw[black] (1811) -- (3211);
   \draw (1911) -- (3311);
   \draw[black] (3111) -- (3211);
  \draw (3211) -- (3311);
       
   \draw (18112) -- (19112);
    \draw[black] (17112) -- (18112);
   \draw[black] (17112) -- (3111);
   \draw[black] (18112) -- (3211);
   \draw (19112) -- (3311);

         \draw[black] (3112)  -- (4112);
   \draw (4112)  -- (5112);
  
   \draw[black] (3112)  -- (12112) ;
   \draw[black] (4112)  -- (13112) ;
   \draw (5112)  -- (14112) ;
   \draw[black] (27112) -- (17112);
   \draw[black] (28112) -- (18112);
   \draw  (29112) -- (19112);
 

\end{tikzpicture}
\end{center}

Then we define $\bt\in\mathbf T(1)$ to be the multi--exponent with $\supp(\bt)=\{\al_{3,2n-3},\al_{3,3}\}$. We shall prove that $\bs-\bt\in S(\mathbf D,(m-1)\omega_3)$. For any $\mathbf p\in \mathbf D^{\typ 1}$ we obviously have $\sum_{\beta\in \mathbf p}(s_{\beta}-t_{\beta})\leq m-1$. So let $\mathbf p=\mathbf p_1\cup \mathbf p_1\in \mathbf D^{\typ 2}\backslash \bq$. If $\al_{3,3}\in \mathbf p_2$, there is nothing to show. Otherwise we get that $\mathbf p_2$ is of the form
 $$\bp_2 =\{\al_{2,3},\al_{2,4} \dots \al_{2,p},\al_{3,p},\al_{3,p+1},\dots \al_{3,2n-3}\},\ 3<p\leq 2n-3$$

and $$\sum_{\beta\in \bp_2}s_{\beta}\leq m=s_{2,3}+\sum^{2n-4}_{k=3}s_{3,k}.$$ It follows

$$\sum_{\beta\in \mathbf p_1}(s_{\beta}-t_{\beta})+\sum_{\beta\in \mathbf p_2}(s_{\beta}-t_{\beta})\leq \sum_{\beta\in \mathbf p_1}(s_{\beta}-t_{\beta})+s_{2,3}+\sum^{2n-3}_{k=3}(s_{3,k}-t_{3,k})\leq 2(m-1).$$
\vspace{3pt}

\textbf{Case {1.2}:}
It remains to consider the case $\sum^{2n-4}_{k=3}s_{3,k}\leq m-1$. Since $s_{1,3}=s_{2,3}=0$ it is enough to construct a multi--exponent $\mathbf{t}\in  S(\mathbf D,\omega_3)$ such that 

\begin{equation}\label{reichtll}\sum_{\beta\in \mathbf p}(s_{\beta}-t_{\beta})\leq M_{\bp}((m-1)\omega_3),\ \forall \mathbf p\in \mathbf D_2^{\typ 1}\cup \mathbf D^{\typ 2}.\end{equation}

We define $\bt\in\mathbf T(1)$ to be the multi--exponent with $\supp(\bt)=\{\al_{3,2n-3}\}$ if $s_{1,2n-2}=s_{2,2n-2}=0$ and otherwise $\supp(\bt)=\{\al_{3,2n-3},\al_{k,2n-2}\}$, where $k=\max \{1\leq j \leq 2 \mid s_{j,2n-2}\neq 0\}$. In either case $\bt\in S(\mathbf D,\omega_3)$ and if $s_{1,2n-2}=s_{2,2n-2}=0$ or $s_{2,2n-2}\neq 0$ it is easy to verify that \eqref{reichtll} holds. So suppose that $s_{2,2n-2}=0$, $s_{1,2n-2}\neq 0$ and let $\mathbf p\in D_2^{\typ 1}\cup \mathbf D^{\typ 2}$.

\begin{center}
\begin{tikzpicture}

  \node (311) at (0.25,0.5 + 4) {};
  \fill[black] (311) circle (1pt);
  \node (411) at (0.25,1+ 4) {};
  \fill[black] (411) circle (1pt);
  \node (511) at (0.25,1.5+ 4) {};
  \fill[black] (0.16,5.41) rectangle (0.34,5.59);
  \node (611) at (0.75,0.5+ 4) {};
  \fill[black] (0.66,4.41) rectangle (0.84,4.59);
  \node (711) at (0.75,1+ 4) {};
  \fill[black] (711) circle (2.5pt);
  \node (911) at (1.25,0.5+ 4) {};
  \fill[black] (911) circle (2.5pt);
 
  \node (1211) at (-0.25,0.5+ 4) {};
  \fill[black] (1211) circle (1pt);
  \node (1311) at (-0.25,1+ 4) {};
 \fill[black] (1311) circle (1pt);
  \node (1411) at (-0.25,1.5+ 4) {};
  \fill[red] (1411) circle (1pt);
  \node (2711) at (-0.75,0.5+ 4) {};
  \fill[black] (2711) circle (1pt);
  \node (2811) at (-0.75,1+ 4) {};
  \fill[black] (2811) circle (1pt);
  \node (2911) at (-0.75,1.5+ 4) {};
  \fill[red] (2911) circle (1pt);


 \draw  (1311) -- (1411);
\draw[black]  (1211) -- (1311);
 \draw (2811) -- (2911);
     \node (1611) at (-0.5,0.5+ 4) {\textcolor{black}{...}};
    \node (1611xx) at (-0.5,1.5+ 4) {...};
    \node (1611xxx) at (-0.5,1+ 4) {\textcolor{black}{...}};
	
  \node (1711) at (-1.25,0.5+ 4) {};
  \fill[black] (1711) circle (1pt);
  \node (1811) at (-1.25,1+ 4) {};
  \fill[black] (1811) circle (1pt);
  \node (1911) at (-1.25,1.5+ 4) {};
  \fill[red] (1911) circle (1pt);
	
   
   
 
  \node (3111) at (-1.75,0.5+ 4) {};
  \fill[black] (3111) circle (1pt);
  \node (3211) at (-1.75,1+ 4) {};
  \fill[black] (3211) circle (1pt);
  \node (3311) at (-1.75,1.5+ 4) {};
  \fill[red] (3311) circle (1pt);
   
   
 
 \node (17112) at (-2.25,0.5+ 4) {};
  \fill[black] (17112) circle (1pt);
  \node (18112) at (-2.25,1+ 4) {};
  \fill[black] (18112) circle (1pt);
  \node (19112) at (-2.25,1.5+ 4) {};
  \fill[red] (19112) circle (1pt);
   
   

\node (27112) at (-2.75,0.5+ 4) {};
  \fill[black] (27112) circle (1pt);
  \node (28112) at (-2.75,1+ 4) {};
  \fill[black] (28112) circle (1pt);
  \node (29112) at (-2.75,1.5+ 4) {};
  \fill[red] (29112) circle (1pt);

\node (3112) at (-3.75,0.5 + 4) {};
  \fill[black] (3112) circle (2.5pt);
  \node (4112) at (-3.75,1+ 4) {};
  \fill[black] (4112) circle (2.5pt);
  \node (5112) at (-3.75,1.5+ 4) {};
  \fill[red] (-3.84,5.41) rectangle (-3.66,5.59);
 
  \node (12112) at (-3.25,0.5+ 4) {};
  \fill[black] (12112) circle (1pt);
  \node (13112) at (-3.25,1+ 4) {};
 \fill[black] (13112) circle (1pt);
  \node (14112) at (-3.25,1.5+ 4) {};
  \fill[red] (14112) circle (1pt);

 \draw  (13112) -- (14112);
\draw[black]  (12112) -- (13112);
 \draw (28112) -- (29112) ;
\draw[black] (27112) -- (28112);
   
    \node (16112) at (-3,1+ 4) {\textcolor{black}{...}};
     \node (16112x) at (-3,1.5+ 4) {...};
    \node (16112xxx) at (-3,0.5+ 4) {\textcolor{black}{...}};
   
   \draw[black] (311)  -- (411);
   \draw (411)  -- (511);
   \draw[black] (611)  -- (711);
   \draw[black] (311)  -- (611);
   \draw[black] (411)  -- (711);
   \draw[black] (611)  -- (911);
   \draw[black] (311)  -- (1211) ;
   \draw[black] (411)  -- (1311) ;
   \draw (511)  -- (1411) ;
   \draw[black] (2711) -- (1711);
   \draw[black] (2811) -- (1811);
   \draw  (2911) -- (1911);
   \draw[black] (1711) -- (1811);
    \draw[black] (2711) -- (2811);
    \draw (1811) -- (1911);
   \draw[black] (1711) -- (3111);
   \draw[black] (1811) -- (3211);
   \draw (1911) -- (3311);
   \draw[black] (3111) -- (3211);
  \draw (3211) -- (3311);
       
   \draw (18112) -- (19112);
    \draw[black] (17112) -- (18112);
   \draw[black] (17112) -- (3111);
   \draw[black] (18112) -- (3211);
   \draw (19112) -- (3311);

         \draw[black] (3112)  -- (4112);
   \draw (4112)  -- (5112);
   \draw[black] (3112)  -- (12112) ;
   \draw[black] (4112)  -- (13112) ;
   \draw (5112)  -- (14112) ;
   \draw[black] (27112) -- (17112);
   \draw[black] (28112) -- (18112);
   \draw  (29112) -- (19112);
 

\end{tikzpicture}
\end{center}

 If $\bp\in \mathbf D_2^{\typ 1}$ the statement follows from $\alpha_{3,2n-3}\in \bp$. So let again $\mathbf p=\mathbf p_1\cup \mathbf p_2\in \mathbf D^{\typ 2}\backslash \bq$. If $\al_{1,2n-2}\in \mathbf p_1$, we are done. Otherwise set 
$$\overline{\bp}_1=\mathbf p_1\backslash \{\alpha_{1,3},\alpha_{2,2n-2}\}\cup \{\al_{3,2n-3}\},\ \overline{\bp}_2=\mathbf p_2\backslash \{\alpha_{2,3}\}\cup\{a_{1,4}\}.$$ This yields $\overline{\mathbf p}_1,\overline{\mathbf p}_2\in \mathbf D_2^{\typ 1}$ and therefore 
$$\sum_{\beta\in \mathbf p_1}(s_{\beta}-t_{\beta})+\sum_{\beta\in \mathbf p_2}(s_{\beta}-t_{\beta})\leq \sum_{\beta\in \overline{\mathbf p}_1}(s_{\beta}-t_{\beta})+\sum_{\beta\in \overline{\mathbf p}_2}(s_{\beta}-t_{\beta})\leq (m-1)+(m-1).$$\vspace{3pt}

This finishes Case 1; so from now on we can assume that $s_{3,2n-3}=0$. 

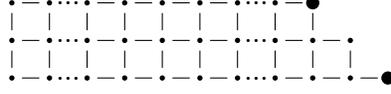
\begin{center}
\begin{tikzpicture}

  \node (311) at (0.25,0.5 + 4) {};
  \fill[black] (311) circle (1pt);
  \node (411) at (0.25,1+ 4) {};
  \fill[black] (411) circle (1pt);
  \node (511) at (0.25,1.5+ 4) {};
  \fill[black] (511) circle (2.5pt);
  \node (611) at (0.75,0.5+ 4) {};
  \fill[black] (611) circle (1pt);
  \node (711) at (0.75,1+ 4) {};
  \fill[black] (711) circle (1pt);
  \node (911) at (1.25,0.5+ 4) {};
  \fill[black] (911) circle (2.5pt);
 
  \node (1211) at (-0.25,0.5+ 4) {};
  \fill[black] (1211) circle (1pt);
  \node (1311) at (-0.25,1+ 4) {};
 \fill[black] (1311) circle (1pt);
  \node (1411) at (-0.25,1.5+ 4) {};
  \fill[black] (1411) circle (1pt);
  \node (2711) at (-0.75,0.5+ 4) {};
  \fill[black] (2711) circle (1pt);
  \node (2811) at (-0.75,1+ 4) {};
  \fill[black] (2811) circle (1pt);
  \node (2911) at (-0.75,1.5+ 4) {};
  \fill[black] (2911) circle (1pt);


 \draw  (1311) -- (1411);
\draw[black]  (1211) -- (1311);
 \draw (2811) -- (2911);
     \node (1611) at (-0.5,0.5+ 4) {\textcolor{black}{...}};
    \node (1611xx) at (-0.5,1.5+ 4) {...};
    \node (1611xxx) at (-0.5,1+ 4) {\textcolor{black}{...}};
	
  \node (1711) at (-1.25,0.5+ 4) {};
  \fill[black] (1711) circle (1pt);
  \node (1811) at (-1.25,1+ 4) {};
  \fill[black] (1811) circle (1pt);
  \node (1911) at (-1.25,1.5+ 4) {};
  \fill[black] (1911) circle (1pt);
	
   
   
 
  \node (3111) at (-1.75,0.5+ 4) {};
  \fill[black] (3111) circle (1pt);
  \node (3211) at (-1.75,1+ 4) {};
  \fill[black] (3211) circle (1pt);
  \node (3311) at (-1.75,1.5+ 4) {};
  \fill[black] (3311) circle (1pt);
   
   
 
 \node (17112) at (-2.25,0.5+ 4) {};
  \fill[black] (17112) circle (1pt);
  \node (18112) at (-2.25,1+ 4) {};
  \fill[black] (18112) circle (1pt);
  \node (19112) at (-2.25,1.5+ 4) {};
  \fill[black] (19112) circle (1pt);
   
   

\node (27112) at (-2.75,0.5+ 4) {};
  \fill[black] (27112) circle (1pt);
  \node (28112) at (-2.75,1+ 4) {};
  \fill[black] (28112) circle (1pt);
  \node (29112) at (-2.75,1.5+ 4) {};
  \fill[black] (29112) circle (1pt);

\node (3112) at (-3.75,0.5 + 4) {};
  \fill[black] (3112) circle (1pt);
  \node (4112) at (-3.75,1+ 4) {};
  \fill[black] (4112) circle (1pt);
  \node (5112) at (-3.75,1.5+ 4) {};
  \fill[black] (5112) circle (1pt);
 
  \node (12112) at (-3.25,0.5+ 4) {};
  \fill[black] (12112) circle (1pt);
  \node (13112) at (-3.25,1+ 4) {};
 \fill[black] (13112) circle (1pt);
  \node (14112) at (-3.25,1.5+ 4) {};
  \fill[black] (14112) circle (1pt);

 \draw  (13112) -- (14112);
\draw[black]  (12112) -- (13112);
 \draw (28112) -- (29112) ;
\draw[black] (27112) -- (28112);
   
    \node (16112) at (-3,1+ 4) {\textcolor{black}{...}};
     \node (16112x) at (-3,1.5+ 4) {...};
    \node (16112xxx) at (-3,0.5+ 4) {\textcolor{black}{...}};
   
   \draw[black] (311)  -- (411);
   \draw (411)  -- (511);
   \draw[black] (611)  -- (711);
   \draw[black] (311)  -- (611);
   \draw[black] (411)  -- (711);
   \draw[black] (611)  -- (911);
   \draw[black] (311)  -- (1211) ;
   \draw[black] (411)  -- (1311) ;
   \draw (511)  -- (1411) ;
   \draw[black] (2711) -- (1711);
   \draw[black] (2811) -- (1811);
   \draw  (2911) -- (1911);
   \draw[black] (1711) -- (1811);
    \draw[black] (2711) -- (2811);
    \draw (1811) -- (1911);
   \draw[black] (1711) -- (3111);
   \draw[black] (1811) -- (3211);
   \draw (1911) -- (3311);
   \draw[black] (3111) -- (3211);
  \draw (3211) -- (3311);
       
   \draw (18112) -- (19112);
    \draw[black] (17112) -- (18112);
   \draw[black] (17112) -- (3111);
   \draw[black] (18112) -- (3211);
   \draw (19112) -- (3311);

         \draw[black] (3112)  -- (4112);
   \draw (4112)  -- (5112);
   \draw[black] (3112)  -- (12112) ;
   \draw[black] (4112)  -- (13112) ;
   \draw (5112)  -- (14112) ;
   \draw[black] (27112) -- (17112);
   \draw[black] (28112) -- (18112);
   \draw  (29112) -- (19112);
 

\end{tikzpicture}
\end{center}

Hence we have simplified the situation to the following
\begin{equation}\label{reichtl2}\sum_{\beta\in \mathbf p}(s_{\beta}-t_{\beta})\leq M_{\bp}((m-p)\omega_3),\ \forall \mathbf p\in \mathbf D_1^{\typ 1}\cup \widetilde{\mathbf D}_2^{\typ 1}\cup \mathbf D^{\typ 2}\backslash \mathbf q \ \Rightarrow \bs-\bt\in S(\mathbf D,(m-p)\omega_3),\end{equation} 
where $\widetilde{\mathbf D}_2^{\typ 1}=\{\bp\in \mathbf D_2^{\typ 1}\mid \al_{2,2n-3}\in \bp\}$. 
Let $\bs^{'}$ be the multi--exponent obtained from $\bs$ by setting all entries $s_{\beta}$ with $\beta\in R_3^+(2n-4)$ to zero and $\bt^{\bs^{'}}=(t^{'}_{\beta})$ be the multi--exponent associated to $\bs^{'}$. By Lemma~\ref{Anfolg} we obtain for all $\mathbf p\in \mathbf D_1^{\typ 1}$  
\begin{equation}\label{gghhtt}\sum_{\beta\in \mathbf p}(s_{\beta}-t^{'}_{\beta})\leq m-1.\end{equation}
Recall that $s_{3,j}\neq 0$ for some $3\leq j \leq 2n-3$ and hence $t^{'}_{3,k}\neq 0$ for some $3\leq k \leq 2n-4$. So we consider the following cases which can appear.\vspace{3pt}

\textbf{Case {2}:} Suppose that $\sum_{\beta} t^{'}_{\beta}=3$. In this case there exists $3\leq j_3<j_2<j_1\leq 2n-4$ such that $t_{1,j_1}=t_{2,j_2}=t_{3,j_3}=1$ (see the red squares below). 

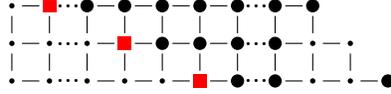
\begin{center}
\begin{tikzpicture}

  \node (311) at (0.25,0.5 + 4) {};
  \fill[black] (311) circle (1pt);
  \node (411) at (0.25,1+ 4) {};
  \fill[black] (411) circle (1pt);
  \node (511) at (0.25,1.5+ 4) {};
  \fill[black] (511) circle (2.5pt);
  \node (611) at (0.75,0.5+ 4) {};
  \fill[black] (611) circle (1pt);
  \node (711) at (0.75,1+ 4) {};
  \fill[black] (711) circle (1pt);
  \node (911) at (1.25,0.5+ 4) {};
  \fill[black] (911) circle (2.5pt);
 
  \node (1211) at (-0.25,0.5+ 4) {};
  \fill[black] (1211) circle (2.5pt);
  \node (1311) at (-0.25,1+ 4) {};
 \fill[black] (1311) circle (2.5pt);
  \node (1411) at (-0.25,1.5+ 4) {};
  \fill[black] (1411) circle (2.5pt);
  \node (2711) at (-0.75,0.5+ 4) {};
  \fill[black] (2711) circle (2.5pt);
  \node (2811) at (-0.75,1+ 4) {};
  \fill[black] (2811) circle (2.5pt);
  \node (2911) at (-0.75,1.5+ 4) {};
  \fill[black] (2911) circle (2.5pt);


 \draw  (1311) -- (1411);
\draw[black]  (1211) -- (1311);
 \draw (2811) -- (2911);
     \node (1611) at (-0.5,0.5+ 4) {\textcolor{black}{...}};
    \node (1611xx) at (-0.5,1.5+ 4) {...};
    \node (1611xxx) at (-0.5,1+ 4) {\textcolor{black}{...}};
	
  \node (1711) at (-1.25,0.5+ 4) {};
  \fill[red] (-1.34,4.41) rectangle (-1.16,4.59);
  \node (1811) at (-1.25,1+ 4) {};
  \fill[black] (1811) circle (2.5pt);
  \node (1911) at (-1.25,1.5+ 4) {};
  \fill[black] (1911) circle (2.5pt);
	
   
   
 
  \node (3111) at (-1.75,0.5+ 4) {};
  \fill[black] (3111) circle (1pt);
  \node (3211) at (-1.75,1+ 4) {};
  \fill[black] (3211) circle (2.5pt);
  \node (3311) at (-1.75,1.5+ 4) {};
  \fill[black] (3311) circle (2.5pt);
   
   
 
 \node (17112) at (-2.25,0.5+ 4) {};
  \fill[black] (17112) circle (1pt);
  \node (18112) at (-2.25,1+ 4) {};
  \fill[red] (-2.34,4.91) rectangle (-2.16,5.09);
  \node (19112) at (-2.25,1.5+ 4) {};
  \fill[black] (19112) circle (2.5pt);
   
   

\node (27112) at (-2.75,0.5+ 4) {};
  \fill[black] (27112) circle (1pt);
  \node (28112) at (-2.75,1+ 4) {};
  \fill[black] (28112) circle (1pt);
  \node (29112) at (-2.75,1.5+ 4) {};
  \fill[black] (29112) circle (2.5pt);

\node (3112) at (-3.75,0.5 + 4) {};
  \fill[black] (3112) circle (1pt);
  \node (4112) at (-3.75,1+ 4) {};
  \fill[black] (4112) circle (1pt);
  \node (5112) at (-3.75,1.5+ 4) {};
  \fill[black] (5112) circle (1pt);
 
  \node (12112) at (-3.25,0.5+ 4) {};
  \fill[black] (12112) circle (1pt);
  \node (13112) at (-3.25,1+ 4) {};
 \fill[black] (13112) circle (1pt);
  \node (14112) at (-3.25,1.5+ 4) {};
  \fill[red] (-3.34,5.41) rectangle (-3.16,5.59);

 \draw  (13112) -- (14112);
\draw[black]  (12112) -- (13112);
 \draw (28112) -- (29112) ;
\draw[black] (27112) -- (28112);
   
    \node (16112) at (-3,1+ 4) {\textcolor{black}{...}};
     \node (16112x) at (-3,1.5+ 4) {...};
    \node (16112xxx) at (-3,0.5+ 4) {\textcolor{black}{...}};
   
   \draw[black] (311)  -- (411);
   \draw (411)  -- (511);
   \draw[black] (611)  -- (711);
   \draw[black] (311)  -- (611);
   \draw[black] (411)  -- (711);
   \draw[black] (611)  -- (911);
   \draw[black] (311)  -- (1211) ;
   \draw[black] (411)  -- (1311) ;
   \draw (511)  -- (1411) ;
   \draw[black] (2711) -- (1711);
   \draw[black] (2811) -- (1811);
   \draw  (2911) -- (1911);
   \draw[black] (1711) -- (1811);
    \draw[black] (2711) -- (2811);
    \draw (1811) -- (1911);
   \draw[black] (1711) -- (3111);
   \draw[black] (1811) -- (3211);
   \draw (1911) -- (3311);
   \draw[black] (3111) -- (3211);
  \draw (3211) -- (3311);
       
   \draw (18112) -- (19112);
    \draw[black] (17112) -- (18112);
   \draw[black] (17112) -- (3111);
   \draw[black] (18112) -- (3211);
   \draw (19112) -- (3311);

         \draw[black] (3112)  -- (4112);
   \draw (4112)  -- (5112);
   \draw[black] (3112)  -- (12112) ;
   \draw[black] (4112)  -- (13112) ;
   \draw (5112)  -- (14112) ;
   \draw[black] (27112) -- (17112);
   \draw[black] (28112) -- (18112);
   \draw  (29112) -- (19112);
 

\end{tikzpicture}
\end{center}

Let $\bp\in \widetilde{\mathbf D}_2^{\typ 1}$ of the following form
$$\bp=\{\al_{1,4},\dots,\al_{1,p},\al_{2,p},\dots,\al_{2,2n-3},\al_{3,2n-3}\}.$$
We suppose that $j_1>p>j_2$, because otherwise there is nothing to show. This yields $s_{2,p}=\cdots=s_{2,2n-4}=0$ and hence
$$\sum_{\beta\in\bp}(s_{\beta}-t^{'}_{\beta})\leq (s_{1,4}-t^{'}_{1,4})+\cdots+(s_{1,2n-3}-t^{'}_{1,2n-3})+(s_{2,2n-3}-t^{'}_{2,2n-3})\leq m-1.$$
Similar arguments show $$\sum_{\beta\in \mathbf p}(s_{\beta}-t^{'}_{\beta})\leq 2(m-1),\ \mbox{ for all $\bp\in \mathbf D^{\typ 2}\backslash \mathbf q $}.$$ Hence \eqref{reichtl2} and \eqref{gghhtt} together imply 
$$\bs-\bt^{\bs^{'}}\in S(\mathbf D,(m-1)\omega_3).$$
\textbf{Case {3}:} In this case we suppose $\sum_{\beta} t^{'}_{\beta}=1$. The proof proceeds similarly to the proof of Case 2 and will be omitted. 

\textbf{Case {4}:} In this case we suppose that $\sum_{\beta} t^{'}_{\beta}=2$. \vspace{3pt}

Here we have again two cases, namely either there exists $3\leq j_3<j_1 \leq 2n-4$ such that $t^{'}_{1,j_1}=t^{'}_{3,j_3}=1$ or there exists $3\leq j_3<j_2 \leq 2n-4$ such that $t^{'}_{2,j_2}=t^{'}_{3,j_3}=1$. The latter case works similarly and will be omitted.

\textbf{Case {4.1}:} Suppose there exists $3\leq j_3<j_1 \leq 2n-4$ such that $t^{'}_{1,j_1}=t^{'}_{3,j_3}=1$. 
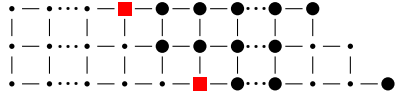
\vspace{3pt}

\begin{center}
\begin{tikzpicture}

  \node (311) at (0.25,0.5 + 4) {};
  \fill[black] (311) circle (1pt);
  \node (411) at (0.25,1+ 4) {};
  \fill[black] (411) circle (1pt);
  \node (511) at (0.25,1.5+ 4) {};
  \fill[black] (511) circle (2.5pt);
  \node (611) at (0.75,0.5+ 4) {};
  \fill[black] (611) circle (1pt);
  \node (711) at (0.75,1+ 4) {};
  \fill[black] (711) circle (1pt);

  \node (911) at (1.25,0.5+ 4) {};
  \fill[black] (911) circle (2.5pt);

  \node (1211) at (-0.25,0.5+ 4) {};
  \fill[black] (1211) circle (2.5pt);
  \node (1311) at (-0.25,1+ 4) {};
 \fill[black] (1311) circle (2.5pt);
  \node (1411) at (-0.25,1.5+ 4) {};
  \fill[black] (1411) circle (2.5pt);
 
  \node (2711) at (-0.75,0.5+ 4) {};
  \fill[black] (2711) circle (2.5pt);
  \node (2811) at (-0.75,1+ 4) {};
  \fill[black] (2811) circle (2.5pt);
  \node (2911) at (-0.75,1.5+ 4) {};
  \fill[black] (2911) circle (2.5pt);

 \draw  (1311) -- (1411);
\draw[black]  (1211) -- (1311);
 \draw (2811) -- (2911);

     \node (1611) at (-0.5,0.5+ 4) {\textcolor{black}{...}};

    \node (1611xx) at (-0.5,1.5+ 4) {...};
    \node (1611xxx) at (-0.5,1+ 4) {\textcolor{black}{...}};

  \node (1711) at (-1.25,0.5+ 4) {};
  \fill[red] (-1.34,4.41) rectangle (-1.16,4.59);
  \node (1811) at (-1.25,1+ 4) {};
  \fill[black] (1811) circle (2.5pt);
  \node (1911) at (-1.25,1.5+ 4) {};
  \fill[black] (1911) circle (2.5pt);

  \node (3111) at (-1.75,0.5+ 4) {};
  \fill[black] (3111) circle (1pt);
  \node (3211) at (-1.75,1+ 4) {};
  \fill[black] (3211) circle (2.5pt);
  \node (3311) at (-1.75,1.5+ 4) {};
  \fill[black] (3311) circle (2.5pt);

 \node (17112) at (-2.25,0.5+ 4) {};
  \fill[black] (17112) circle (1pt);
  \node (18112) at (-2.25,1+ 4) {};
  \fill[black] (18112) circle (1pt);
  \node (19112) at (-2.25,1.5+ 4) {};
  \fill[red] (-2.34,5.41) rectangle (-2.16,5.59);

\node (27112) at (-2.75,0.5+ 4) {};
  \fill[black] (27112) circle (1pt);
  \node (28112) at (-2.75,1+ 4) {};
  \fill[black] (28112) circle (1pt);
  \node (29112) at (-2.75,1.5+ 4) {};
  \fill[black] (29112) circle (1pt);

\node (3112) at (-3.75,0.5 + 4) {};
  \fill[black] (3112) circle (1pt);
  \node (4112) at (-3.75,1+ 4) {};
  \fill[black] (4112) circle (1pt);
  \node (5112) at (-3.75,1.5+ 4) {};
  \fill[black] (5112) circle (1pt);

  \node (12112) at (-3.25,0.5+ 4) {};
  \fill[black] (12112) circle (1pt);
  \node (13112) at (-3.25,1+ 4) {};
 \fill[black] (13112) circle (1pt);
  \node (14112) at (-3.25,1.5+ 4) {};
  \fill[black] (14112) circle (1pt);

 \draw  (13112) -- (14112);
\draw[black]  (12112) -- (13112);
 \draw (28112) -- (29112) ;
\draw[black] (27112) -- (28112);

    \node (16112) at (-3,1+ 4) {\textcolor{black}{...}};
     \node (16112x) at (-3,1.5+ 4) {...};
    
    \node (16112xxx) at (-3,0.5+ 4) {\textcolor{black}{...}};

   \draw[black] (311)  -- (411);
   \draw (411)  -- (511);

   \draw[black] (611)  -- (711);
   \draw[black] (311)  -- (611);
   \draw[black] (411)  -- (711);

   \draw[black] (611)  -- (911);

   \draw[black] (311)  -- (1211) ;
   \draw[black] (411)  -- (1311) ;
   \draw (511)  -- (1411) ;

   \draw[black] (2711) -- (1711);
   \draw[black] (2811) -- (1811);
   \draw  (2911) -- (1911);

   \draw[black] (1711) -- (1811);
    \draw[black] (2711) -- (2811);
    \draw (1811) -- (1911);

   \draw[black] (1711) -- (3111);
   \draw[black] (1811) -- (3211);
   \draw (1911) -- (3311);

   \draw[black] (3111) -- (3211);
  \draw (3211) -- (3311);

   \draw (18112) -- (19112);
    \draw[black] (17112) -- (18112);

   \draw[black] (17112) -- (3111);
   \draw[black] (18112) -- (3211);
   \draw (19112) -- (3311);

         \draw[black] (3112)  -- (4112);
   \draw (4112)  -- (5112);

  
   \draw[black] (3112)  -- (12112) ;
   \draw[black] (4112)  -- (13112) ;
   \draw (5112)  -- (14112) ;

   \draw[black] (27112) -- (17112);
   \draw[black] (28112) -- (18112);
   \draw  (29112) -- (19112);

\end{tikzpicture}
\end{center}
This case can be divided again into two further cases. One case treats $\sum^{2n-4}_{k=3}s_{1,k}=m$ and the other case $\sum^{2n-4}_{k=3}s_{1,k}\leq m-1$. In the latter case we can construct a multi--exponent $\bt\in S(\mathbf D, \omega_3)$ similarly as in Case 2 such that $\bs-\bt\in S(\mathbf D,(m-1) \omega_3)$. The details will be omitted.

\textbf{Case {4.1.1}:}
We suppose that $\sum^{2n-4}_{k=3}s_{1,k}= m$. If $s_{2,2n-3}=0$, we set $\bt\in\mathbf T(1)$ to be the multi--exponent with $\supp(\bt)=\{\al_{3,j_3},\al_{1,j_1}\}$. Then the statement can be easily deduced. So suppose from now on that $s_{2,2n-3}\neq 0$. 
This forces also that $s_{1,3}\neq 0$, because otherwise 
$$\sum^{2n-4}_{k=4}s_{1,k}+s_{2,2n-3}=m+s_{2,2n-3}>m.
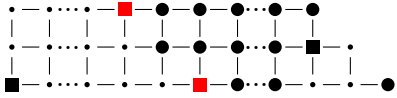
$$
\begin{center}
\begin{tikzpicture}

  \node (311) at (0.25,0.5 + 4) {};
  \fill[black] (311) circle (1pt);
  \node (411) at (0.25,1+ 4) {};
  \fill[black] (0.16,4.91) rectangle (0.34,5.09);
  \node (511) at (0.25,1.5+ 4) {};
  \fill[black] (511) circle (2.5pt);
  \node (611) at (0.75,0.5+ 4) {};
  \fill[black] (0.75,0.5+ 4) circle (1pt);
  \node (711) at (0.75,1+ 4) {};
  \fill[black] (711) circle (1pt);

  \node (911) at (1.25,0.5+ 4) {};
  \fill[black] (911) circle (2.5pt);

  \node (1211) at (-0.25,0.5+ 4) {};
  \fill[black] (1211) circle (2.5pt);
  \node (1311) at (-0.25,1+ 4) {};
 \fill[black] (1311) circle (2.5pt);
  \node (1411) at (-0.25,1.5+ 4) {};
  \fill[black] (1411) circle (2.5pt);
 
  \node (2711) at (-0.75,0.5+ 4) {};
  \fill[black] (2711) circle (2.5pt);
  \node (2811) at (-0.75,1+ 4) {};
  \fill[black] (2811) circle (2.5pt);
  \node (2911) at (-0.75,1.5+ 4) {};
  \fill[black] (2911) circle (2.5pt);

 \draw  (1311) -- (1411);
\draw[black]  (1211) -- (1311);
 \draw (2811) -- (2911);

     \node (1611) at (-0.5,0.5+ 4) {\textcolor{black}{...}};

    \node (1611xx) at (-0.5,1.5+ 4) {...};
    \node (1611xxx) at (-0.5,1+ 4) {\textcolor{black}{...}};

  \node (1711) at (-1.25,0.5+ 4) {};
  \fill[red] (-1.34,4.41) rectangle (-1.16,4.59);
  \node (1811) at (-1.25,1+ 4) {};
  \fill[black] (1811) circle (2.5pt);
  \node (1911) at (-1.25,1.5+ 4) {};
  \fill[black] (1911) circle (2.5pt);

  \node (3111) at (-1.75,0.5+ 4) {};
  \fill[black] (3111) circle (1pt);
  \node (3211) at (-1.75,1+ 4) {};
  \fill[black] (3211) circle (2.5pt);
  \node (3311) at (-1.75,1.5+ 4) {};
  \fill[black] (3311) circle (2.5pt);

 \node (17112) at (-2.25,0.5+ 4) {};
  \fill[black] (17112) circle (1pt);
  \node (18112) at (-2.25,1+ 4) {};
  \fill[black] (18112) circle (1pt);
  \node (19112) at (-2.25,1.5+ 4) {};
  \fill[red] (-2.34,5.41) rectangle (-2.16,5.59);

\node (27112) at (-2.75,0.5+ 4) {};
  \fill[black] (27112) circle (1pt);
  \node (28112) at (-2.75,1+ 4) {};
  \fill[black] (28112) circle (1pt);
  \node (29112) at (-2.75,1.5+ 4) {};
  \fill[black] (29112) circle (1pt);

\node (3112) at (-3.75,0.5 + 4) {};
  \fill[black] (-3.84,4.41) rectangle (-3.66,4.59);
  \node (4112) at (-3.75,1+ 4) {};
  \fill[black] (4112) circle (1pt);
  \node (5112) at (-3.75,1.5+ 4) {};
  \fill[black] (5112) circle (1pt);

  \node (12112) at (-3.25,0.5+ 4) {};
  \fill[black] (12112) circle (1pt);
  \node (13112) at (-3.25,1+ 4) {};
 \fill[black] (13112) circle (1pt);
  \node (14112) at (-3.25,1.5+ 4) {};
  \fill[black] (14112) circle (1pt);

 \draw  (13112) -- (14112);
\draw[black]  (12112) -- (13112);
 \draw (28112) -- (29112) ;
\draw[black] (27112) -- (28112);

    \node (16112) at (-3,1+ 4) {\textcolor{black}{...}};
     \node (16112x) at (-3,1.5+ 4) {...};
    
    \node (16112xxx) at (-3,0.5+ 4) {\textcolor{black}{...}};

   \draw[black] (311)  -- (411);
   \draw (411)  -- (511);

   \draw[black] (611)  -- (711);
   \draw[black] (311)  -- (611);
   \draw[black] (411)  -- (711);

   \draw[black] (611)  -- (911);

   \draw[black] (311)  -- (1211) ;
   \draw[black] (411)  -- (1311) ;
   \draw (511)  -- (1411) ;

   \draw[black] (2711) -- (1711);
   \draw[black] (2811) -- (1811);
   \draw  (2911) -- (1911);

   \draw[black] (1711) -- (1811);
    \draw[black] (2711) -- (2811);
    \draw (1811) -- (1911);

   \draw[black] (1711) -- (3111);
   \draw[black] (1811) -- (3211);
   \draw (1911) -- (3311);

   \draw[black] (3111) -- (3211);

  \draw (3211) -- (3311);

   \draw (18112) -- (19112);
    \draw[black] (17112) -- (18112);

   \draw[black] (17112) -- (3111);
   \draw[black] (18112) -- (3211);
   \draw (19112) -- (3311);

         \draw[black] (3112)  -- (4112);

   \draw (4112)  -- (5112);
%
  %
   \draw[black] (3112)  -- (12112) ;
   \draw[black] (4112)  -- (13112) ;
	
   \draw (5112)  -- (14112) ;

   \draw[black] (27112) -- (17112);
   \draw[black] (28112) -- (18112);
   \draw  (29112) -- (19112);

\end{tikzpicture}
\end{center}
If in addition $s_{1,2n-2}=0$, then we can define $\bt\in\mathbf T(1)$ to be the multi--exponent with $\supp(\bt)=\{\al_{2,2n-3},\al_{1,3}\}$ and the statement follows easily. So we can assume that $s_{1,2n-2}$ is also non--zero. 

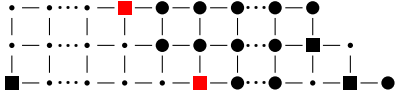
\begin{center}
\begin{tikzpicture}

  \node (311) at (0.25,0.5 + 4) {};
  \fill[black] (311) circle (1pt);
  \node (411) at (0.25,1+ 4) {};
  \fill[black] (0.16,4.91) rectangle (0.34,5.09);
  \node (511) at (0.25,1.5+ 4) {};
  \fill[black] (511) circle (2.5pt);
  \node (611) at (0.75,0.5+ 4) {};
  \fill[black] (0.66,4.41) rectangle (0.84,4.59);
  \node (711) at (0.75,1+ 4) {};
  \fill[black] (711) circle (1pt);

  \node (911) at (1.25,0.5+ 4) {};
  \fill[black] (911) circle (2.5pt);

  \node (1211) at (-0.25,0.5+ 4) {};
  \fill[black] (1211) circle (2.5pt);
  \node (1311) at (-0.25,1+ 4) {};
 \fill[black] (1311) circle (2.5pt);
  \node (1411) at (-0.25,1.5+ 4) {};
  \fill[black] (1411) circle (2.5pt);
 
  \node (2711) at (-0.75,0.5+ 4) {};
  \fill[black] (2711) circle (2.5pt);
  \node (2811) at (-0.75,1+ 4) {};
  \fill[black] (2811) circle (2.5pt);
  \node (2911) at (-0.75,1.5+ 4) {};
  \fill[black] (2911) circle (2.5pt);

 \draw  (1311) -- (1411);
\draw[black]  (1211) -- (1311);
 \draw (2811) -- (2911);

     \node (1611) at (-0.5,0.5+ 4) {\textcolor{black}{...}};

    \node (1611xx) at (-0.5,1.5+ 4) {...};
    \node (1611xxx) at (-0.5,1+ 4) {\textcolor{black}{...}};

  \node (1711) at (-1.25,0.5+ 4) {};
  \fill[red] (-1.34,4.41) rectangle (-1.16,4.59);
  \node (1811) at (-1.25,1+ 4) {};
  \fill[black] (1811) circle (2.5pt);
  \node (1911) at (-1.25,1.5+ 4) {};
  \fill[black] (1911) circle (2.5pt);

  \node (3111) at (-1.75,0.5+ 4) {};
  \fill[black] (3111) circle (1pt);
  \node (3211) at (-1.75,1+ 4) {};
  \fill[black] (3211) circle (2.5pt);
  \node (3311) at (-1.75,1.5+ 4) {};
  \fill[black] (3311) circle (2.5pt);

 \node (17112) at (-2.25,0.5+ 4) {};
  \fill[black] (17112) circle (1pt);
  \node (18112) at (-2.25,1+ 4) {};
  \fill[black] (18112) circle (1pt);
  \node (19112) at (-2.25,1.5+ 4) {};
  \fill[red] (-2.34,5.41) rectangle (-2.16,5.59);

\node (27112) at (-2.75,0.5+ 4) {};
  \fill[black] (27112) circle (1pt);
  \node (28112) at (-2.75,1+ 4) {};
  \fill[black] (28112) circle (1pt);
  \node (29112) at (-2.75,1.5+ 4) {};
  \fill[black] (29112) circle (1pt);

\node (3112) at (-3.75,0.5 + 4) {};
  \fill[black] (-3.84,4.41) rectangle (-3.66,4.59);
  \node (4112) at (-3.75,1+ 4) {};
  \fill[black] (4112) circle (1pt);
  \node (5112) at (-3.75,1.5+ 4) {};
  \fill[black] (5112) circle (1pt);

  \node (12112) at (-3.25,0.5+ 4) {};
  \fill[black] (12112) circle (1pt);
  \node (13112) at (-3.25,1+ 4) {};
 \fill[black] (13112) circle (1pt);
  \node (14112) at (-3.25,1.5+ 4) {};
  \fill[black] (14112) circle (1pt);

 \draw  (13112) -- (14112);
\draw[black]  (12112) -- (13112);
 \draw (28112) -- (29112) ;
\draw[black] (27112) -- (28112);

    \node (16112) at (-3,1+ 4) {\textcolor{black}{...}};
     \node (16112x) at (-3,1.5+ 4) {...};
    
    \node (16112xxx) at (-3,0.5+ 4) {\textcolor{black}{...}};

   \draw[black] (311)  -- (411);
   \draw (411)  -- (511);

   \draw[black] (611)  -- (711);
   \draw[black] (311)  -- (611);
   \draw[black] (411)  -- (711);

   \draw[black] (611)  -- (911);

   \draw[black] (311)  -- (1211) ;
   \draw[black] (411)  -- (1311) ;
   \draw (511)  -- (1411) ;

   \draw[black] (2711) -- (1711);
   \draw[black] (2811) -- (1811);
   \draw  (2911) -- (1911);

   \draw[black] (1711) -- (1811);
    \draw[black] (2711) -- (2811);
    \draw (1811) -- (1911);

   \draw[black] (1711) -- (3111);
   \draw[black] (1811) -- (3211);
   \draw (1911) -- (3311);

   \draw[black] (3111) -- (3211);
  \draw (3211) -- (3311);

   \draw (18112) -- (19112);
    \draw[black] (17112) -- (18112);

   \draw[black] (17112) -- (3111);
   \draw[black] (18112) -- (3211);
   \draw (19112) -- (3311);

         \draw[black] (3112)  -- (4112);
   \draw (4112)  -- (5112);

  
   \draw[black] (3112)  -- (12112) ;
   \draw[black] (4112)  -- (13112) ;
   \draw (5112)  -- (14112) ;

   \draw[black] (27112) -- (17112);
   \draw[black] (28112) -- (18112);
   \draw  (29112) -- (19112);

\end{tikzpicture}
\end{center}

This is the only case where there is no multi--exponent $\bt\in S(\mathbf D,\omega_3)$ such that $\bs-\bt\in S(\mathbf D,(m-1)\omega_3)$. We shall define a multi--exponent $\bt \in S(\mathbf D,2\omega_3)$ such that $\bs-\bt\in S(\mathbf D,(m-2)\omega_3)$. Let $\bt$ be the multi--exponent with $\supp(\bt)=\{\al_{3,j_3},\al_{1,j_1},\al_{1,3},\al_{1,2n-2},\al_{2,2n-3}\}$. Obviously we have $\bt \in S(\mathbf D,2\omega_3)$. If $\mathbf p\in \mathbf D_1^{\typ 1}$, then we can also deduce immediately
$$\sum_{\beta\in \mathbf p}(s_{\beta}-t_{\beta})\leq m-2.$$
So let $\mathbf p\in \mathbf D_2^{\typ 1}$. There is only something to prove if $\mathbf p$ is of the following form 
$$\mathbf p=\{\al_{1,4},\dots,\al_{1,p},\al_{2,p},\dots,\al_{2,2n-3},\al_{3,2n-3}\},\ \mbox{ for some $p\leq j_3$.}$$
Since 
$$s_{1,3}+\cdots+s_{1,p}+s_{2,p}+\cdots+s_{2,j_3}\leq m-s_{3,j_3}\leq m-1<\sum^{2n-4}_{k=3}s_{1,k}$$
we obtain by subtracting $s_{1,3}$ on both sides
$$s_{1,4}+\cdots+s_{1,p}+s_{2,p}+\cdots+s_{2,j_3}<s_{1,4}+\cdots+s_{1,2n-4}.$$
Therefore 
$$\sum_{\beta\in \mathbf p}(s_{\beta}-t_{\beta})\leq \sum^{2n-3}_{k=4}(s_{1,k}-t_{1,k})+(s_{2,2n-3}-t_{2,2n-3})=\sum^{2n-3}_{k=4}s_{1,k}+s_{2,2n-3}-2\leq m-2.$$
Let $\mathbf p=\mathbf p_1\cup \mathbf p_2\in \mathbf D^{\typ 2}$ be a type 2 Dyck path. There is only something to show if $\mathbf p_1$ is of the form 
$$\mathbf p_1=\{\al_{1,3},\dots,\al_{1,p},\al_{2,p},\dots,\al_{2,2n-2}\},\ \mbox{ for some where $p\leq j_3$.}$$
 We get similar as above
$$\sum_{\beta \in \mathbf p}(s_{\beta}-t_{\beta})\leq \sum^{2n-3}_{k=3}(s_{1,k}-t_{1,k})+(s_{2,2n-3}-t_{2,2n-3})+(s_{2,2n-2}-t_{2,2n-2})+\sum_{\beta \in \mathbf p_2}(s_{\beta}-t_{\beta})\leq 2(m-2).$$ 
$\hfill\Box$
\subsection{}\label{code}

\vspace{0,5cm}
We used the program Eclipse and the following code:

$public$ $class$ $B3\{$\\
$static$ $int$ $dim = 0;$\\
$public$ $static$ $void$ $main$($String$[] $args$)$\{$\\

$int$ $m1,m2,m3=0;$\\
                    $for(m1 = 0; m1 <= 9; m1\!+\!+)\{$\\
                        $for(m2 = 0; m2 <= 9; m2\!+\!+)\{$\\
                            $for(m3 = 0; m3 <= 9; m3\!+\!+)\{$\\
                                $if(m1+m2+m3 <= 9)\{$\\
                                $int$ $s1,s2,s3,s4,s5,s6,s7,s8,s9 = 0;$\\
    $for(s9 = 0 ; s9 <= m1 ; s9\!+\!+)\{$\\
                    $for(s8 = 0; s8 <= m3 ; s8\!+\!+)\{$\\
                    $for(s7 = 0; s7 <= m2 ; s7\!+\!+)\{$\\
                    $for(s6 = 0; s6 <= m1 + m2 ; s6\!+\!+)\{$\\
                    $for(s5 = 0; s5 <= 2$*$m2 + m3 ; s5\!+\!+)\{$\\
                    $for(s4 = 0; s4 <= 2$*$m1 + 2$*$m2 + m3 ; s4\!+\!+)\{$\\
                    $for(s3 = 0; s3 <= m2 + m3 ; s3\!+\!+)\{$\\
                    $for(s2 = 0; s2 <= m1 + m2 + m3 ; s2\!+\!+)\{$\\
                    $for(s1 = 0; s1 <= m1 + m2 + m2 + m3  ; s1\!+\!+)\{$\\
                   
                    $if(s2 + s3 + s4  + s8 + s9      <= m1 + m2 + m3)\{$\\
                        $if(s3 + s4 + s5     + s8 + s9      <= m1 + m2 + m3)\{$\\
                        $if(s4 + s5 + s6   + s8 + s9     <= m1 + m2 + m3)\{$\\   
                        $if(s5 + s6 + s7 + s8 + s9      <= m1 + m2 + m3)\{$\\
                       
                        $if(s3 +     s5         + s8           <=      m2 + m3)\{$\\
                        $if(s5 +     s7 + s8           <=      m2 + m3)\{$\\
                        $if(s6 + s7 +     s9      <=      m1 + m2)\{$\\
                           
                        $if(s1 + s2 + s3 + s4 + s5 +   s7 +     s9      <= m1 + 2$*$m2 + m3)\{$\\   
                    $    if(s1 +s3 + s4 + s5 + s6 + s7 +     s9      <= m1 + 2$*$m2 + m3)\{$\\
                        $if(s2 + s3 + s4 + s5 +   s7 + s8 + s9      <= m1 + 2$*$m2 + m3)\{$\\
                        $if(s3 + s4 + s5 + s6 + s7 + s8 + s9      <= m1 + 2$*$m2 + m3)\{$\\
                           
                    $    if(s1 + s2 + s3 + s4 + s5 + s6 + s7 +    2$*$s9      <= 2$*$m1 + 2$*$m2 + m3)\{$\\   
                    $    if(s2 + s3 + s4 + s5 + s6 + s7 + s8 + 2$*$s9      <= 2$*$m1 + 2$*$m2 + m3)\{$\\
                           
                        $if(s1 + s2 + 2$*$s3 + 2$*$s4 + 2$*$s5 + s6 + s7 + s8   + 2$*$s9   <= 2$*$m1 + 3$*$m2 + 2$*$m3)\{$\\
                    $    if(s2 + 2$*$s3 + 2$*$s4 + 2$*$s5 + s6 + s7 + 2$*$s8 + 2$*$s9   <= 2$*$m1 + 3$*$m2 + 2$*$m3)\{$\\
                    $    if(s3 +   s4 + 2$*$s5 + s6 + s7 + 2$*$s8 +   s9   <=  m1 + 2$*$m2 + 2$*$m3)\{$\\
                            $    dim\!+\!+;\}\}\}\}\}\}\}\}\}\}\}\}\}\}\}\}\}\}\}\}\}\}\}\}\}$\\
                        $System.out.println("\mid S("+m1+"w1+"+m2+"w2+"+m3+"w3)\mid="+dim);$\\
                        $dim=0;$\}\}\}\}\}\}

\vspace{0,5cm}

\bibliographystyle{plain}
\bibliography{kus-pbw-biblist}
\end{document}